\documentclass[reqno]{amsart}
\usepackage{amssymb}
\usepackage{color}
\usepackage{mathrsfs}
\usepackage{amsmath, amsfonts, vmargin, enumerate}
\usepackage{graphics}
\usepackage{verbatim}
\usepackage{amsthm}
\usepackage{latexsym,bm}
\usepackage{euscript,amscd}
\usepackage{dashrule}
\usepackage{arydshln}
\usepackage{eufrak}
\usepackage{dsfont}

\input xypic
\xyoption {all}

\def\>{\rangle}
\def\<{\langle}

\newtheorem{thm}{Theorem}[section]
\newtheorem{prop}[thm]{Proposition}
\newtheorem{cor}[thm]{Corollary}
\newtheorem{lem}[thm]{Lemma}

\newtheorem{step}[thm]{Observation}

\theoremstyle{definition}
\newtheorem{defi}[thm]{Definition}
\newtheorem{example}[thm]{Example}

\theoremstyle{remark}
\newtheorem{remark}[thm]{Remark}
\newtheorem{pb}[thm]{Problem}

\numberwithin{equation}{section}

\newcommand{\un}{1\mkern -4mu{\rm l}}

\newcommand{\bN}{{\mathbb N}}
\newcommand{\bP}{{\mathbb P}}

\newcommand{\bR}{{\mathbb R}}
\newcommand{\bS}{{\mathbb S}}

\newcommand{\bV}{{\mathbb V}}

\newcommand{\cK}{{\mathcal K}}
\renewcommand{\cH}{{\mathcal H}}
\newcommand{\bM}{{\mathbb M}}
\newcommand{\bC}{{\mathbb C}}

\newcommand{\cC}{{\mathcal C}}
\newcommand{\cE}{{\mathcal E}}
\newcommand{\cF}{{\mathcal F}}

\newcommand{\cZ}{{\mathcal Z}}
\renewcommand{\cL}{{\mathcal L}}

\newcommand{\fC}{{\mathfrak C}}
\newcommand{\fP}{{\mathfrak P}}

\newcommand{\Tr}{\mathrm{Tr}}

\newcommand{\la}{\langle}
\newcommand{\ra}{\rangle}

\newcommand{\ba}{\begin{array}}
\newcommand{\ea}{\end{array}}
\newcommand{\be}{\begin{eqnarray*}}
\newcommand{\ee}{\end{eqnarray*}}
\newcommand{\beg}{\begin{eqnarray}}
\newcommand{\eeg}{\end{eqnarray}}
\newcommand{\beq}{\begin{equation}}
\newcommand{\eeq}{\end{equation}}
\newcommand{\beqn}{\begin{equation*}}
\newcommand{\eeqn}{\end{equation*}}
\newcommand{\ov}{\overline}
\renewcommand{\Re}{\mathrm{Re}}
\newcommand{\rT}{\mathrm{t}}
\newcommand{\rd}{\mathrm{d}}
\newcommand{\Aj}{\eta_1}
\newcommand{\Ad}{\ov{\eta_2}}
\newcommand{\Ps}[4]{\Psi_{#1#2}(\eta_{#3},\eta_{#4})}
\newcommand*\oline[1]{%
  \vbox{%
    \hrule height 0.5pt
    \kern0.25ex
    \hbox{%
      \kern-0.1em
      \ifmmode#1\else\ensuremath{#1}\fi
      \kern-0.1em
    }
  }
}
\newcommand{\id}{\mathrm{id}}
\newcommand{\jed}{\mathds{1}}
\newcommand{\veps}{\varepsilon}
\newcommand{\eps}{\epsilon}
\renewcommand{\i}{\mathrm{i}}
\newcommand{\e}[1]{e^{\i\theta_{#1}}}
\newcommand{\me}[1]{e^{-\i\theta_{#1}}}
\newcommand{\ti}[1]{\tilde{#1}}
\newcommand{\ran}{\mathrm{ran}\,}
\newcommand{\vep}{\varepsilon}

\title[Merging of positive maps]{Merging of positive maps: a construction of various classes of  positive maps on matrix algebras}
\author{Marcin Marciniak}
\address{Institute of Theoretical Physics and Astrophysics, University of Gda{\'n}sk, Wita Stwosza 57, 80 -954 Gda{\'n}sk, Poland}
\email{matmm@ug.edu.pl}
\author{Adam Rutkowski}
\address{Institute of Theoretical Physics and Astrophysics, University of Gda{\'n}sk, Wita Stwosza 57, 80 -954 Gda{\'n}sk, Poland}
\address{National Quantum Information Center of Gda\'nsk, 81-824 Sopot, Poland}
\email{fizar@ug.edu.pl}
\begin{document}
\maketitle
\begin{abstract}
For two positive maps $\phi_i:B(\cK_i)\to B(\cH_i)$, $i=1,2$, we construct a new linear map $\phi:B(\cH)\to B(\cK)$, where $\cK=\cK_1\oplus\cK_2\oplus\bC$, $\cH=\cH_1\oplus\cH_2\oplus\bC$, by means of some additional ingredients such as operators and functionals. We call it a \textit{merging} of maps $\phi_1$ and $\phi_2$. The properties of this construction are discussed. In particular, conditions for positivity of $\phi$, as well as for $2$-positivity, complete positivity, optimality and nondecomposability, are provided. In particular, we show that for a pair composed of $2$-positive and $2$-copositive maps, there is a nondecomposable merging of them. One of our main results asserts, that for a canonical merging of a pair composed of completely positive and completely copositive extremal maps, their canonical merging is an exposed positive map. This result provides a wide class of new examples of exposed positive maps. As an application, new examples of entangled PPT states are described.
\end{abstract}

\section{Introduction}
In recent years positive maps on operator algebras began to play a significant role in various branches of mathematical physics. For instance, in quantum information theory they became an important tool for detecting entanglement while in the theory of dynamical system they serve as a natural generalization of dynamical maps. After pioneer work of Erling St{\o}rmer \cite{Sto_pos} several papers appeared with several examples. However, in spite of great efforts of many mathematicians the classification of positive linear maps on C$^*$-algebras is still an open problem. Although there are many partial results scattered across the literature, it seems that we are far from full knowledge on all features of these objects. Even in the finite dimensional case the situation is unclear. For example, no algebraic formula for general positive map between matrix algebras is known.

Since we are dealing with convex structures, among all positive maps extremal ones are the key to solving the problem of classification. The explicit form of extremal positive maps is described fully only for the simplest cases: maps from $\bM_2(\bC)$ into itself and maps from $\bM_2(\bC)$ into $\bM_3(\bC)$. This is a consequence of the results of St{\o}rmer and Woronowicz \cite{Sto_pos,Wor_low} that all positive maps are decomposable in these cases. In general case, it is known that maps $B(\cK)\to B(\cH)$ of the form \beq\label{ads}X\mapsto AXA^*,\qquad\qquad X\mapsto AX^\rT A^*,\eeq where $\cdot^\rT$ stands for the transposition and $A\in B(\cK,\cH)$, are extremal in the cone of all positive maps between $B(\cK)$ and $B(\cH)$ (\cite{YH}). It was Choi, who gave the first example of an extremal positive map which is not of the form \eqref{ads} , hence it is not decomposable \cite{Choi75b}.

Due to Straszewicz theorem \cite{Stra}, extremal positive maps are approximated by elements of the thinner class of exposed positive maps. The Choi map is not exposed, but some variants of it \cite{CKL} turn out to be exposed \cite{HK1}. It was proved in \cite{MM} that maps \eqref{ads} are exposed. Further examples of exposed maps are given in \cite{HK,HK2,ChS,ChS2,Wor_non,MO}. Geometric approach to exposed maps was presented in \cite{MajT}.

The aim of this paper is to provide some scheme for constructing interesting examples of positive maps. It turns out that having two positive maps one can 'merge' them into a new map using some additional ingredients such as operators and functionals. It turns out that under some conditions the merging procedure produces a positive map. Further, we provide some necessary conditions and sufficient conditions for such properties of the merging as 2-positivity, complete positivity or (non)decomposability. For example, we show that for a pair composed of a 2-positive map and a 2-copositive one, there is a merging which is a nondecomposable positive map. One of our main results (Theorem \ref{mainm}) asserts that merging of maps \eqref{ads} is an exposed positive map. Thus, we provide a wide class of new examples of exposed positive maps. It seems also that the presented construction could be a good starting point for the attempt to describe a general form of an exposed positive map.

The paper is organized as follows. In Section 2 we provide some necessary preliminary definitions and results. Section 3 is devoted to the construction of merging of two positive maps and description of its general properties. In Section 4 we discuss properties of some special examples of merging. In particular, we prove that merging of completely positive and completely copositive extremal maps is an exposed positive maps. These maps generalize the example described by Miller and Olkiewicz in \cite{MO}. We also show that another generalization of this example, which was given in \cite{RSC} is also a result of merging of some two positive maps and is optimal but not extremal. Finally, in Section 5 we provide a discussion of the special case, when the merging procedure gives maps from $\bM_3(\bC)$ into itself. We characterize all positive maps of that form, and we formulate some conditions for various properties: 2-positivity, complete positivity, decomposability and nondecomposability.

\section{Notations, definitions and preliminary results}
If $\cK$ is a Hilbert space then by $B(\cK)$ (respectively $B(\cK)_+$) we denote the space of all bounded (respectively the cone of all bounded positive) operators on $\cK$. For Hilbert spaces $\cK$ and $\cH$ we denote by $B(B(\cK),B(\cH))$ the space of all bounded linear maps from $B(\cK)$ into $B(\cH)$. An element $\phi\in B(B(\cK),B(\cH))$ is a \textit{positive map} if $\phi\left(B(\cK)_+\right)\subset B(\cH)_+$.
Let $\bP(\cK,\cH)$ (or shortly $\bP$) denote the cone of all positive maps from $B(\cK)$ into $B(\cH)$. If $\phi,\psi\in \bP$ then $\psi\leq\phi$ means $\phi-\psi\in\bP$. A positive map $\phi$ is \textit{extremal} if $\phi$ generates an extremal ray in the cone $\bP$, i.e. $\psi\leq\phi$ implies $\psi=\lambda\phi$ for some $\lambda\in [0,1]$ whenever $\psi\in\bP$.

Assume now that $\cK$ and $\cH$ are finite dimensional spaces. Let $\cK\ni\xi\mapsto\ov{\xi}\in\cK$ and $\cH\ni x\mapsto\ov{x}\in\cH$ be an antilinear involutions. For $Y\in B(\cH)$ define its transpose $Y^\rT\in B(\cH)$ by $Y^\rT x=\ov{{Y^*} \ov{x}}$, $x\in\cH$. We consider a bilinear pairing $\la\cdot,\cdot\ra_\rd$ between $B(B(\cK),B(\cH))$ and $B(\cK)\otimes B(\cH)$ given by
\beq
\la \phi,X\otimes Y\ra_\rd=\Tr\left(\phi(X)Y^\rT\right),
\eeq
where $\phi\in B(B(\cK),B(\cH)$, $X\in B(\cK)$ and $Y\in B(\cH)$ (\cite{EK}, see also \cite{Sto_ext}). For a cone $\bV\subset B(B(\cK),B(\cH))$ we define its dual cone $\bV^\circ\subset B(\cK)\otimes B(\cH)$ by
\beq
\bV^\circ=\{Z\in B(\cK)\otimes B(\cH):\,\mbox{$\la\phi,Z\ra_\rd\geq 0$ for all $\phi\in\fC$}\}
\eeq
It is well known (\cite{Per,HHH}, see also \cite{MajM}) that $\bP^\circ=\bS$ 
where
\beq
\bS=\left\{\sum_{i=1}^n X_i\otimes Y_i:\,n\in\bN,\,X_i\in B(\cK)_+,\,Y_i\in B(\cH)_+,\,i=1,\ldots,n\right\}
\eeq
is the cone of the so called 'unnormalized separable states' or separable positive operators. We say that an element $\phi\in\bP$ is an \textit{exposed positive map} if there is $Z_0\in \bS$ such that
\beq
\bR_+\phi=\{\psi\in\bP:\,\la \psi,Z_0\ra_\rd=0\}.
\eeq

For $F\subset \bP$ (respectively $G\subset\bS$) we define $F'\subset\bS$ (respectively $G'\subset\bP$) by
$$
F'=\{Z\in \bS:\,\mbox{$\la \phi,Z\ra_\rd=0$ for all $\phi\in F$}\}
\qquad 
(G'=\{\phi\in\bP:\,\mbox{$\la\phi,Z\ra_\rd=0$ for all $Z\in G$}\})
$$ 
It is clear that $F'$ is a closed face of $\bS$ and $G'$ is a closed face of $\bP$. It can be shown (\cite{EK}) that $\phi\in\bP$ is an exposed positive map if and only if $\{\phi\}''=\bR_+\phi$.

If $\xi\in\cK$ and $x\in\cH$, then by $x\xi^*$ we denote an operator from $\cK$ into $\cH$ defined by
\beq
(x\xi^*)\zeta=\la\xi,\zeta\ra x,\qquad \zeta\in \cK.
\eeq
Notice that extremal elements of the cone $\bS$ are of the form $\eta\eta^*\otimes yy^*$, where $\eta\in\cK$ and $x\in\cH$. Thus, exposed positive maps can be characterized by the following condition (\cite{MM}): $\phi$ is exposed if and only if 
\beq
\label{expo}
\forall\,\psi\in\bP:\;\Big(\,\forall\,(\eta,y)\in\cK\times\cH:\,\la y,\phi(\eta\eta^*)y\ra=0\;\;\Rightarrow\;\;\la y,\psi(\eta\eta^*)y\ra=0\,\Big)\;\;\Rightarrow\;\;\psi\in\bR_+\phi.
\eeq

Given $k\in\bN$, a map $\phi:B(\cK)\to B(\cH)$ is called $k$-positive if  $\id_{\bM_k(\bC)}\otimes\phi:\bM_k(\bC)\otimes B(\cK)\to \bM_k(\bC)\otimes B(\cH)$ is a positive map. Similarly, $\phi$ is said to be a $k$-copositive map if $\mathrm{tran}_{\bM_k(\bC)}\otimes\phi:\bM_k(\bC)\otimes B(\cK)\to \bM_k(\bC)\otimes B(\cH)$ is positive, where $\mathrm{tran}_{\bM_k(\bC)}$ denotes the transposition map on the matrix algebra. We say that $\phi$ is  completely positive (respectively completely copositive) if it is $k$-positive (respectively $k$-copositive) for every $k\in\bN$.
Let $\mathbb{CP}$ (respectively $\mathbb{CP}^\rT$) denote the cone of all completely positive (respectively completely copositive) maps. It was proved by Choi \cite{Choi75} that the dual cone $\mathbb{CP}^\circ$ is nothing but the cone $\left(B(\cK)\otimes B(\cH)\right)_+$ of all positive operators on $\cK\otimes\cH$.

A useful tool in analysis of positive maps is the so called Choi matrix. Recall that the Choi-Jamio{\l}kowski isomorphism \cite{Choi75,Jam} is a map $B(B(\cK),B(\cH))\ni\phi\mapsto \cC_\phi\in B(\cK)\otimes B(\cH)$ given by
\beq
\label{CJ}
\cC_\phi=\sum_{i,j=1}^k \eps_i^{}\eps_j^*\otimes\phi(\eps_i^{}\eps_j^*)
\eeq
where $k=\dim\cK$ and $\eps_1,\ldots,\eps_k$ is some fixed orthonormal basis. $\cC_\phi$ is called a Choi matrix of the map $\phi$. The famous result of Choi (\cite{Choi75}) says that a map $\phi$ is completely positive if and only if the Choi matrix $\cC_\phi$ is positive definite.

Given a map $\phi\in B(B(\cK),B(\cH))$ one may consider its dual functional $\widetilde{\phi}$ acting on $B(\cK)\otimes B(\cH)$ given by $\widetilde{\phi}(Z)=\la\phi,Z\ra_\rd$ for $Z\in B(\cK)\otimes B(\cH)$ (\cite{Sto_ext}).
According to \cite[Lemma 4.2.3]{Sto_ks} (see also \cite{Sto_ext}) $\cC_\phi^\rT$ is a 'density matrix' of the functional $\widetilde{\phi}$, i.e.
\beq
\la \phi , Z\ra_\rd=\Tr(\cC_\phi^\rT Z), \qquad Z\in B(\cK)\otimes B(\cH).
\eeq
Therefore, $\phi$ is positive if and only if $\Tr(\cC_\phi^\rT Z)\geq 0$ for every separable operators $Z$, while $\phi$ is completely positive if and only if $\Tr(\cC_\phi^\rT Z)\geq 0$ for all positive definite operators. We say that a positive definite operator $Z$ on $\cK\otimes \cH$ is entangled if $Z\in (B(\cK)\otimes B(\cH))_+\setminus\mathbb{S}$. It follows that a positive definite operator is entangled if and only if there is a positive but not completely positive map such that $\Tr(\cC_\phi^\rT Z)<0$. We say that such a map $\phi$ detects entanglement of $Z$ or $\phi$ is an entanglement witness for $Z$ (see \cite{ChSew} for a review on entanglement witnesses and references therein).

If $\phi=\phi_1+\phi_2$ where $\phi_1$ is a completely positive map and $\phi_2$ is completely copositive one, then $\phi$ is called decomposable. Let $\mathbb{D}$ denote the cone of decomposable maps. By the result of \cite{Sto_dec} (see also \cite{EK,MajM}) $\mathbb{D}^\circ=\mathbb{T}$, where
\beq
\mathbb{T}=\{Z\in (B(\cK)\otimes B(\cH))_+:\,Z^\Gamma\in (B(\cK)\otimes B(\cH))_+\}
\eeq
By $Z^\Gamma$ we denote partial transposition of $Z$, i.e. $Z^\Gamma=\id_{B(\cK)}\otimes\mathrm{tran}_{B(\cH)}(Z)$. Elements of $\mathbb{T}$ are called PPT operators in analogy to PPT states \cite{Per,HHH}. Each separable positive operator is a PPT operator. The converse statement is true only if $\dim\cK=1$ or $\dim\cH=1$ or $(\dim\cK,\dim\cH)\in \{(2,2),(2,3),(3,2)\}$. This is a consequence of results of St{\o}rmer and Woronowicz \cite{Sto_pos,Wor_low} that $\mathbb{P}=\mathbb{D}$ if and only if one of the above dimension conditions holds.

It follows from the above remarks that a positive map $\phi$ is nondecomposable if and only if it is an entanglement witness for some entangled PPT operator, i.e. there is $Z\in \mathbb{T}$ such that $\Tr(\cC_\phi^\rT Z)<0$.
We will use this criterion several times.

Finally, let us recall that a map $\phi:B(\cK)\to B(\cH)$ is called optimal if there is no nonzero completely positive map $\psi:B(\cK)\to B(\cH)$ such that $\psi\leq \phi$ (\cite{Lew2000}).
Equivalently, $\phi$ is optimal if the face in $\bP$ generated by $\phi$ contains no nonzero completely positive maps \cite{HKopt,Kyeopt}. We say that a positive map $\phi:B(\cK)\to B(\cH)$ satisfies spanning property if 
\beq
\label{span}
\mathrm{span}\{ \xi\otimes x:\,\xi\in\cK,\,x\in\cH,\, \phi(\xi\xi^*)x=0\}=\cK\otimes\cH.
\eeq
It was shown in \cite{Lew2000} that spanning property is a sufficient condition for optimality. In \cite{Kyeopt} it was pointed out that $\phi$ satisfies spanning property if and only if the exposed face $\{\phi\}''$ generated by $\phi$ contains no nonzero completely positive maps. Therefore, one can reformulate this condition to the following one, which is simmilar to \eqref{expo}
\beq
\label{spanpro}
\forall\,\psi\in\bC\bP:\;\Big(\,\forall\,(\eta,y)\in\cK\times\cH:\,\la y,\phi(\eta\eta^*)y\ra=0\;\;\Rightarrow\;\;\la y,\psi(\eta\eta^*)y\ra=0\,\Big)\;\;\Rightarrow\;\;\psi=0.
\eeq

\section{Merging of positive maps}
\subsection{Block-matrices}
Let $\cH_1,\ldots,\cH_n$ be Hilbert spaces and $\cH=\bigoplus_{i=1}^n\cH_i$. By $W_i$ we denote the canonical isometrical embedding of $\cH_i$ into $\cH$. Assume that an antilinear involution $\cH\ni\xi\mapsto\ov{\xi}\in\cH$ is given, and $\ov{\cH_i}=\cH_i$ for every $i=1,\ldots,n$. For any $X\in B(\cH)$ we consider its block decomposition $X=(X_{ij})_{i,j=1,\ldots,n}$ where $X_{ij}=W_i^*XW_j\in B(\cH_j,\cH_i)$. 
In the sequel we usually will identify $B(\cH_j,\cH_i)$ with the subspace $W_i B(\cH_j,\cH_i) W_j^*\subset B(\cH)$. Given $i,j=1,\ldots,n$, we consider two operations on blocks: Hermitian conjugation $B(\cH_j,\cH_i)\ni X_{ij}\mapsto X_{ij}^*\in B(\cH_i,\cH_j)$ and transposition $B(\cH_j,\cH_i)\ni X_{ij}\mapsto X_{ij}^\rT\in B(\cH_i,\cH_j)$, where $X_{ij}^\rT$ is defined by the condition $\la \xi_j,X_{ij}^\rT\xi_i\ra=\la\ov{\xi_i},X_{ij}\ov{\xi_j}\ra$ for $\xi_i\in\cL_i$, $\xi_j\in\cH_j$. Note that Hermitian conjugation is an antilinear map while transposition is a linear one.

Thorough the paper we will frequently use the following criterion for positivity of block-matrices which is rather obvious.
\begin{prop}
\label{blockpos}
$X\in B(\cH)$ is positive if and only if for any $y_1,\ldots,y_n$, $y_i\in\cH_i$, $i=1,\ldots,n$, the scalar matrix $(\la y_i,X_{ij}y_j\ra)_{i,j=1,\ldots,n}\in \bM_n(\bC)$ is positive definite.
\end{prop}

\subsection{Definition and basic properties}
Let $\cK_1,\cK_2,\cH_1,\cH_2$ be Hilbert spaces.
Suppose that two positive maps $\phi_1:B(\cK_1)\to B(\cH_1)$ and $\phi_2:B(\cK_2)\to B(\cH_2)$ are given.

Let $\cK=\cK_1\oplus\cK_2\oplus\cK_3$ and $\cH=\cH_1\oplus\cH_2\oplus\cH_3$, where $\cK_3$ and $\cH_3$ are one dimensional spaces. We fix normalized vectors $\varepsilon\in\cK_3$ and $e\in\cH_3$, so that $\cK_3=\bC \varepsilon$ and $\cH_3=\bC e$. We consider antilinear involutions on $\cK$ and $\cH$ which are uniquely determined by involutions on $\cK_1, \cK_2$ and $\cH_1,\cH_2$ respectively and the condition that $\varepsilon$ and $e$ are real vectors, i.e. $\ov{\varepsilon}=\varepsilon$ and $\ov{e}=e$.
Our aim is to construct a new map $\phi: B(\cK)\to B(\cH)$ by means of the two given maps $\phi_1$ and $\phi_2$.

For the construction we need the following additional ingredients: linear operators $B_i:\cK_i\to \cH_i$ and $C_i:\cK_i\to\cH_i$, and linear functionals $\omega_i:B(\cK_i)\to\bC$, $i=1,2$.
Finally, for $i=1,2$, let $P_i$ be an orthogonal projection in $\cH_i$ onto the range of $\phi_i(\un_{\cK_i})$.

For $X\in B(\cK)$, let $(X_{ij})_{i,j=1,2,3}$ be the block decomposition of $X$ established by the decomposition $\cK=\bigoplus_{i=1}^3\cK_i$. Since $B(\cK_3,\cK_i)=B(\bC,\cK_i)\simeq\cH_i$, $i=1,2,3$, we can consider column-blocks $X_{i3}$ as vectors from $\cK_i$. Analogously, $B(\cK_j,\cK_3)=B(\cK_j,\bC)\simeq \cK_j^*$, so row-blocks $X_{3j}$ are nothing but functionals on $\cK_j$. The Hermitian conjugation and transposition transform column-vectors into row-functionals and vice-versa. Observe also that the block $X_{33}$ is nothing but scalar.
\begin{defi}
By a \textit{merging of the maps $\phi_1$ and $\phi_2$ by means of operators $B_i$, $C_i$ and functionals $\omega_i$}, we mean a map $\phi:B(\cK)\to B(\cH)$ given by
\beq
\label{merging}
\phi (X)
=
\left(\ba{ccc}
\phi_1(X_{11})+\omega_2(X_{22})P_1 & 0 & B_1X_{13} + C_1X_{31}^\rT \\
0 & \phi_2(X_{22})+\omega_1(X_{11})P_2 & B_2X_{23} + C_2X_{32}^\rT \\
X_{31}B_1^* + X_{13}^\rT C_1^* & X_{32}B_2^* + X_{23}^\rT C_2^* & X_{33}
\ea\right).
\eeq
Here $X=(X_{ij})_{i,j=1,2,3}\in B(\cK)$, where $X_{ij}\in B(\cK_j,\cK_i)$ are blocks of $X$. Similarly, block structure of $\phi(X)$ reflects the decomposition $\cH=\bigoplus_{i=1}^3\cH_i$, i.e. each block of $\phi(X)$ is an element of a respective subspace $B(\cH_j,\cH_i)\subset B(\cH)$.  
\end{defi}
It is obvious that the map of the above form is a linear one. The question is whether any interesting positive maps can be obtained by this construction.
\begin{example}[Trivial example]
Assume that $B_i$, $C_i$ are zeros and $\omega_i$ are positive functionals. Then $$\phi(X)=\left(\ba{ccc}\phi_1(X_{11}) +\omega_2(X_{22})P_1& 0 & 0 \\
0 & \phi_2(X_{22})+\omega_1(\eta_1\eta_1^*)P_2 & 0 \\
0 & 0 & X_{33}
\ea\right).
$$
Hence $\phi$ has the structure of the direct sum of positive maps. In particular, if $\omega_i$ are also zeros, then $\phi=\phi_1\oplus\phi_2\oplus\id_\bC:B(\cK_1)\oplus B(\cK_2)\oplus \bC\to B(\cH_1)\oplus B(\cH_2)\oplus \bC$.
\end{example}

We will say that a merging is \textit{nontrivial} if at least one of the operators $B_1$, $B_2$, $C_1$, $C_2$ is nonzero.

\begin{example}[Example of Miller and Olkiewicz]
\label{e:MO}
Let $\cK_i=\cH_i=\bC$, $i=1,2$. Define maps $\phi_i:\bC\to\bC$ by $\phi_i(x)=\frac{1}{2}x$ for $i=1,2$. Clearly, they are positive maps. Consider merging $\phi$ of these maps by means of the following ingredients $B_i,C_i,\omega_i:\bC\to\bC$, $i=1,2$:
$$ B_1x=C_2x=\frac{1}{\sqrt{2}}x,\quad C_1x=B_2x=0, \quad
\omega_1(x)=\omega_2(x)=\frac{1}{2}x,\qquad x\in\bC.$$
Then $\phi:B(\bC^3)\to B(\bC^3)$ has the form
\beq
\label{MO}
\phi\left(\ba{ccc}
x_{11} & x_{12} & x_{13} \\ x_{21} & x_{22} & x_{23} \\ x_{31} & x_{32} & x_{33}
\ea\right)
=\left(\ba{ccc}
\frac{1}{2}(x_{11}+x_{22}) & 0 & \frac{1}{\sqrt{2}}x_{13} \\[5pt]
0 & \frac{1}{2}(x_{11}+x_{22}) & \frac{1}{\sqrt{2}}x_{32} \\[5pt]
\frac{1}{\sqrt{2}}x_{31} & \frac{1}{\sqrt{2}}x_{23} & x_{33}
\ea\right).
\eeq
It was shown by Miller and Olkiewicz in \cite{MO} that the map \eqref{MO} is a bistochastic  exposed nondecomposable positive map.

This is a basic example. The idea of merging appeared as an attempt to generalize this example. For further generalizations we will consider a 'denormalized' version of \eqref{MO}
\beq
\label{MOunnorm}
\tilde{\phi}(X)=V\phi(X)V^*=\left(\ba{ccc}
x_{11}+x_{22} & 0 & x_{13} \\
0 & x_{11}+x_{22} & x_{32} \\
x_{31} & x_{23} & x_{33}
\ea\right)
\eeq
where $V=\mathrm{diag}(\sqrt{2},\sqrt{2},1)$. One can easily observe that $\tilde{\phi}$ is a merging of two identity maps on $\bC$ by means of
$$ \tilde{B}_1x=\tilde{C}_2x=x,\quad \tilde{C}_1x=\tilde{B}_2x=0, \quad
\tilde{\omega}_1(x)=\tilde{\omega}_2(x)=x,\qquad x\in\bC.$$
\end{example}
\begin{example}
\label{e:MOgen}
Let us consider the following higher dimensional generalization of the last example. Let $\cK_1$ and $\cK_2$ be arbitrary finite dimensional Hilbert spaces and $\cH_i=\cK_i$ for $i=1,2$. Consider maps $\phi_i:B(\cK_i)\to B(\cK_i)$, $i=1,2$, where
\beq
\phi_1(X)=X,\qquad X\in B(\cK_1),
\eeq
\beq
\phi_2(X)=X^\rT,\qquad X\in B(\cK_2).
\eeq
Let operators $B_i,C_i:\cK_i\to\cK_i$ and functionals $\omega_i$ on $B(\cK_i)$ be given by
$$B_1\eta_1=\eta_1,\; C_1\eta_1=0,\quad \eta_1\in\cK_1,\qquad\qquad
B_2\eta_2=0,\; C_2\eta_2=\eta_2,\quad \eta_2\in\cK_2,$$
$$\omega_i(X)=\Tr(X),\quad X\in B(\cK_i),\;i=1,2.$$
Then the merging of maps $\phi_1$ and $\phi_2$ by means of $B_i$, $C_i$, $\omega_i$ is a map $\phi_{\cK_1,\cK_2}:B(\cK_1\oplus\cK_2\oplus\bC)\to B(\cK_1\oplus\cK_2\oplus\bC)$ of the form
\beq
\label{MOhd}
\phi_{\cK_1,\cK_2}(X)=\left(\ba{ccc}
X_{11}+\Tr(X_{22})\jed_{B(\cK_1)} & 0 & X_{13} \\
0 & X_{22}^\rT+\Tr(X_{11})\jed_{B(\cK_2)} & X_{32}^\rT \\
X_{31} & X_{23}^\rT & X_{33}
\ea\right)
\eeq
It will be shown in Theorem \ref{mainm} that similarly to Miller-Olkiewicz example the above map is an exposed positive map.
\end{example}
\begin{example}
In \cite{RSC} the following generalization $\Lambda_d:\bM_d(\bC)\to \bM_d(\bC)$ of the map \eqref{MO} was considered
\beq
\label{RSCex}
\Lambda_d(X)=\dfrac{1}{d-1}\left(\ba{ccccc}
\sum\limits_{i=1}^{d-1}x_{ii} & \cdots & 0 & 0 & \sqrt{d-1}x_{1d} \\
\vdots & & \vdots & \vdots & \vdots \\
0 & \cdots & \sum\limits_{i=1}^{d-1}x_{ii} & 0 & \sqrt{d-1}x_{d-2,d} \\
0 & \cdots & 0 & \sum\limits_{i=1}^{d-1}x_{ii} & \sqrt{d-1}x_{d,d-1} \\
\sqrt{d-1}x_{d,1} & \cdots & \sqrt{d-1}x_{d,d-2} & \sqrt{d-1}x_{d-1,d} & (d-1)x_{d,d}
\ea\right)
\eeq
where $X=(x_{ij})\in\bM_d(\bC)$.
It was shown that $\Lambda_d$ is nondecomposable and optimal positive map.
Let $\phi_1:\bM_{d-2}(\bC)\to\bM_{d-2}(\bC)$ be given by $\phi_1(X)=(d-1)^{-1}\Tr(X)$ for $X\in\bM_{d-2}(\bC)$, and $\phi_2:\bC\to\bC$ given by $\phi_2(x)=(d-1)^{-1}x$ for $x\in\bC$. One can observe that the map $\Lambda_d$ is a merging of $\phi_1$ and $\phi_2$ by means of $B_i$, $C_i$, $\omega_i$, where $B_1,C_1:\bC^{d-2}\to\bC^{d-2}$, $B_2,C_2:\bC\to\bC$, $\omega_1:\bM_{d-2}(\bC)\to\bC$, $\omega_2:\bC\to\bC$ are given by $B_1\eta=(d-1)^{-1/2}\eta$, $C_1=0$, $B_2=0$, $C_2x=(d-1)^{-1/2}x$, $\omega_1(X)=(d-1)^{-1}\Tr(X)$, $\omega_2(x)=(d-1)^{-1}x$. As in Example \ref{e:MO} we will consider 'denormalized' $\tilde{\Lambda}_{d}=V\Lambda_dV^*$ where $V=\mathrm{diag}(\sqrt{d-1},\ldots,\sqrt{d-1},1)$.
\end{example}
We can consider the following two further generalizations of \eqref{MOunnorm} in the direction established by the above example.
\begin{example}
\label{e:RSCgen}
Let $\cK_1,\cK_2$ be arbitrary Hilbert spaces and $\cH_1=\cK_1$, $\cH_2=\cK_2$. Consider $\phi_i(X)=\Tr(X)\jed_{B(\cK_i)}$ for $X\in B(\cK_i)$, $i=1,2$, and their merging $\Lambda_{\cK_1,\cK_2}$ by means of $B_1=\jed_{B(\cK_1)}$, $C_1=0$, $B_2=0$, $C_2=\jed_{B(\cK_2)}$ and $\omega_i(X)=\Tr(X)$ for $X\in B(\cK_i)$, $i=1,2$. Then $\Lambda_{\cK_1,\cK_2}$ has the following block-decomposition form
\beq
\label{RSCgen}
\Lambda_{\cK_1,\cK_2}(X)=
\left(\ba{ccc}
\big(\Tr (X_{11})+\Tr (X_{22})\big)\jed_{B(\cK_1)} & 0 & X_{13} \\
0 &\big( \Tr (X_{11})+\Tr (X_{22})\big)\jed_{B(\cK_2)} & X_{32}^\rT \\
X_{31} & X_{23}^\rT & X_{33}
\ea\right),
\eeq
where $X=(X_{ij})$, $X_{ij}\in B(\cK_j,\cK_i)$.
\end{example}
\begin{example}
\label{e:RSCgen2}
As previously, let Hilbert spaces $\cK_1$, $\cK_2$ be given. Let $\phi_1(X)=\Tr(X)\jed_{B(\cK_1)}$ for $X\in B(\cK_1)$, and $\phi_2(X)=X^t$ for $X\in B(\cK_2)$. By $\Omega_{\cK_1,\cK_2}$ we denote the merging of $\phi_1$ and $\phi_2$ by means of the same ingredients $B_i$, $C_i$, $\omega_i$ as Example \ref{e:RSCgen}.
Then
\beq
\label{RSCgen2}
\Omega_{\cK_1,\cK_2}(X)=
\left(\ba{ccc}
\big(\Tr (X_{11})+\Tr (X_{22})\big)\jed_{B(\cK_1)} & 0 & X_{13} \\
0 &X_{22}^t+\Tr (X_{11})\jed_{B(\cK_2)} & X_{32}^\rT \\
X_{31} & X_{23}^\rT & X_{33}
\ea\right)
\eeq
for $X=(X_{ij})\in B(\cK)$.
\end{example}

\subsection{Positivity of merging}
Now, we  discuss properties of the merging operation.
Our first goal is to describe some necessary and sufficient conditions for positivity of merging. Let us start with the following necessary condition.
\begin{prop}
\label{mergprop}
\label{posnecess}
Assume that a map $\phi$ is a merging of positive maps $\phi_1$ and $\phi_2$ by means of $B_i$, $C_i$, $\omega_i$, $i=1,2$. If $\phi$ is positive, then for each $i=1,2$,
\begin{enumerate}
\item
$\omega_i$ is a positive functional,
\item
for every $\eta_i\in\cK_i$ and $y_i\in\cH$\footnote{Here and later we will use a  convention that elements of $\cK$ are denoted by Greek letters $\xi,\eta,\ldots$, while elements of $\cH$ by Latin characters $x,y,\ldots$. Moreover, elements of subspaces $\cK_i\subset\cK$ and $\cH_i\subset\cH$, $i=1,2$, are always indexed by $i$.},
\beq
\label{epsineq}
(|\la y_i,B_i\eta_i\ra|+|\la y_i,C_i\ov{\eta_i}\ra|)^2\leq \la y_i^{},\phi_i^{}(\eta_i^{}\eta_i^*)y_i^{}\ra.
\eeq
\end{enumerate}
\end{prop}
\begin{proof}
(1) Assume that $\eta_1\in\cK_1$. Then the rank one operator $\eta_1\eta_1^*\in B(\cK)$ has the following block decomposition
$$
\eta_1\eta_1^*=\left(\ba{ccc}
\eta_1\eta_1^* & 0 & 0 \\ 0 & 0 & 0 \\ 0 & 0 & 0
\ea\right).
$$
Thus, according to \eqref{merging}
$$
\phi(\eta_i\eta_i^*)=\left(\ba{ccc}
\phi_1(\eta_1\eta_1^*) & 0 & 0 \\ 0 & \omega_1(\eta_1\eta_1^*)P_2 & 0 \\ 0 & 0 & 0
\ea\right)
$$
Since the above matrix is positive, its diagonal entries are positive. In particular, $\omega_1$ is nonnegative on all rank one positive operators, hence it is a positive functional. Positivity of $\omega_2$ is proved by the same arguments.

(2) Now consider an element $\eta\in\cK$ of the form $\eta=\eta_1+\alpha\epsilon$ for $\eta_1\in\cK_1$ and $\alpha\in \bC$. The operator $\eta\eta^*$ has block decomposition of the form
$$
\eta\eta^*=
\left(\ba{ccc}
\eta_1\eta_1^* & 0 & \ov{\alpha}\eta_1 \\
0 & 0 & 0 \\
\alpha\eta_1^* & 0 & |\alpha|^2
\ea\right).
$$
Thus,
$$
\phi(\eta\eta^*)=\left(\ba{ccc}
\phi_1(\eta_1\eta_1^*) & 0 & \ov{\alpha}B_1\eta_1 + \alpha C_1\ov{\eta_1} \\
0 & \omega_1(\eta_1\eta_1^*)P_2 & 0 \\
\alpha(B_1\eta_1)^* + \ov{\alpha}(C_1\ov{\eta_1})^* & 0 & |\alpha|^2
\ea\right).
$$
Since it is a positive block-matrix, we conclude (c.f. Proposition \ref{blockpos}) that for any $y_1\in\cH_1$, the scalar matrix
$$
\left(\ba{ccc}
\la y_1,\phi_1(\eta_1\eta_1^*)y_1\ra & 0 & \ov{\alpha}\la y_1,B_1\eta_1\ra +\alpha\la y_1, C_1\ov{\eta_1}\ra \\
0 & \omega_1(\eta_1\eta_1^*)\Vert P_2y_2\Vert^2 & 0 \\
\alpha\la B_1\eta_1,y_1\ra+\ov{\alpha}\la C_1\ov{\eta_1}, y_1\ra & 0 & |\alpha|^2
\ea\right)
$$
is positive definite. Therefore, we get the inequality
\beq
\label{uyu}
|\ov{\alpha}\la y_1,B_1\eta_1\ra +\alpha\la y_1, C_1\ov{\eta_1}\ra|^2\leq |\alpha|^2\la y_1,\phi_1(\eta_1\eta_1^*)y_1\ra.
\eeq
It is satisfied for any $\alpha\in\bC$. Let $\theta\in\bR$ be such that
$\la y_1,B_1\eta_1\ra\ov{\la y_1,C_1\ov{\eta_1}\ra}=|\la y_1,B_1\eta_1\ra| |\la y_1,C_1\ov{\eta_1}\ra|e^{i\theta}$ and $\alpha=e^{-i\theta/2}$. Then the inequality \eqref{uyu} takes the form \eqref{epsineq} for $i=1$.
\end{proof}

For $i=1,2$, let $\psi_i:B(\cK_i)\to B(\cH_i)$ and $\chi_i:B(\cK_i)\to B(\cH_i)$ be positive maps defined by
\beq
\psi_i(X)=B_iXB_i^*,\qquad\chi_i(X)=C_iX^\rT C_i^*,\qquad\qquad X\in B(\cK_i).
\eeq
\begin{cor}
If the merging of positive maps $\phi_1$, $\phi_2$ by means of $B_i,C_i,\omega_i$ is positive, then for $i=1,2$
\begin{enumerate}
\item
$\psi_i+\chi_i\leq\phi_i$,
\item
$\mathrm{ran}\,B_i\subset \mathrm{ran}\,\phi_i(\jed_{B(\cK_i)})$ and $\mathrm{ran}\,C_i\subset \mathrm{ran}\,\phi_i(\jed_{B(\cK_i)})$.
\end{enumerate}
\end{cor}
\begin{proof}
(1) Let $\eta_1\in \cK_1$ and $y_1\in\cH_1$.
Observe that
$|\la y_1,B_1\eta_1\ra|^2=\la y_1,\psi_1(\eta_1\eta_1^*)y_1\ra$ and $|\la y_1,C_1\ov{\eta_1}\ra|^2=\la y_1,\chi_1(\eta_1\eta_1^*)y_1\ra$.
Hence we conclude from \eqref{epsineq}
$$\la y_1,\psi_1(\eta_1\eta_1^*)y_1\ra + \la y_1,\chi_1(\eta_1\eta_1^*)y_1\ra\leq \la y_1,\phi_1(\eta_1\eta_1^*)y_1\ra.$$
As it holds for any choice of $\eta_1$ and $y_1$, the inequality $\psi_1+\chi_1\leq\phi_1$ follows. Case $i=2$ is shown by similar arguments.

(2) It immediately follows from \eqref{epsineq}
\end{proof}

\begin{remark}
Having the above result in mind we conclude that in order to produce some nontrivial positive map by the merging procedure one should take maps $\phi_1$ and $\phi_2$ with some 'regularity' properties.
For example, no notrivial merging of two extremal nondecomposable maps produces a positive map.
However, we will see that for properly chosen 'regular' maps there is a possibility for nontrivial merging. Surprisingly, merging of 'regular' maps can produce highly 'nonregular' positive maps.
\end{remark}

Now, we are ready to give characterization of positive merging in terms of merging ingredients. To this end, for each $i=1,2$ and every $\eta_i\in\cK_i$, $y_i\in\cH_i$, we define the following parameters
\beq
\mu_i(\eta_i,y_i)=\sqrt{\la y_i,\phi_i(\eta_i^{}\eta_i^*)y_i\ra}
\eeq
\beq
\label{epsi}
\varepsilon_i(\eta_i,y_i)=|\la y_i,B_i\eta_i\ra|+|\la y_i,C_i\ov{\eta_i}\ra|
\eeq
\beq
\label{deltai}
\delta_i(\eta_i,y_i)=\sqrt{\mu_i(\eta_i,y_i)^2-\varepsilon_i(\eta_i,y_i)^2}
\eeq
\beq
\sigma_1(\eta_1,y_2)=\sqrt{\omega_1^{}(\eta_1^{}\eta_1^*)}
\,\Vert P_2y_2\Vert,\quad
\sigma_2(\eta_2,y_1)=\sqrt{\omega_2^{}(\eta_2^{}\eta_2^*)}
\,\Vert P_1y_1\Vert
\eeq
Observe that each of the above functions $\cK_i\times\cH_i\to\bR_+$ has the homogeneity property, i.e. $f(\alpha\eta_i,\beta y_i)=|\alpha||\beta|f(\eta_i,y_i)$ for every $\alpha,\beta\in\bC$, where $f$ stands for any of these functions.
\begin{thm}
\label{t:pos}
The merging $\phi$ of positive maps $\phi_1$, $\phi_2$ by means of $B_i$, $C_i$, $\omega_i$ is a positive map if and only if the following conditions are satisfied
\begin{enumerate}[(i)]
\item $\omega_i$ are positive functionals for $i=1,2$,
\item $\varepsilon(\eta_i,y_i)\leq\mu_i(\eta_i,y_i)$ for $i=1,2$, $\eta_i\in\cK_i$, $y_i\in\cH_i$,
\item for every $\eta_1\in\cK_1$, $\eta_2\in \cK_2$, $y_1\in\cH_1$, $y_2\in\cH_2$,
\beq
\label{e:posineq}
\delta_1(\eta_1,y_1)\delta_2(\eta_2,y_2)+\sigma_1(\eta_1,y_2)\sigma_2(\eta_2,y_1)\geq \varepsilon_1(\eta_1,y_1)\varepsilon_2(\eta_2,y_2)
\eeq
\end{enumerate}
\end{thm}
\begin{proof}
Necessity of conditions (i) and (ii) was proved already (cf. Proposition \ref{posnecess}).
Positivity of $\phi$ is equivalent to positive definiteness of the scalar matrix
\beq
\label{uyuy}
\left(\ba{ccc}
\veps_1^2+\delta_1^2+\sigma_2^2 & 0 & \ov{\alpha}\beta \la y_1,B_1\eta_1\ra + \alpha\beta \la y_1,C_1\ov{\eta_1}\ra \\
0 & \veps_2^2+\delta_2^2+\sigma_1^2 & \ov{\alpha}\beta \la y_2,B_2\eta_2\ra + \alpha\beta \la y_2,C_2\ov{\eta_2}\ra \\
\alpha\ov{\beta}\la B_1\eta_1,y_1\ra+ \ov{\alpha}\ov{\beta}\la C_1\ov{\eta_1},y_1\ra & \alpha\ov{\beta}\la B_2\eta_2,y_2\ra+ \ov{\alpha}\ov{\beta}\la C_2\ov{\eta_2},y_2\ra & |\alpha|^2|\beta|^2
\ea\right)
\eeq
for every $\eta_i\in\cK_i$, $y_i\in\cH_i$, $i=1,2$, and $\alpha,\beta\in\bC$. (Here, for simplicity, $\varepsilon_i$ stands for $\varepsilon_i(\eta_i,y_i)$. Similar convention for $\delta_i$ and $\sigma_i$.)
Without loss of generality we may assume $\alpha=1$ and $\beta=1$.
The determinant of the matrix \eqref{uyuy} is equal to
\be
\lefteqn{\big(\veps_1(\eta_1,y_1)^2+\delta_1(\eta_1,y_1)^2+\sigma_2(\eta_2,y_1)^2\big) \big(\veps_2(\eta_2,y_2)^2+\delta_2(\eta_2,y_2)^2+\sigma_1(\eta_1,y_2)^2\big)}\\
&&
-\left(\veps_1(\eta_1,y_1)^2+\delta_1(\eta_1,y_1)^2+\sigma_2(\eta_2,y_1)^2\right)|
\la y_2,B_2\eta_2\ra + 
\la y_2,C_2\ov{\eta_2}\ra|^2 \\
&&
-\left(\veps_2(\eta_2,y_2)^2+\delta_2(\eta_2,y_2)^2+\sigma_1(\eta_1,y_2)^2\right)|
\la y_1,B_1\eta_1\ra + 
\la y_1,C_1\ov{\eta_1}\ra|^2 
\ee
Let $\vartheta_i\in\bR$ satisfy equality
$e^{2i\vartheta_i}\la y_i,B_i\eta_i\ra\ov{\la y_i,C_i\ov{\eta_i}\ra}= |\la y_i,B_i\eta_i\ra||\la y_i,C_i\ov{\eta_i}\ra|.$
If we replace $\eta_i$ by $e^{i\vartheta_i}\eta_i$, then the above expression will be equal to
\beg
\lefteqn{\big(\veps_1(\eta_1,y_1)^2+\delta_1(\eta_1,y_1)^2+\sigma_2(\eta_2,y_1)^2\big) \big(\veps_2(\eta_2,y_2)^2+\delta_2(\eta_2,y_2)^2+\sigma_1(\eta_1,y_2)^2\big)} \nonumber\\
&&
-\left(\veps_1(\eta_1,y_1)^2+\delta_1(\eta_1,y_1)^2+\sigma_2(\eta_2,y_1)^2\right) \varepsilon_2(\eta_2,y_2)^2 \nonumber\\
&&
-\left(\veps_2(\eta_2,y_2)^2+\delta_2(\eta_2,y_2)^2+\sigma_1(\eta_1,y_2)^2\right) \varepsilon_1(\eta_1,y_1)^2 =\nonumber\\
&=&
\delta_1(\eta_1,y_1)^2\sigma_1(\eta_1,y_2)^2+ \delta_2(\eta_2,y_2)^2\sigma_2(\eta_2,y_1)^2 \label{popo}\\ &&
+\sigma_1(\eta_1,y_2)^2\sigma_2(\eta_2,y_1)^2+ \delta_1(\eta_1,y_1)^2\delta_2(\eta_2,y_2)^2- \varepsilon_1(\eta_1,y_1)^2\varepsilon_2(\eta_2,y_2)^2. \nonumber
\eeg
Consequently, the map $\phi$ is positive if and only if the expression in the line \eqref{popo} is nonnegative for every $\eta_i\in\cK_i$ and $y_i\in\cH_i$, $i=1,2$. In particular, it should be nonnegative if we fix $\eta_i$'s and $y_i$'s but replace $\eta_1$ by $s\eta_1$ for arbitrary $s\in\bR$. Therefore, we are lead to the condition that
$$
\delta_1^2\sigma_1^2s^4+(\sigma_1^2\sigma_2^2+\delta_1^2\delta_2^2- \varepsilon_1^2\varepsilon_2^2)s^2 +\delta_2^2\sigma_2^2\geq 0
$$
for every $s\in\bR$, where we omit arguments in the formula. Consequently,
$$ \sigma_1^2\sigma_2^2+\delta_1^2\delta_2^2-\varepsilon_1^2\varepsilon_2^2 \geq -2\delta_1\delta_2\sigma_1\sigma_2.$$
The last inequality is equivalent to
$$\left(\sigma_1\sigma_2+\delta_1\delta_2\right)^2\geq \varepsilon_1^2\varepsilon_2^2$$
which yields (iii).
\end{proof}

As an application of the above theorem we immediately get
\begin{cor}
\label{c:examp}
The maps $\phi_{\cK_1,\cK_2}$, $\Lambda_{\cK_1,\cK_2}$ and $\Omega_{\cK_1,\cK_2}$ described in Examples \ref{e:MOgen}, \ref{e:RSCgen} and \ref{e:RSCgen2} are positive.
\end{cor}
\begin{proof}
For the map $\phi_{\cK_1,\cK_2}$ given by \eqref{MOhd} one checks that
\begin{align*}
&\mu_1(\eta_1,y_1)=|\la y_1,\eta_1\ra|,& &\mu_2(\eta_2,y_2)=|\la y_2,\ov{\eta_2}\ra |, &\\
&\varepsilon_1(\eta_1,y_1)=|\la y_1,\eta_1\ra|, && \varepsilon_2(\eta_2,y_2)=|\la y_2,\ov{\eta_2}\ra |,& \\
&\delta_1(\eta_1,y_1)=0, &&\delta_2(\eta_2,y_2)=0, &\\
& \sigma_1(\eta_1,y_2)=\Vert\eta_1\Vert\Vert y_2\Vert, &&
\sigma_2(\eta_2,y_1)=\Vert\eta_2\Vert\Vert y_1\Vert.
\end{align*}
Obviously, all conditions listed in Theorem \ref{t:pos} are satisfied.
As regards the map $\Lambda_{\cK_1,\cK_2}$ given by \eqref{RSCgen} one has
\begin{align*}
&\mu_1(\eta_1,y_1)=\Vert\eta_1\Vert\Vert y_1\Vert,& &\mu_2(\eta_2,y_2)=\Vert\eta_2\Vert\Vert y_2\Vert, &\\
&\varepsilon_1(\eta_1,y_1)=|\la y_1,\eta_1\ra|, && \varepsilon_2(\eta_2,y_2)=|\la y_2,\ov{\eta_2}\ra |,& \\
&\delta_1(\eta_1,y_1)=\sqrt{\Vert\eta_1\Vert^2\Vert y_1\Vert^2-|\la y_1,\eta_1\ra|^2}, &&\delta_2(\eta_2,y_2)=\sqrt{\Vert\eta_2\Vert^2\Vert y_2\Vert^2-|\la y_2,\ov{\eta_2}\ra |^2}, &\\
& \sigma_1(\eta_1,y_2)=\Vert\eta_1\Vert\Vert y_2\Vert, &&
\sigma_2(\eta_2,y_1)=\Vert\eta_2\Vert\Vert y_1\Vert.
\end{align*}
Again, all conditions of Theorem \ref{t:pos} are fulfilled. Finally, for the map $\Omega_{\cK_1,\cK_2}$ given by \eqref{e:RSCgen2} one has
\begin{align*}
&\mu_1(\eta_1,y_1)=\Vert\eta_1\Vert\Vert y_1\Vert,& &\mu_2(\eta_2,y_2)=|\la y_2,\ov{\eta_2}\ra|, &\\
&\varepsilon_1(\eta_1,y_1)=|\la y_1,\eta_1\ra|, && \varepsilon_2(\eta_2,y_2)=|\la y_2,\ov{\eta_2}\ra |,& \\
&\delta_1(\eta_1,y_1)=\sqrt{\Vert\eta_1\Vert^2\Vert y_1\Vert^2-|\la y_1,\eta_1\ra|^2}, &&\delta_2(\eta_2,y_2)=0, &\\
& \sigma_1(\eta_1,y_2)=\Vert\eta_1\Vert\Vert y_2\Vert, &&
\sigma_2(\eta_2,y_1)=\Vert\eta_2\Vert\Vert y_1\Vert.
\end{align*}
\end{proof}

Next corollary shows that there is possible some perturbation of merging ingredients which saves positivity.
\begin{cor}
\label{c:nu}
Let $\nu\in\bR$, $\nu>0$. If a merging of $\phi_1$ and $\phi_2$ by means of $B_1,B_2,C_1,C_2,\omega_1,\omega_2$ is a positive map, then also a merging of $\phi_1$ and $\phi_2$ by means of $B_1,B_2,C_1,C_2,\nu\omega_1,\nu^{-1}\omega_2$ is positive.
\end{cor}
\begin{proof}
For the perturbed merging all products in the inequality \eqref{e:posineq} remain unchanged.
\end{proof}

We finish this subsection with a necessary condition for 2-positivity of merging.
\begin{thm}
\label{2pos}
Let $\phi$ be a merging of positive maps $\phi_1$, $\phi_2$ by means of $B_i,C_i,\omega_i$. Assume that $\phi$ is a positive map, so that conditions (i), (ii), (iii) of Theorem \ref{t:pos} are satisfied. 
If $\phi$ is $2$-positive (respectively $2$-copositive), then
$C_i=0$ (respectively $B_i=0$) for $i=1,2$, and 
\beq
\label{zui}
\delta_1(\eta_1,y_1)\delta_2(\eta_2,y_2)\geq \veps_1(\eta_1,y_1)\veps_2(\eta_2,y_2)
\eeq
for every $\eta_i\in\cK_i$, $y_i\in\cH_i$, where $i=1,2$.
\end{thm}
\begin{proof}
Let us fix $i=1,2$ and let $\eta_i\in\cK_i$.
By $f_1,f_2$ denote an orthonormal basis in $\bC_2$ and consider an element $H\in \bM_2(\bC)\otimes B(\cK)$ given by
$$H=f_1^{}f_1^*\otimes \eps\eps^*+f_1^{}f_2^*\otimes \eps\eta_i^*+f_2^{}f_1^*\otimes \eta_i^{}\eps^* +f_2^{}f_2^*\otimes \eta_i^{}\eta_i^*$$
Clearly, it is positive because $H=(f_1^{}\otimes \eps+f_2^{}\otimes \eta_i^{})(f_1^{}\otimes \eps+f_2^{}\otimes \eta_i^{})^*$. 
Observe that
\be
\lefteqn{(\id_{\bM_2(\bC)}\otimes\phi)(H)=}\\
&=&
f_1^{}f_1^*\otimes ee^*+ f_1^{}f_2^*\otimes \big((C_i\ov{\eta_i})e^*+e(B_i\eta_i)^*\big)+ f_2^{}f_1^*\otimes \big((B_i\eta_i)e^*+e(C_i\ov{\eta_i})^*\big) \\
&&{}+ f_2^{}f_2^*\otimes \big(\phi_i(\eta_i^{}\eta_i^*)+\omega_i(\eta_i^{}\eta_i^*)P_j\big).
\ee
where $j=3-i$.
Further, let $z=f_1^{}\otimes C_i\ov{\eta_i}-f_2^{}\otimes e\in \bC^2\otimes \cH$. Then, one can check that
$$
\la z,(\id_{\bM_2(\bC)}\otimes\phi)(H)z\ra= -2\Vert C_i\ov{\eta_i}\Vert^2.
$$
It follows from 2-positivity of $\phi$ that  $C_i\ov{\eta_i}=0$. Since $\eta_i$ is arbitrary, $C_i=0$.

Now, let $\eta_i\in\cK_i$ and $y_i\in\cH_i$, $i=1,2$, be arbitrary. Consider an element $G\in \bM_2(\bC)\otimes B(\cK)=\bM_2(B(\cK))$ given in block-matrix form by
$$
G=\left(\ba{ccc|ccc}
\eta_1^{}\eta_1^* & 0 & 0 & 0 & \eta_1^{}\eta_2^* & \eta_1 \\
0 & 0 & 0 & 0 & 0 & 0 \\
0 & 0 & 0 & 0 & 0 & 0 \\ \hline
0 & 0 & 0 & 0 & 0 & 0 \\
\eta_2^{}\eta_1^* & 0 & 0 & 0 & \eta_2^{}\eta_2^* & \eta_2 \\
\eta_1^* & 0 & 0 & 0 & \eta_2^* & 1
\ea\right)
$$
It is positive, because $G=\rho\rho^*$ where $\rho\in \bC^2\otimes\cK=\cK\oplus \cK$ is equal to $\rho=\left( \eta_1 \; 0 \; 0\,  |\, 0 \; \eta_2 \; 1\right)^\rT.$ Further, note that
$$\id_{\bM_2(\bC)}\otimes \phi(G)=
\left(\ba{ccc|ccc}
\phi_1(\eta_1^{}\eta_1^*) & 0 & 0 & 0 & 0 & B_1\eta_1 \\
0 & \omega_1(\eta_1^{}\eta_1^*)P_1 & 0 & 0 & 0 & 0 \\
0 & 0 & 0 & 0 & 0 & 0 \\ \hline
0 & 0 & 0 & \omega_2(\eta_2^{}\eta_2^*) & 0 & 0 \\
0 & 0 & 0 & 0 & \phi_2(\eta_2^{}\eta_2^*) & B_2\eta_2 \\
(B_1\eta_1)^* & 0 & 0 & 0 & (B_2\eta_2)^* & 1
\ea\right)
$$
Since $\phi$ is $2$-positive, the above matrix is positive definite. Then, applying Proposition \ref{blockpos} to this matrix and coefficients of $\rho$, we conclude that the following scalar matrix is positive definite
$$
\left(\ba{ccc|ccc}
\la y_1,\phi_1(\eta_1^{}\eta_1^*)y_1\ra & 0 & 0 & 0 & 0 & \la y_1,B_1\eta_1\ra \\
0 & 0 & 0 & 0 & 0 & 0 \\
0 & 0 & 0 & 0 & 0 & 0 \\ \hline
0 & 0 & 0 & 0 & 0 & 0 \\
0 & 0 & 0 & 0 & \la y_2,\phi_2(\eta_2^{}\eta_2^*)y_2\ra & \la y_2,B_2\eta_2\ra \\
\la B_1\eta_1,y_1\ra & 0 & 0 & 0 & \la B_2\eta_2,y_2\ra & 1
\ea\right).
$$
Its determinant is equal to
\be
\lefteqn{\mu_1(\eta_1,y_1)^2 \mu_2(\eta_2,y_2)^2 -\mu_1(\eta_1,y_1)^2 \veps_2(\eta_2,y_2)^2 -\mu_2(\eta_2,y_2)^2 \veps_1(\eta_1,y_1)^2=} \\
&=&\big(\veps_1(\eta_1,y_1)^2 +\delta_1(\eta_1,y_1)^2\big)\big(\veps_2(\eta_2,y_2)^2 +\delta_2(\eta_2,y_2)^2\big)\\ && {}- \big(\veps_1(\eta_1,y_1)^2 +\delta_1(\eta_1,y_1)^2\big) \veps_2(\eta_2,y_2)^2 
-\big(\veps_2(\eta_2,y_2)^2 +\delta_2(\eta_2,y_2)^2\big)\veps_1(\eta_1,y_1)^2 \\
&=& \delta_1(\eta_1,y_1)\delta_2(\eta_2,y_2) - \veps_1(\eta_1,y_1)\veps_2(\eta_2,y_2)
\ee
Since it is nonnegative, inequality \eqref{zui} follows.
\end{proof}
\begin{remark}
In Theorem \ref{t:3cp} it will be shown that in the special case $\cK_i=\cH_i=\bC$, $i=1,2$, the converse implication is also true. Moreover, 2-positivity is equivalent to complete positivity.
\end{remark}

\subsection{Canonical merging of a pair of positive maps}
The aim of this subsection is to show that for two maps $\phi_1$ and $\phi_2$ which allow condition (2) of Proposition \ref{mergprop} for nonzero $\psi_i$ or $\chi_i$, it is possible nontrivial positive merging. Moreover, it results in interesting examples of nondecomposable positive maps.

Before we formulate the main result of this subsection let us recall some properties of $2$-positive and $2$-copositive maps. Let $\phi:B(\mathfrak{k})\to B(\mathfrak{h})$ be a nonzero positive map, where $\mathfrak{k}$, $\mathfrak{h}$ are some Hilbert spaces.
Since $\phi$ is nonzero, there are normalized, vectors $\xi\in\mathfrak{k}$ and $x\in\mathfrak{h}$ such that for some positive number $\lambda$,
\beq
\phi(\xi\xi^*)x=\lambda x.
\eeq
Define two operators $B,C:\mathfrak{k}\to\mathfrak{h}$ by
\beq
\label{B0}
B\eta=\lambda^{-1/2}\phi(\eta\xi^*)x,\qquad\eta\in\cK,
\eeq
\beq
\label{C0}
C\eta=\lambda^{-1/2}\phi(\xi\ov{\eta}^*)x,\qquad\eta\in\cK,
\eeq
and maps $\psi,\chi:B(\mathfrak{k})\to B(\mathfrak{h})$ by
\beq
\psi(X)=BXB^*,\qquad
\chi(X)=CX^\rT C^*,\qquad\qquad X\in B(\mathfrak{k}).
\eeq
It was shown by one of the authors in \cite{MM10} (see also \cite{Sto_pos}) that if $\phi$ is 2-positive (respectively 2-copositive), then $\psi\leq \phi$ (respectively $\chi\leq \phi$).

Now, let us come back to the pair of maps $\phi_i:B(\cK_i)\to B(\cH_i)$, $i=1,2$, where $\cK_i$, $\cH_i$ are finite dimensional Hilbert spaces. We assume that both are nonzero, hence for $i=1,2$, there are normalized vectors $\xi_i\in\cK_i$ and $x_i\in\cH_i$ such that for positive constants $\lambda_i$,
\beq
\label{lambdai}
\phi_i^{}(\xi_i^{}\xi_i^*)x_i^{}=\lambda_i^{} x_i^{}.
\eeq
Let $B_1:\cK_1\to\cH_1$ and $C_2:\cK_2\to\cH_2$ be given as in \eqref{B0} and \eqref{C0} by
\beq
\label{B1}
B_1\eta_1=\lambda_1^{-1/2}\phi_1(\eta_1^{}\xi_1^*)x_1,\qquad\eta_1\in\cK_1,
\eeq
\beq
\label{C2}
C_2\eta_2=\lambda_2^{-1/2}\phi_2(\xi_2^{}\ov{\eta_2^{}}^*)x_2,\qquad\eta\in\cK_2.
\eeq
Further, let $C_1:\cK_1\to\cH_1$ and $B_2:\cK_2\to\cH_2$ be zero operators. Finally, let functionals $\omega_i$ on $B(\cK_i)$, $i=1,2$, be defined by
\beq
\label{om1}
\omega_1(X)=\Tr (B_1XB_1^*),\qquad X\in B(\cK_1),
\eeq
\beq
\label{om2}
\omega_2(X)=\Tr (C_2X^\rT C_2^*),\qquad X\in B(\cK_2),
\eeq
\begin{defi}
\label{canon}
A \textit{canonical merging} of positive maps $\phi_1$ and $\phi_2$ is a merging by means of $B_i,C_i,\omega_i$ described above.
\end{defi}
\begin{thm}
\label{2pos2copos}
If $\phi_1$ is a $2$-positive map and $\phi_2$ is a $2$-copositive one, then the canonical merging $\phi$ of $\phi_1$ and $\phi_2$ is a positive nondecomposable map.
\end{thm}
\begin{proof}
Firstly, we will show positivity of $\phi$.
We will use the criterion of Theorem \ref{t:pos}. Obviously, condition (i) is satisfied.
Let $\eta_i\in\cK_i$, $y_i\in\cH_i$, where $i=1,2$, be arbitrary vectors. Then
$$
\varepsilon_1(\eta_1,y_1)=|\la y_1,B_1\eta_1\ra| ,\qquad\qquad \varepsilon_2(\eta_2,y_2)=|\la y_2,C_2\ov{\eta_2}\ra |.
$$
Observe that
\be
\veps_1(\eta_1,y_1)&=&|\la y_1,B_1\eta_1\ra|=\la y_1,(B_1\eta_1)(B_1\eta_1)^*y_1\ra^{1/2}=
\la y_1^{},B_1^{}\eta_1^{}\eta_1^*B_1^*y_1^{}\ra^{1/2}= \\
&=& \la y_1^{},\psi_1^{}(\eta_1^{}\eta_1^*)y_1^{}\ra^{1/2}\leq
\la y_1^{},\phi_1^{}(\eta_1^{}\eta_1^*)y_1^{}\ra^{1/2}=\mu_1(\eta_1,y_1)
\ee
where the inequality follows from the result of \cite{MM10}. Similarly one can show that $\veps_2(\eta_2,y_2)\leq \mu_2(\eta_2,y_2)$. Thus, condition (ii) of Theorem \ref{t:pos} is fulfilled. Further, we have
$$
\sigma_1(\eta_1,y_2)=\Vert B_1\eta_1\Vert\Vert P_2y_2\Vert ,\qquad\qquad
\sigma_2(\eta_2,y_1)=\Vert C_2\ov{\eta_2}\Vert\Vert P_1 y_1\Vert
$$
Recall that $P_i$ is a projection onto the range of $\phi_i(\jed_{B(\cK_i)})$. Since $B_1^{}B_1^*\leq \phi_1(\jed_{B(\cK_1)})$ and $C_2^{}C_2^*\leq \phi_2(\jed_{B(\cK_2)})$, we have $B_1=P_1B_1$ and $C_2=P_2C_2$. Thus,
\be
\veps_1(\eta_1,y_1)\,\veps_2(\eta_2,y_2)&=&|\la y_1,B_1\eta_1\ra||\la y_2,C_2\ov{\eta_2}\ra |=|\la P_1y_1,B_1\eta_1\ra||\la P_2y_2,C_2\ov{\eta_2}\ra | \\
&\leq &
\Vert P_1y_1\Vert \Vert B_1\eta_1\Vert \Vert P_2y_2\Vert \Vert C_2\ov{\eta_2}\Vert
=
\sigma_1(\eta_1,y_2)\,\sigma_2(\eta_2,y_1).
\ee
Hence, the condition (iii) of Theorem \ref{t:pos} holds and positivity of $\phi$ follows.

In order to prove nondecomposability we will show that the map $\phi$ is an entanglement witness for some PPT operator, i.e. there is $Z\in \mathbb{T}$ such that $\Tr(\cC_\phi^\rT Z)<0$. To this end let us fix some orthonormal bases $\cE_\cK$ and $\cE_\cH$ of spaces $\cK$ and $\cH$ respectively, in such a way that
$$\cE_\cK=(\eps_{1,1},\ldots,\eps_{1,k_1},\eps_{2,1},\ldots,\eps_{2,k_2},\eps),\qquad
\cE_\cH=(e_{1,1},\ldots,e_{1,l_1},e_{2,1},\ldots,e_{2,l_2},e),
$$
where $(\eps_{i,1},\ldots,\eps_{i,k_i})$ and $(e_{i,1},\ldots,e_{i,l_i})$ are bases of $\cK_i$ and $\cH_i$ respectively, $i=1,2$, and $\cK_3=\bC \eps$, $\cH_3=\bC e$. Moreover, we assume that $\eps_{i,1}=\xi_i$ and $e_{i,1}=x_i$ for $i=1,2$ and elements of $\cE_\cK$ and $\cE_\cH$ are real vectors, i.e. they are invariant with respect to antilinear involutions.

Define $Z\in B(\cK)\otimes B(\cH)$ by
\be
Z&=&\lambda_1^{-1}\, \eps_{1,1}^{}\eps_{1,1}^*\otimes (e_{1,1}^{}e_{1,1}^*+ee^*) + \lambda_2^{-1}\, \eps_{2,1}^{}\eps_{2,1}^*\otimes (e_{2,1}^{}e_{2,1}^*+ee^*) \\
&&{} -\lambda_1^{-1/2}(\eps_{1,1}\eps^*\otimes e_{1,1}e^*+\eps\eps_{1,1}^*\otimes ee_{1,1}^*) -\lambda_2^{-1/2}(\eps_{2,1}^{}\eps^*\otimes ee_{2,1}^*+\eps\eps_{2,1}^*\otimes e_{2,1}e^*) \\
&&{}+\eps\eps^*\otimes (e_{1,1}^{}e_{1,1}^*+e_{2,1}^{}e_{2,1}^*+ee^*)
\ee
It is positive because it can be decomposed in the form
\be
Z&=&\lambda_1^{-1}\, (\eps_{1,1}\otimes e)(\eps_{1,1}\otimes e)^* + \lambda_2^{-1}(\eps_{2,1}\otimes e_{2,1})(\eps_{2,1}\otimes e_{2,1})^* + (\eps\otimes e_{1,1})(\eps\otimes e_{1,1})^* \\
&&{}+(\lambda_1^{-1/2}\eps_{1,1}\otimes e_{1,1}-\eps\otimes e)(\lambda_1^{-1/2}\eps_{1,1}\otimes e_{1,1}-\eps\otimes e)^* \\
&&{}+(\lambda_2^{-1/2}\eps_{2,1}\otimes e-\eps\otimes e_{2,1})(\lambda_2^{-1/2}\eps_{2,1}\otimes e-\eps\otimes e_{2,1})^*
\ee
In a simillar way we get
\be
Z^\Gamma&=&\lambda_1^{-1}\, (\eps_{1,1}\otimes e_{1,1})(\eps_{1,1}\otimes e_{1,1})^* + \lambda_2^{-1}(\eps_{2,1}\otimes e)(\eps_{2,1}\otimes e)^* + (\eps\otimes e_{2,1})(\eps\otimes e_{2,1})^* \\
&&{}+(\lambda_1^{-1/2}\eps_{1,1}\otimes e-\eps\otimes e_{1,1})(\lambda_1^{-1/2}\eps_{1,1}\otimes e-\eps\otimes e_{1,1})^* \\
&&{}+(\lambda_2^{-1/2}\eps_{2,1}\otimes e_{2,1}-\eps\otimes e)(\lambda_2^{-1/2}\eps_{2,1}\otimes e_{2,1}-\eps\otimes e)^*
\ee
Therefore, $Z$ is a PPT operator.
We will show that $\Tr(\cC_\phi^\rT Z)<0$. Firstly, let us note that $Z^\rT=Z$, hence
\beg
\Tr(\cC_\phi^\rT Z)&=&\Tr(\cC_\phi Z)=\label{TrCZ}\\
&=&
\lambda_1^{-1}(\la e_{1,1},\phi(\eps_{1,1}^{}\eps_{1,1}^*)e_{1,1}^{}\ra
+ \la e,\phi(\eps_{1,1}^{}\eps_{1,1}^*)e\ra) \nonumber \\
&&{}+ \lambda_2^{-1}(\la e_{2,1}^{},\phi(\eps_{2,1}^{}\eps_{2,1}^*)e_{2,1}^{}\ra+ \la e,\phi(\eps_{2,1}^{}\eps_{2,1}^*)e\ra) \nonumber\\
&& {}-\lambda_1^{-1/2}(\la e_{1,1}^{},\phi(\eps_{1,1}^{}\eps^*)e\ra +\la e,\phi(\eps\eps_{1,1}^*)e_{1,1}^{}\ra) \nonumber\\
&&{}- \lambda_2^{-1/2}(\la e^{},\phi(\eps_{2,1}^{}\eps^*)f_{2,1}\ra +\la e_{2,1}^{},\phi(\eps\eps_{2,1}^*)e\ra) \nonumber\\
&&{}+
\la e_{1,1}^{},\phi(\eps\eps^*)e_{1,1}^{}\ra+\la e_{2,1}^{},\phi(\eps\eps^*)e_{2,1}^{}\ra + \la e,\phi(\eps\eps^*)e\ra \nonumber
\eeg
Secondly, observe that
\beg
\phi(\eps_{1,1}^{}\eps_{1,1}^*)&=&\left(\ba{ccc}
\phi_1(\eps_{1,1}^{}\eps_{1,1}^*) & 0 & 0 \\
0 & \Vert B_1\eps_{1,1}\Vert^2 P_2 & 0 \\
0 & 0 & 0
\ea\right),\\
\phi(\eps_{2,1}^{}\eps_{2,1}^*)&=&\left(\ba{ccc}
\Vert C_2\eps_{2,1}\Vert^2 P_2 & 0 & 0 \\
0 & \phi_2(\eps_{2,1}^{}\eps_{2,1}^*) & 0 \\
0 & 0 & 0
\ea\right),
\eeg
\beq
\phi(\eps_{1,1}^{}\eps^*)=\left(\ba{ccc}
0 & 0 & B_1\eps_{1,1} \\
0 & 0 & 0 \\
0 & 0 & 0
\ea\right),\quad
\phi(\eps_{2,1}^{}\eps^*)=\left(\ba{ccc}
0 & 0 & 0 \\
0 & 0 & C_2\ov{\eps_{2,1}} \\
0 & 0 & 0
\ea\right),\quad
\phi(\eps\eps^*)=\left(\ba{ccc}
0 & 0 & 0 \\
0 & 0 & 0 \\
0 & 0 & 1
\ea\right)
\eeq
Therefore, having in mind \eqref{lambdai}, we are lead to
\beq
\la e_{i,1}^{},\phi(\eps_{i,1}^{}\eps_{i,1}^*)e_{i,1}^{}\ra=\la x_i,\phi_i(\xi_i^{}\xi_i^*)x_i\ra=\lambda_i,\qquad i=1,2,
\eeq
\beq
\la e,\phi(\eps_{i,1}^{}\eps_{i,1}^*)e\ra =0, \qquad i=1,2.
\eeq
Taking into account also \eqref{B1} and \eqref{C2} we get
\beq
\la e_{1,1}^{},\phi(\eps_{1,1}^{}\eps^*)e\ra=\la e_{1,1}^{},B_1^{}\eps_{1,1}^{}\ra = \la x_1^{},B_1^{}\xi_1^{}\ra=\lambda_1^{-1/2}\la x_1^{},\phi_1^{}(\xi_1^{}\xi_1^*)x_1^{}\ra=\lambda_1^{1/2}
\eeq
\beq
\la e_{2,1}^{},\phi(\eps^{}\eps_{2,1}^*)e\ra=\la e_{2,1}^{},C_2^{}\ov{\eps_{2,1}^{}}\ra = \la x_2^{},C_2^{}\ov{\xi_2^{}}\ra=\lambda_2^{-1/2}\la x_2^{},\phi_2^{}(\xi_2^{}\xi_2^*)x_2^{}\ra=\lambda_2^{1/2}
\eeq
We have also
$$
\la e,\phi(\eps\eps^*)e\ra =1\quad\qquad\textrm{and}\quad\qquad
\left\la e_{i,1}^{},\phi(\eps\eps^*)e_{i,1}^{}\right\ra=0,\quad i=1,2.
$$
Finally, it follows from \eqref{TrCZ} that $\Tr(\cC_\phi^\rT Z)=-1$.
\end{proof}

The above theorem provides a useful tool for constructing examples of nondecomposable maps. Let us illustrate it by the following example.
\begin{example}
For $\gamma\geq 0$, define $\psi_\gamma:B(\bC^k)\to B(\bC^k)$ by
\beq
\label{psig}
\psi_\gamma(X)=(\gamma+1)\Tr(X)\jed -X
\eeq
It is known (\cite{Sto_ks}) that $\psi_\gamma$ is $2$-positive if and only if $\gamma\geq 1$. We will consider merging of maps $\phi_1:B(\bC^{k_1})\to B(\bC^{k_1})$ and $\phi_2:B(\bC^{k_2})\to B(\bC^{k_2})$ given by
$$ 
\phi_1(X)=(\gamma_1+1)\Tr(X)\jed-X,\quad X\in B(\bC^{k_1}),
$$ 
$$ 
\phi_2(X)=(\gamma_2+1)\Tr(X)\jed-X^\rT,\quad X\in B(\bC^{k_2}).
$$ 
Clearly, if $\gamma_1\geq 1$ and $\gamma_2\geq 1$, then $\phi_1$ is 2-positive and $\phi_2$ is 2-copositive. Therefore, we can apply Theorem \ref{2pos2copos}.
To this end we should specify vectors $\xi_1,x_1\in\bC^{k_1}$ and $\xi_2,x_2\in\bC^{k_2}$ satisfying condition \eqref{lambdai}. For $i=1,2$, let $e_{i,1},\ldots,e_{i,k_i}$ be the standard orthonormal basis of $\bC^{k_i}$ and let $e$ be a unit vector spanning $\bC$. Hence $e_{1,1},\ldots,e_{1,k_1},e_{2,1},\ldots,e_{2,k_2},e$ form an orthonormal basis for $\bC^{k_1}\oplus\bC^{k_2}\oplus\bC$. We assume that the vectors are invariant with respect to the antilinear involution. Let $\xi_i=x_i=e_{i,1}$ for $i=1,2$. Since $\phi_i(e_{i,1}^{}e_{i,1}^*)=(\gamma_i+1)\jed_{k_i} - e_{i,1}^{}e_{i,1}^*$, we have $\phi_i(e_{i,1}^{}e_{i,1}^*)e_{i,1}^{}=\gamma_i^{}e_{i,1}^{}$, hence $\lambda_i^{}=\gamma_i^{}$. One checks that
$$B_1=\gamma_1^{1/2}e_{1,1}^{}e_{1,1}^*-\gamma_1^{-1/2}(\jed_{k_1}^{}-e_{1,1}^{}e_{1,1}^*),\qquad
C_2=\gamma_2^{1/2}e_{2,1}^{}e_{2,1}^*-\gamma_2^{-1/2}(\jed_{k_2}^{}-e_{2,1}^{}e_{2,1}^*).
$$
Hence
$$
B_1^{}B_1^*=\gamma_1^{-1}\jed_{k_1}+(\gamma_1^{}-\gamma_1^{-1})e_{1,1}^{}e_{1,1}^*,
\qquad
C_1^{}C_1^*=\gamma_2^{-1}\jed_{k_2}+(\gamma_2^{}-\gamma_2^{-1})e_{2,1}^{}e_{2,1}^*
$$
and functionals $\omega_1$, $\omega_2$ defined in \eqref{om1}, \eqref{om2} are given by
$$
\omega_1(X)=\gamma_1^{-1}\Tr(X)+(\gamma_1^{}-\gamma_1^{-1})\la e_{1,1},Xe_{1,1}\ra,\qquad X\in B(\bC^{k_1}),
$$
$$
\omega_2(X)=\gamma_2^{-1}\Tr(X)+(\gamma_2^{}-\gamma_2^{-1})\la e_{2,1},Xe_{2,1}\ra,\qquad X\in B(\bC^{k_2}).
$$
Therefore, if $\phi:B(\bC^{k_1}\oplus\bC^{k_2}\oplus\bC)\to B(\bC^{k_1}\oplus\bC^{k_2}\oplus\bC)$ is the map constructed in Theorem \ref{2pos2copos}, then for $X\in B(\bC^{k_1}\oplus\bC^{k_2}\oplus\bC)$ given in the block form $X=(X_{ij})$,
\beg
\label{gammas}
\lefteqn{\phi(X)=}\\
&=&
\left(\ba{ccc}
r_1\jed_{k_1}-X_{11} & 0 & \iota(\gamma_1)(X_{13})_1^{} e_{1,1}-\gamma_1^{-1/2}X_{13} \\
0 & r_2\jed_{k_2}-X_{22}^\rT & \iota(\gamma_2)(X_{32}^\rT)_1^{} e_{2,1}-\gamma_2^{-1/2}X_{32}^\rT \\
\iota(\gamma_1)(X_{31}^\rT)_1^{} e_{1,1}^*-\gamma_1^{-1/2}X_{31} & \iota(\gamma_2)(X_{23})_1^{} e_{2,1}^*-\gamma_2^{-1/2}X_{23}^\rT & X_{33}
\ea\right)\nonumber
\eeg
where $\iota(x)=x^{1/2}+x^{-1/2}$ for $x>0$, $(\eta)_1$ denotes the first coordinate of a vector $\eta\in\bC^{k_i}$ and
$$
r_1=(\gamma_1+1)\Tr(X_{11})+\gamma_2^{-1}\Tr(X_{22})+ (\gamma_2-\gamma_2^{-1})\la e_{2,1},X_{22}e_{2,1}\ra
$$
$$
r_2=(\gamma_2+1)\Tr(X_{22})+\gamma_1^{-1}\Tr(X_{11})+ (\gamma_1-\gamma_1^{-1})\la e_{1,1},X_{11}e_{1,1}\ra
$$
\end{example}
\begin{cor}
\label{wn22}
If $\gamma_1\geq 1$ and $\gamma_2\geq 1$, then the map $\phi : B(\bC^{k_1+k_2+1})\to B(\bC^{k_1+k_2+1})$ given by \eqref{gammas} is positive and nondecomposable.
\end{cor}
\begin{example}
For illustration, let us describe precisely the previous example for $k_1=k_2=3$ and $\gamma_1=\gamma_2=\gamma\geq 1$. In this case we obtain a map $\phi:\bM_7(\bC)\to\bM_7(\bC)$ given by
$$
\phi(X)=\left(\ba{ccccccc}
R-x_{11} & -x_{12} & -x_{13} & 0 & 0 & 0 & \gamma^{1/2} x_{17} \\
-x_{21} & R-x_{22} & -x_{23} & 0 & 0 & 0 & -\gamma^{-1/2} x_{27} \\
-x_{31} & -x_{32} & R-x_{33} & 0 & 0 & 0 & -\gamma^{-1/2} x_{37} \\
0 & 0 & 0 & Q-x_{44} & -x_{54} & -x_{64} & \gamma^{1/2} x_{74} \\
0 & 0 & 0 & -x_{45} & Q-x_{55} & -x_{65} & -\gamma^{-1/2} x_{75} \\
0 & 0 & 0 & -x_{46} & -x_{56} & Q-x_{66} & -\gamma^{-1/2} x_{76} \\
\gamma^{1/2}x_{71} & -\gamma^{-1/2}x_{72} & -\gamma^{-1/2}x_{73} & \gamma^{1/2}x_{47} & -\gamma^{-1/2}x_{57} & -\gamma^{-1/2}x_{67} & x_{77}
\ea\right)
$$
where
$$R=(\gamma+1)(x_{11}+x_{22}+x_{33})+\gamma x_{44}+\gamma^{-1}(x_{55}+x_{66})$$
$$Q=\gamma x_{11}+\gamma^{-1}(x_{22}+x_{33})+(\gamma+1)(x_{44}+x_{55}+x_{66})$$
By Corollary \ref{wn22} the above map is positive and nondecomposable.
\end{example}

\section{Special classes of merging}
The purpose of this section is to analyze maps $\phi_{\cK_1,\cK_2}$, $\Lambda_{\cK_1,\cK_2}$, $\Omega_{\cK_1,\cK_2}$ (and their variations) described in Examples \ref{e:MOgen}, \ref{e:RSCgen} and \ref{e:RSCgen2} which are obtained as merging of some specific positive maps.
\subsection{Examples of exposed positive maps}
Suppose that linear operators $A_1:\cK_1\to\cH_1$ and $A_2:\cK_2\to\cH_2$ are given.
Let $\phi_1:B(\cK_1)\to B(\cH_1)$ and $\phi_2:B(\cK_2)\to B(\cH_2)$ be maps given by
\beq
\label{e:phi12}
\phi_1(X)=A_1XA_1^*,\quad X\in B(\cK_1),
\qquad\qquad
\phi_2(X)=A_2X^\rT A_2^*,\quad X\in B(\cK_2).
\eeq
It was shown in \cite{MM} (see also \cite{YH}) that $\phi_1$ and $\phi_2$ are exposed elements in the cones $\fP(\cK_1,\cH_1)$ and $\fP(\cK_2,\cH_2)$ respectively.
Our aim is to show the following
\begin{thm}\label{main}
For any finite dimensional Hilbert spaces $\cK_1,\cK_2,\cH_1,\cH_2$ and any pair of nonzero operators $A_1\in B(\cK_1,\cH_1)$ and $A_2\in B(\cK_2,\cH_2)$ the canonical merging $\phi$ of maps $\phi_1$ and $\phi_2$ given by \eqref{e:phi12} (cf. Definition \ref{canon}) is an exposed positive map.
\end{thm}

It was shown in \cite{MM10} that for maps \eqref{e:phi12} the operators $B_1$ and $C_2$ given by \eqref{B1} and \eqref{C2} are equal to $A_1$ and $A_2$ respectively except for a phases, i.e. there are real numbers $\theta_1$, $\theta_2$ such that $B_1=\e{1}A_1$ and $C_2=\e{2}A_2$. Therefore, the map $\phi$ is of the form
\beq
\label{block}
\phi(X)=
\left(\ba{ccc}
A_1X_{11}A_1^* + \Tr (A_2X_{22}^\rT A_2^*)P_1 & 0 & \e{1}A_1X_{13} \\[2mm]
0 & A_2X_{22}^\rT A_2^* + \Tr (A_1X_{11}A_1^*)P_2 & \e{2}A_2X_{32}^\rT \\
\me{1}X_{31}A_1^* & \me{2}X_{23}^\rT A_2^* & X_{33}
\ea\right)
\eeq
where $X=(X_{ij})\in B(\cK)$ and $P_i$ is the projection on $\ran A_i$ for $i=1,2$.

Having in mind Corollary \ref{c:nu} we can generalize this map a bit more. Thus, we came to a formulation of the main result of this section
\begin{thm}
\label{mainm}
Let $\cK_1,\cK_2,\cH_1,\cH_2$ be finite dimensional Hilbert spaces, $\cK=\cK_1\oplus\cK_2\oplus\bC$, $\cH=\cH_1\oplus\cH_2\oplus\bC$ and let $A_1\in B(\cK_1,\cH_1)$ and $A_2\in B(\cK_2,\cH_2)$ be nonzero operators. Then for any $\theta_1,\theta_2\in\bR$ and $\nu\in\bR$, $\nu>0$, the map $\phi:B(\cK)\to B(\cH)$ given by
\beq
\label{block}
\phi(X)=
\left(\ba{ccc}
A_1X_{11}A_1^* + \nu\, \Tr (A_2X_{22}^\rT A_2^*)P_1 & 0 & \e{1}A_1X_{13} \\[2mm]
0 & A_2X_{22}^\rT A_2^* + \nu^{-1}\Tr (A_1X_{11}A_1^*)P_2 & \e{2}A_2X_{32}^\rT \\
\me{1}X_{31}A_1^* & \me{2}X_{23}^\rT A_2^* & X_{33}
\ea\right)
\eeq
for $X=(X_{ij})\in B(\cK)$, is an exposed positive map.
\end{thm}

Before the proof we will list some preliminary results. The first lemma seems to be known. We attach its proof for reader's convenience.
\begin{lem}
\label{lift}
Suppose $\cK$, $\cH$, $\ti{\cK}$, $\ti{\cH}$ are finite dimensional Hilbert spaces and $E:\ti{\cK}\to \cK$, $F:\ti{\cH}\to\cH$ are injective linear operators. For a linear map $\tilde{\phi}:B(\ti{\cK})\to B(\ti{\cH})$ we consider a map $\cL_{E,F}\ti{\phi}:B(\cK)\to B(\cH)$, defined by $\cL_{E,F}\ti{\phi}(X)=F\ti{\phi}(E^*XE)F^*$ for $X\in B(\cK)$. If $\ti{\phi}$ is an exposed element of the cone $\bP(\ti{\cK},\ti{\cH})$ then $\cL_{E,F}\ti{\phi}$ is an exposed element of $\bP(\cK,\cH)$.
\end{lem}
\begin{proof}
Let $\ti{\cZ}$ (respectively $\cZ$) denote the set of all pairs $(\ti{\eta},\ti{y})\in\ti{\cK}\times\ti{\cH}$ (respectively $(\eta,y)\in\cK\times\cH$) such that $\la \ti{y},\ti{\phi}(\ti{\eta}\ti{\eta}^*)\ti{y}\ra=0$ (respectively $\la y,\cL_{E,F}\ti{\phi}(\eta\eta^*)y\ra=0$). One can easily observe that $(\eta,y)\in\cZ$ if and only if $(E^*\eta,F^*y)\in\ti{\cZ}$. It follows that
\beq
\label{jup}
\left(\ker E^*\times\cH\right)\cup \left(\cK\times \ker F^*\right)\subset \cZ .
\eeq
We will show that $\cL_{E,F}\ti{\phi}$ satisfies condition \eqref{expo}. Assume that $\psi:B(\cK)\to B(\cH)$ is a positive map such that $\la y,\psi(\eta\eta^*)y\ra=0$ for every $(\eta,y)\in\cZ$.
Let $E^+:\cK\to \ti{\cK}$ and $F^+:\cH\to \ti{\cH}$ be Penrose inverses of $E$ and $F$. Since $E$ and $F$ are injective, $E^+$ and $F^+$ are characterized as unique operators which satisfy the following set of relations: $E^+E=\id_{\ti{\cK}}$, $F^+F=\id_{\ti{\cH}}$, $EE^+=P_E$, $FF^+=P_F$, where $P_E\in B(\cK)$ and $P_F\in B(\cH)$ denote orthogonal projections onto ranges of operators $E$ and $F$ respectively.
It follows from \eqref{jup} that for each $\eta\in\cK$, the range of the positive operator $\psi(\eta\eta^*)$ is contained in $\ran F$. Moreover, $\psi(\eta\eta^*)$ depends only on $\eta_E$, where $\eta_E\in \ran E$ is a unique vector such that $\eta=\eta_E+\eta_0$ for some $\eta_0\in \ker E^*$. Therefore, $\psi$ satisfies the following condition
\beq
\label{si}
\psi(X)=P_F\psi(P_EXP_E)P_F
=
FF^+\psi(E^{+*}E^*XEE^+)F^{+*}F^*,\qquad X\in B(\cK).
\eeq
Let $\ti{\psi}:B(\ti{\cK})\to B(\ti{\cH})$ be defined by $\ti{\psi}(\ti{X})=F^+\psi(E^{+*}\ti{X}E^+)F^{+*}$ for $\ti{X}\in B(\ti{\cK})$. Let $(\ti{\eta},\ti{y})\in\ti{\cZ}$. Since $E^*E^{+*}\ti{\eta}=\ti{\eta}$ and $F^*F^{+*}\ti{y}=\ti{y}$, we have $(E^{+*}\ti{\eta},F^{+*}\ti{y})\in \cZ$. Thus, the assumption on $\psi$ implies
$$
\la \ti{y},\ti{\psi}(\ti{\eta}\ti{\eta}^*)\ti{y}\ra=\la \ti{y}, F^+\psi(E^{+*}\ti{\eta}\ti{\eta}^*E^+)F^{+*}\ti{y}\ra=
\la F^{+*}\ti{y},\psi\left((E^{+*}\ti{\eta})(E^{+*}\ti{\eta})^*\right)F^{+*}\ti{y}\ra=0.
$$
Therefore $\ti{\psi}=\lambda\ti{\phi}$ for some $\lambda\in\bR_+$, because $\ti{\phi}$ is an exposed positive map. Observe that the condition \eqref{si} reads as $\psi=\cL_{E,F}\ti{\psi}$. Hence, we arrive at $\psi=\cL_{E,F}\ti{\psi}=\cL_{E,F}(\lambda\ti{\phi})=\lambda\cL_{E,F}\ti{\phi}$.
\end{proof}

\begin{lem}\label{lemoff}
Let $\cK=\bigoplus_{i=1}^n\cK_i$ and $\cH=\bigoplus_{i=1}^n\cH_i$ and let $\phi: B(\cK)\to B(\cH)$ be a positive map such that $\phi(B(\cK_i))\subset B(\cH_i)$ for $i=1,2,\ldots,n$. Then for any $i,j=1,\ldots,n$ such that $i\neq j$, $\phi(B(\cK_j,\cK_i)\subset B(\cH_j,\cH_i)\oplus B(\cH_i,\cH_j)$.
\end{lem}
\begin{proof}
For $k,l=1,\ldots,n$, let $\phi_{kl}:B(\cK)\to B(\cH_l,\cH_k)$ be a linear map defined by
$\phi_{kl}(X)=W_k^*\phi(X)W_l$, $X\in B(\cK)$. Let $i,j$ be given numbers from $1,\ldots, n$ such that $i\neq j$. Consider nonzero vectors $\eta_i\in\cK_i$ and $\eta_j\in\cK_j$. 
In order to prove the Lemma one need to show that $\phi_{kl}(\eta_i\eta_j^*)=0$ until $\{k,l\}=\{i,j\}$.
Let us observe that $\phi_{kl}(\eta_i\eta_j^*)^*=\phi_{lk}(\eta_j\eta_i^*)$. Moreover, it follows from the assumption that
$\phi_{kl}(\eta_i\eta_i^*)=0$ until $k=l=i$ and $\phi_{kl}(\eta_j\eta_j^*)=0$ until $k=l=j$.
The element $\phi((\eta_i+z\eta_j)(\eta_i+z\eta_j)^*)$ is positive in $B(\cH)$ for every $\gamma\in\bC$, and consequently $\phi_{kk}((\eta_i+\gamma\eta_j)(\eta_i+\gamma\eta_j)^*)$ is positive in $B(\cH_k)$ for every $k=1,\ldots,n$. Hence, for any $y_k\in\cH_k$,
$$\la y_k,\phi_{kk}(\eta_i\eta_i^*)y_k\ra+2\Re\,\ov{\gamma}\la y_k,\phi_{kk}(\eta_i\eta_j^*)y_k\ra+|\gamma|^2\la y_k,\phi_{kk}(\eta_i\eta_i^*)y_k\ra\geq 0.$$
Since it holds for any $\gamma\in\bC$,
$$|\la y_k,\phi_{kk}(\eta_i\eta_j^*)y_k\ra|^2\leq  \la y_k,\phi_{kk}(\eta_i\eta_i^*)y_k\ra\la y_k,\phi_{kk}(\eta_i\eta_i^*)y_k\ra.$$
Due to the assumption, at least one of the factors in right hand side of the above inequality is zero, so $\la y_k,\phi_{kk}(\eta_i\eta_j^*)y_k\ra=0$ for any $y_k\in\cH_k$. By Proposition \ref{sesq} we conclude that $\phi_{kk}(\eta_i\eta_j^*)=0$.
Now, let $k\neq l$. It follows from Lemma \ref{blockpos} that
$$
\left|\ba{cc} \la y_k,\phi_{kk}((\eta_i+\gamma\eta_j)(\eta_i+\gamma\eta_j)^*) y_k\ra & \la y_k,\phi_{kl}((\eta_i+\gamma\eta_j)(\eta_i+\gamma\eta_j)^*) y_l\ra \\ \la y_l,\phi_{lk}((\eta_i+\gamma\eta_j)(\eta_i+\gamma\eta_j)^*) y_k\ra & \la y_l,\phi_{ll}((\eta_i+\gamma\eta_j)(\eta_i+\gamma\eta_j)^*) y_l\ra \ea\right|\geq 0
$$
for every $ y_k\in\cH_k$ and $ y_l\in\cH_l$. Assume now that $\{k,l\}\neq\{i,j\}$. It follows that at least one diagonal term is zero. Hence off-diagonal terms should also vanish, so
$$\ov{\gamma}\la y_k,\phi_{kl}(\eta_i\eta_j^*) y_l\ra+\gamma\la  y_k,\phi_{kl}(\eta_j\eta_i^*) y_l\ra=0.$$
Since it holds for any $\gamma\in\bC$, $\la y_k,\phi_{kl}(\eta_i^{}\eta_j^*) y_l\ra=0$, and consequently $\phi_{kl}(\eta_i^{}\eta_j^*)=0$.
\end{proof}
\begin{cor}\label{corRQ}
Let $\cK$ and $\cH$ be Hilbert spaces. Assume $\cK=\cK_1\oplus\cK_2$ and $\cH=\cH_1\oplus \cH_2$ where $\cK_2=\bC \eps$ and $\cH_2=\bC e$ are one dimensional subspaces generated by unit vectors $\eps\in\cK$ and $e\in\cH$. Let $\phi: B(\cK)\to B(\cH)$ be a positive map such that $\phi(B(\cK_i))\subset B(\cH_i)$ for $i=1,2$. Then there are linear maps $R,Q:\cK_1\to\cH_1$ such that
\beq\label{RQ}
\phi(\eta \eps^*)= (R\eta)e^* + e(Q\ov{\eta})^*,\qquad \eta\in\cK_1.
\eeq
\end{cor}
\begin{proof}
It follows from Lemma \ref{lemoff} that $\phi(\eta \eps^*)=\phi_{12}(\eta \eps^*)+\phi_{21}(\eta \eps^*)\in B(\cH_2,\cH_1)\oplus B(\cH_1,\cH_2)= B(\bC e,\cH_1)\oplus B(\cH_1,\bC e)$ for any $\eta\in\cK_1$.
Every element of $B(\bC e,\cH_1)$ is of the form $y e^*$ for some $y\in\cH_1$ while elements of $B(\cH_1,\bC e)$ are of the form $ey^*$. It follows that $\phi(\eta \eps^*)=(R\eta)e^*$ for some $R\eta\in\cH_1$ and the mapping $R:\cK_1\to\cH_1$ should be linear. Analogously, $\phi_{21}(\eta \eps^*)=e(Q'\eta)^*$ for some $Q'\eta\in\cH_1$ where $Q':\cK_1\to\cH_1$ is antilinear. In order to obtain the form \eqref{RQ} take $Q:\cK_1\to\cH_1$ defined by $Q\eta=Q'\ov{\eta}$, $\eta\in\cK_1$, which is a linear map.
\end{proof}

Let $V$ and $W$ be complex vector spaces. We say that a map $\Phi:V\times V\to W$ is a  \textit{sesquiline\-ar vector valued form} if
\beq
\Phi(\alpha_1u_1+\alpha_2u_2,\beta_1v_1+\beta_2v_2)= \ov{\alpha_1}\beta_1\Phi(u_1,v_1)+ \ov{\alpha_1}\beta_2\Phi(u_1,v_2)+ \ov{\alpha_2}\beta_1\Phi(u_2,v_1)+ \ov{\alpha_2}\beta_2\Phi(u_2,v_2)
\eeq
for every $u_1,u_2,v_1,v_2\in V$ and $\alpha_1,\alpha_2,\beta_1\beta_2\in\bC$. The following fact will be used several times
\begin{prop}\label{sesq}
If $\Phi(v,v)=0$ for every $v\in V$, then $\Phi(u,v)=0$ for every pair of vectors $u,v\in V$.
\end{prop}
\begin{proof}
It follows from the polarization identity
\beq
\Phi(u,v)=
\frac{1}{4}\sum_{k=0}^3(-i)^k \Phi\left(u+i^kv,u+i^kv\right)
\eeq
which is well known for scalar products (see for example \cite[Theorem 0.19]{Tes}).
\end{proof}

Now, we show that the statement of Theorem \ref{main} is true for special case $A_i=\id_{\cK_i}$, $\theta_i=0$ for $i=1,2$, and $\mu=1$. Namely, we have
\begin{thm}
\label{special}
If $\cK_1$, $\cK_2$ are finite dimensional Hilbert spaces, then the map $\phi_{\cK_1,\cK_2}$ given by \eqref{MOhd} is an exposed positive map.
\end{thm}
\begin{proof}
Any $\eta, y\in\cK$ can be uniquely represented as $\eta=\sum_{i=1}^3\eta_i$ and $ y=\sum_{j=1}^3 y_j$ where $\eta_i, y_i\in\cK_i$. We assume that $\eta_3=\alpha e$ and $y_3=\beta e$ and we will identify $\eta_3$ and $ y_3$ with numbers $\alpha$ and $\beta$ respectively.
Then
\beq
\eta\eta^*=\left(\ba{ccc}
\eta_1\eta_1^* & \eta_1\eta_2^* & \ov{\alpha}\eta_1\\
\eta_2\eta_1^* & \eta_2\eta_2^* & \ov{\alpha}\eta_2 \\
\alpha\eta_1^* & \alpha\eta_2^* & |\alpha|^2
\ea\right)
\eeq
and
\beq
\label{phixixi}
\phi_{\cK_1,\cK_2}(\eta\eta^*)=
\left(\ba{ccc}
\eta_1\eta_1^*+\Vert \eta_2\Vert^2\jed_{B(\cK_1)} & 0 &\ov{\alpha}\eta_1 \\
0 & \ov{\eta_2}\ov{\eta_2}^* + \Vert \eta_1\Vert^2\jed_{B(\cK_2)} & \alpha \ov{\eta_2} \\
\alpha \eta_1^* & \ov{\alpha} \ov{\eta_2}^* & |\alpha|^2
\ea
\right)
\eeq
If $\eta,y\in\cK$, then
\beg
\la y,\phi_{\cK_1,\cK_2}(\eta\eta^*) y\ra 
&=&
|\alpha|^2|\beta|^2 + 2\Re\,\ov{\alpha}\beta\la y_1,\eta_1\ra + 2\Re\,\alpha\beta\la y_2,\ov{\eta_2}\ra\label{kji}
\\
&&{}+|\la y_1,\eta_1\ra|^2 + \Vert \ov{\eta_2}\Vert^2\Vert  y_1\Vert^2 + |\la y_2,\ov{\eta_2}\ra|^2 + \Vert \eta_1\Vert^2\Vert  y_2\Vert^2. \nonumber
\eeg
One can directly compute that
\beq
\label{cZ}
\la y,\phi_{\cK_1,\cK_2}(\eta\eta^*) y\ra=
\begin{cases}
\Vert \eta_1\Vert^2\Vert  y_2\Vert^2 +\Vert \eta_2\Vert^2\Vert  y_1\Vert^2 +|\la  y_1,\eta_1\ra|^2+|\la y_2,\ov{\eta_2}\ra|^2 & \mbox{if $
\alpha=0$,}
\\
|\alpha|^{-2}\left(\left||\alpha|^2\ov{\beta}+\ov{\alpha}\la y_1,\eta_1\ra+\alpha\la  y_2,\ov{\eta_2}\ra\right|^2+\Vert \alpha y_1\otimes \ov{\eta_2}-\ov{\alpha}\eta_1\otimes  y_2\Vert^2\right), & \mbox{if
$\alpha\neq 0$.}
\end{cases}
\eeq
We denote by $\cZ_\phi$ \label{stronaZ} the set of all pairs $\eta, y$ such that $\la  y,\phi_{\cK_1,\cK_2}(\eta\eta^*) y\ra=0$. It follows that $(\eta,y)\in \cZ_\phi$ if and only if one of the following conditions holds:
\beg
\alpha=0,\;\eta_1\neq 0,\; {\eta_2}= 0 & \textrm{and} &  y_1\perp \eta_1,\; y_2=0 \label{0n0}\\
\alpha=0,\;\eta_1= 0,\; {\eta_2}\neq 0 & \textrm{and} &  y_1=0,\; y_2\perp \ov{\eta_2} \label{00n}\\
\alpha=0,\;\eta_1\neq 0,\; {\eta_2}\neq 0 & \textrm{and} &  y_1=0,\;  y_2=0 \label{0nn}\\
\alpha\neq 0,\; \eta_1=0,\; {\eta_2}=0 & \textrm{and} & \beta=0 \label{n00}\\
\alpha\neq 0,\; \eta_1\neq 0,\; {\eta_2}=0 & \textrm{and} & \la \eta_1, y_1\ra=-\ov{\alpha}\beta,\;  y_2=0 \label{nn0}\\
\alpha\neq 0,\; \eta_1=0,\; {\eta_2}\neq 0 & \textrm{and} &  y_1=0,\; \la \ov{\eta_2}, y_2\ra=-\alpha\beta \label{n0n}\\
\alpha\neq 0,\; \eta_1\neq 0,\; {\eta_2}\neq 0 & \textrm{and} &
y_1=-\dfrac{\ov{\alpha}\beta}{\Vert \eta_1\Vert^2+\Vert {\eta_2}\Vert^2} \eta_1, \;
y_2=-\dfrac{\alpha\beta}{\Vert \eta_1\Vert^2+\Vert {\eta_2}\Vert^2}\ov{\eta_2}
\label{nnn}
\eeg

We will show that the map \eqref{e:MOgen} satisfies condition \eqref{expo}. Assume $\psi: B(\cK)\to B(\cK)$ is a positive map such that
\beq\label{basic}
\la y,\psi(\eta\eta^*)y\ra=0
\eeq
for every pair $(\eta,y)\in\cZ_\phi$. We will show that $\psi\in\bR_+\phi_{\cK_1,\cK_2}$. The rest of the proof is divided onto several observations.

\begin{step}\label{st_inclusions}
$\psi(B(\cK_1\oplus\cK_2))\subset B(\cK_1\oplus\cK_2)$ and $\psi(B(\cK_3))\subset B(\cK_3)$.
\end{step}
\begin{proof}
In order to show the first inclusion it is enough to prove that for any $\eta\in \cK_1\oplus \cK_2$, $\psi(\eta\eta^*)\subset B(\cK_1\oplus \cK_2)$. It follows from \eqref{0n0} -- \eqref{0nn} that $(\eta_1+\eta_2,\beta e)\in\cZ_\phi$ for any $\eta_1\in\cK_1$, $\eta_2\in\cK_2$ and $\beta\in\bC$. Thus $\la e,\psi(\eta\eta^*),e\ra=0$ for $\eta\in\cK_1\oplus\cK_2$. Since $\psi(\eta\eta^*)$ is a positive operator, $e\in\ker\psi(\eta\eta^*)$ and the range of $\psi(\eta\eta^*)$ is contained in $e^\perp=\cK_1\oplus\cK_2$.

The second inclusion follows simillarly form the fact that $(\alpha \eps, \eta)\in\cZ_\phi$ for any $\alpha\in\bC$ and $\eta\in\cK_1\oplus\cK_2$ (cf. \eqref{n00}).
\end{proof}
It follows from the above observation that $\psi(ee^*)=\lambda ee^*$ for some $\lambda\geq 0$.
\begin{step}\label{form}
There are sesqulinear vector valued forms $\Psi_{kl}:(\cK_1\oplus\cK_2)\times(\cK_1\oplus\cK_2)\to B(\cK_l,\cK_k)$ for $k,l=1,2$ and linear maps $R_k,Q_k:\cK_1\oplus\cK_2\to\cK_k$ for $k=1,2$ such that
\beq
\psi(\eta\eta^*)=
\left(\ba{ccc}
\Psi_{11}(\eta_0,\eta_0) & \Psi_{12}(\eta_0,\eta_0) & \ov{\alpha}R_1\eta_0+\alpha Q_1\ov{\eta_0} \\
\Psi_{21}(\eta_0,\eta_0) & \Psi_{22}(\eta_0,\eta_0) & \ov{\alpha}R_2\eta_0 +\alpha Q_2\ov{\eta_0} \\
\alpha (R_1\eta_0)^*+\ov{\alpha}(Q_1\ov{\eta_0})^* & \alpha (R_2\eta_0)^* +\ov{\alpha}(Q_2\ov{\eta_0})^* & \lambda|\alpha|^2
\ea\right)\label{gen}
\eeq
for any $\eta\in\cK$ where $\eta=\eta_0+\alpha e$ for a unique $\eta_0\in\cK_1\oplus\cK_2$ and $\alpha\in\bC$.

Moreover, for every $\eta_0,\eta_0'\in\cK_1\oplus \cK_2$,
\beq
\Psi_{kl}(\eta_0,\eta_0')^*=\Psi_{lk}(\eta_0',\eta_0),\qquad k,l=1,2.
\eeq
\end{step}
\begin{proof}
It follows from previous observation that $\psi(\eta_0{\eta_0'}^*)\in B(\cK_1\oplus\cK_2)=\bigoplus_{k,l=1}^2B(\cK_l,\cK_k)$ for any $\eta_0,\eta_0'\in\cK_1\oplus\cK_2$. We define
\beq\label{Psikl}
\Psi_{kl}(\eta_0',\eta_0)=W_k^*\psi(\eta_0{\eta_0'}^*)W_l
\eeq
for $k,l=1,2$, where $W_1,W_2$ are embedding of $\cK_1,\cK_2$ into $\cK$. Clearly $\Psi_{kl}$ are sesquilinear vector valued forms and $\psi(\eta_0{\eta_0'}^*)=\sum_{k,l=1}^2\Psi_{kl}(\eta_0',\eta_0)$, i.e.
\beq
\psi(\eta_0{\eta_0'}^*)=\left(
\ba{ccc}
\Psi_{11}(\eta_0',\eta_0) & \Psi_{12}(\eta_0',\eta_0) & 0 \\
\Psi_{21}(\eta_0',\eta_0) & \Psi_{22}(\eta_0',\eta_0) & 0 \\
0 & 0 & 0 \ea\right)\label{xi0xi0}
\eeq
Further, it follows from Observation \ref{st_inclusions} and Corollary \ref{corRQ} that there are linear maps $R,Q:\cK_1\oplus\cK_2\to\cK_1\oplus\cK_2$ such that $\phi(\eta_0e^*)=(R\eta_0)e^*+e(Q\ov{\eta_0})^*$ for any $\eta_0\in\cK_1\oplus\cK_2$. Define $R_k=F_k^*P$ and $Q_k=F_k^*Q$ for $k=1,2$. Hence
\beq
\psi(\eta_0e^*)=\left(
\ba{ccc}
0 & 0 & R_1\eta_0 \\
0 & 0 & R_2\eta_0 \\
Q_1\ov{\eta_0} & Q_2\ov{\eta_0} & 0
\ea\right)
\label{xie}
\eeq
Having $\psi(\eta\eta)=\psi(\eta_0\eta_0^*)+\ov{\alpha}\psi(\eta_0e^*)+\alpha\psi(e\eta_0^*)+|\alpha|^2\psi(ee^*)$ and taking into account \eqref{xi0xi0}, \eqref{xie}, and  $\psi(ee^*)=\lambda ee^*$ one arrives at \eqref{gen}. 

The second part of the observation follows directly from the definition \eqref{Psikl}.
\end{proof}
In the sequel we will use operators $R_{ik}=R_i\big|_{\cK_k}$ and $Q_{ik}=Q_i\big|_{\cK_k}$ for $i,k=1,2$.
In order to complete the proof of the theorem we need to verify the following conditions:
\label{tabela}
\begin{align*}
\text{(A1)} &\quad
\Ps{1}{1}{1}{1}=\lambda \eta_1\eta_1^*, & 
\text{(A2)} &\quad
\Ps{2}{2}{2}{2}=\lambda \ov{\eta_2}\ov{\eta_2}^*, \\ 
\text{(B1)} &\quad
\Ps{1}{1}{2}{2}=\lambda \Vert \ov{\eta_2}\Vert^2 \jed_{B(\cK_1)}, & 
\text{(B2)} &\quad
\Ps{2}{2}{1}{1}=\lambda \Vert \eta_1\Vert^2 \jed_{B(\cK_2)}, \\
\text{(C1)} &\quad
\Ps{1}{1}{1}{2}=\Ps{1}{1}{2}{1}=0, & 
\text{(C2)} &\quad
\Ps{2}{2}{1}{2}=\Ps{2}{2}{2}{1}=0, \\ 
\text{(D1)} &\quad
\Ps{1}{2}{1}{1}=0, &
\text{(D2)} &\quad
\Ps{1}{2}{2}{2}=0, \\
\text{(E1)} &\quad
\Ps{1}{2}{1}{2}=0, &
\text{(E2)} &\quad
\Ps{1}{2}{2}{1}=0, \\
\text{(F1)} &\quad
R_{11}=\lambda \jed_{B(\cK_1)}, &
\text{(F2)} &\quad
Q_{22}=\lambda \jed_{B(\cK_2)}, \\
\text{(G1)} &\quad
R_{22}=0, &
\text{(G2)} &\quad
Q_{11}=0, \\
\text{(H1)} &\quad
R_{12}=0, &
\text{(H2)} &\quad
Q_{21}=0, \\
\text{(J1)} &\quad
R_{21}=0, &
\text{(J2)} &\quad
Q_{12}=0, \\
\end{align*}
where $\eta_1\in\cK_1$ and $\eta_2\in\cK_2$ are arbitrary.

%
\begin{step}\label{lambdazero}
If $\lambda=0$, then $R_{ik}=0$ and $Q_{ik}=0$ for every $i,k=1,2$.
\end{step}
\begin{proof}
Let $\eta=\eta_0+\alpha e$ for some $\eta_0\in\cK_1\oplus\cK_2$, $\alpha\in\bC$. Positivity of the matrix
\beq
\psi(\eta\eta^*)=
\left(\ba{ccc}
\Psi_{11}(\eta_0,\eta_0) & \Psi_{12}(\eta_0,\eta_0) & \ov{\alpha}R_1\eta_0+\alpha Q_1\ov{\eta_0} \\
\Psi_{21}(\eta_0,\eta_0) & \Psi_{22}(\eta_0,\eta_0) & \ov{\alpha}R_2\eta_0 +\alpha Q_2\ov{\eta_0} \\
\alpha (R_1\eta_0)^*+\ov{\alpha}(Q_1\ov{\eta_0})^* & \alpha (R_2\eta_0)^* +\ov{\alpha}(Q_2\ov{\eta_0})^* & 0
\ea\right)
\eeq
(c.f. \eqref{gen}) implies $\ov{\alpha}R_i\eta_0+\alpha Q_i\ov{\eta_0}=0$ for $i=1,2$. Since it holds for any $\alpha$, we conclude that $R_i\eta_0=0$ and $Q_i\ov{\eta_0}=0$. As $\eta_0$ is arbitrary, the statement follows.
\end{proof}

%
\begin{step}\label{diag}
For $\eta_1\in \cK_1$ and $\eta_2\in\cK_2$,
\beq\label{Psi1111i2222}
\Psi_{11}(\eta_1,\eta_1)=\lambda \eta_1\eta_1^*,\quad\qquad
\Psi_{22}(\eta_2,\eta_2)=\lambda \ov{\eta_2}\ov{\eta_2}^*.
\eeq
Moreover,
\beq
\label{R11Q22}
R_{11}=\lambda \jed_{B(\cK_1)}, \qquad\qquad\qquad\qquad
Q_{22}=\lambda \jed_{B(\cK_2)}.
\eeq
and
\beq
\label{Q11R22}
Q_{11}=0, \qquad\qquad\qquad\qquad\qquad
R_{22}=0.
\eeq
\end{step}
\begin{proof}
Let $\eta_1\in\cK_1$ and $\eta_1\neq 0$. It follows from \eqref{0n0} that $(\eta_1, y_1)\in\cZ_\phi$ for any $y_1\in\cK_1\cap\eta_1^\perp$.
Thus \eqref{basic} leads to $\la y_1,\Psi_{11}(\eta_1,\eta_1) y_1\ra=0$.
Since $\Psi_{11}(\eta_1,\eta_1)$ is a positive operator, its restriction to $\cK_1\cap\eta_1^\perp$ is zero. Therefore it is a nonnegative multiple of $\eta_1\eta_1^*$, say $\Psi_{11}(\eta_1,\eta_1)=\mu\xi_1\xi_1^*$, $\mu\geq 0$.
Now, for $\alpha\in\bC$, let $\eta=\eta_1+\alpha e$ and $ y=-\ov{\alpha}\eta_1+\rho+\Vert \xi_1\Vert^2e$, where $\rho\in\cK_1\cap\eta_1^\perp$. Then $(\eta, y)\in\cZ_\phi$ (c.f. \eqref{nn0}). Observe that,
\beq\label{genmu}
\psi(\eta\eta^*)=
\left(\ba{ccc}
\mu\eta_1\eta_1^* & \Psi_{12}(\eta_1,\eta_1) & \ov{\alpha}R_{11}\eta_1+\alpha Q_{11}\ov{\eta_1} \\
\Psi_{21}(\eta_1,\eta_1) & \Psi_{22}(\eta_1,\eta_1) & \ov{\alpha}R_{21}\eta_1+\alpha Q_{21}\ov{\eta_1} \\
\alpha(R_{11}\eta_1)^*+\ov{\alpha}(Q_{11}\ov{\eta_1})^* & \alpha(R_{21}\eta_1)^*+\ov{\alpha}(Q_{21}\ov{\eta_1})^* & \lambda|\alpha|^2
\ea\right)
\eeq
and
\beg
\lefteqn{\la y,\psi(\eta\eta^*)y \ra=}\nonumber\\
&=&
(\lambda+\mu)|\alpha|^2\Vert \eta_1\Vert^4 -2|\alpha|^2\Vert \eta_1\Vert^2\Re\la \eta_1,R_{11}\eta_1\ra \label{ui}\\
&&{}-2\Vert \eta_1\Vert^2\Re\,\alpha^2\la \eta_1,Q_{11}\ov{\eta_1}\ra\label{uj}\\
&&{}+2\Vert \eta_1\Vert^2\Re\,\alpha(\la\rho,Q_{11}\ov{\eta_1}\ra+\la R_{11}\eta_1,\rho\ra).\label{uk}
\eeg
The above is zero for any $\alpha\in\bC$. The expression in line \eqref{ui} is independent on the phase of $\alpha$, so the sum of lines \eqref{uj} and \eqref{uk} must be independent on the phase of $\alpha$ too. It is possible only if the following conditions simultaneously hold
\beq\label{hg}
\Re\la \eta_1,R_{11}\eta_1\ra=\frac{1}{2}(\lambda+\mu)\Vert \eta_1\Vert^2,
\eeq
\beq\label{yu}
\la \eta_1,Q_{11}\ov{\eta_1}\ra=0,
\eeq
\beq\label{df}
\la\rho,Q_{11}\ov{\eta_1}\ra+\la R_{11}\eta_1,\rho\ra=0.
\eeq
Since the last equality holds for any $\rho\in\cK_1\cap\eta_1^\perp$, one can replace $\rho$ by $i\rho$. So,
\beq
\label{dg}
-i\la\rho,Q_{11}\ov{\eta_1}\ra+i\la R_{11}\eta_1,\rho\ra=0.
\eeq
Combining \eqref{df} and \eqref{dg} yields $\la\rho,Q_{11}\ov{\eta_1}\ra=0$ and $\la\rho,R_{11}\eta_1\ra=0$ for any $\rho \in\cK_1\cap\eta_1^\perp$. Thus both $R_{11}\eta_1$ and $Q_{11}\ov{\eta_1}$ are multiples of $\eta_1$. Then, it follows from \eqref{yu} that $Q_{11}\ov{\eta_1}=0$ and the first condition in \eqref{Q11R22} is proved.

Now, apply Proposition \ref{blockpos} for the matrix \eqref{genmu} and vectors $y_1=\eta_1$, $y_2=0$ and $y_3=e$. It follows that the scalar matrix
\beq
\left(\ba{ccc}
\mu\Vert \eta_1\Vert^4 & 0 & \ov{\alpha}\la \eta_1,R_{11}\eta_1\ra \\
0 & 0 & 0 \\
\alpha\la R_{11}\eta_1,A_1\eta_1\ra & 0 & \lambda|\alpha|^2
\ea\right)\eeq
is positive. Thus
\beq\label{wyzn}
|\la \eta_1,R_{11}\eta_1\ra|^2\leq \lambda\mu\Vert \eta_1\Vert^4,
\eeq
and from \eqref{hg} we conclude
$\frac{1}{4}(\lambda+\mu)^2\Vert \eta_1\Vert^4=(\Re\la \eta_1,R_{11}\eta_1\ra)^2\leq |\la \eta_1,R_{11}\eta_1\ra|^2\leq\lambda\mu\Vert \eta_1\Vert^4$.
The inequality $\frac{1}{4}(\lambda+\mu)^2\leq\lambda\mu$ implies $\mu=\lambda$, so \eqref{Psi1111i2222} is proved.

It remains to show that $R_{11}\eta_1=\lambda \eta_1$. We showed already that $R_{11}\eta_1=\kappa \eta_1$ for some $\kappa\in\bC$. Now, \eqref{hg} implies $\Re\,\kappa=\lambda$ while \eqref{wyzn} yields $|\kappa|\leq \lambda$. Hence $\kappa=\lambda$ and the proof of the first equality in \eqref{R11Q22} is finished.

The proof of the second parts in \eqref{Psi1111i2222}, \eqref{R11Q22}, \eqref{Q11R22} is similar.
\end{proof}

\begin{step}\label{diag12}
It follows that
\beq\label{Q21R12}
Q_{21}=0
\quad\qquad\qquad\quad
R_{12}=0.
\eeq
Moreover,
\beq\label{Psi1211}
\Psi_{12}(\eta_1,\eta_1)=\eta_1(R_{21}\eta_1)^*
\quad\quad\quad
\Psi_{12}(\eta_2,\eta_2)=(Q_{12}\ov{\eta_2})\ov{\eta_2}^*,
\eeq
for any $\eta_1\in\cK_1$ and $\eta_2\in\cK_2$.
\end{step}
\begin{proof}
Consider $\eta=\eta_1+\alpha e$, $\eta_1\in\cK_1$. According to previous observations, the matrix \eqref{genmu} is of the form
\beq\label{gby}
\psi(\eta\eta^*)=
\left(\ba{ccc}
\lambda\eta_1\eta_1^* & \Psi_{12}(\eta_1,\eta_1) & \ov{\alpha}\lambda \eta_1 \\
\Psi_{12}(\eta_1,\eta_1)^* & \Psi_{22}(\eta_1,\eta_1) & \ov{\alpha}R_{21}\eta_1+\alpha Q_{21}\ov{\eta_1} \\
\alpha\lambda \eta_1^* & \alpha(R_{21}\eta_1)^*+\ov{\alpha} (Q_{21}\ov{\eta_1})^* & \lambda|\alpha|^2
\ea\right)
\eeq
For $\lambda=0$, it reduces to
\beq
\left(\ba{ccc}
0 & \Psi_{12}(\eta_1,\eta_1) & 0 \\
\Psi_{12}(\eta_1,\eta_1)^* & \Psi_{22}(\eta_1,\eta_1) & \ov{\alpha}R_{21}\eta_1+\alpha Q_{21}\ov{\eta_1} \\
0 & \alpha(R_{21}\eta_1)^*+\ov{\alpha} (Q_{21}\ov{\eta_1})^* & 0
\ea\right).
\eeq
It is a positive block-matrix, so its upper-left $2\times 2$ principal minor is also positive. Since one of its diagonal terms is zero, its off-diagonal ones should vanish. Hence $\Ps{1}{2}{1}{1}=0$. Let us remind that $R_{21}=0$ due to Observation \ref{lambdazero}, so the first part of \eqref{Psi1211} is satisfied.

If $\lambda>0$, then we apply Lemma \ref{blockpos} to derive that for any $y_1\in\cK_1$ and $y_2\in\cK_2$, the following scalar matrix is positive definite.
$$
\left(\ba{ccc}
\lambda|\la y_1,\eta_1\ra|^2 & \la y_1,\Psi_{12}(\eta_1,\eta_1)y_2\ra & \ov{\alpha}\lambda\la y_1,\eta_1\ra \\
\la\Psi_{12}(\eta_1,\eta_1)y_2,y_1\ra & \la y_2,\Psi_{22}(\eta_1,\eta_1)y_2\ra & \ov{\alpha}\la y_2,R_{21}\eta_1\ra+\alpha\la y_2,Q_{21}\ov{\eta_1}\ra \\
\alpha\lambda\la \eta_1,y_1\ra & \alpha\la R_{21}\eta_1,y_2\ra+\ov{\alpha}\la Q_{21}\ov{\eta_1},y_2\ra & \lambda|\alpha|^2
\ea\right)
$$
Straight calculation shows that its determinant is equal to
$$
-\lambda \left|\ov{\alpha}\la y_1,\eta_1\ra\la Q_{21}\ov{\eta_1},y_2\ra +\alpha\left(\la y_1,\eta_1\ra\la R_{21}\eta_1, y_2\ra-\la y_1,\Psi_{12}(\eta_1,\eta_1)y_2\ra\right)\right|^2
$$
Since it is nonnegative for any $\alpha\in\bC$ and every vectors $y_i\in\cK_i$, $i=1,2$, $\Psi_{12}(\eta_1,\eta_1)=\eta_1(R_{21}\eta_1)^*$ and $\eta_1(Q_{21}\ov{\eta_1})^*=0$. The latter holds for any $\eta_1\in\cK_1$, hence $Q_{21}=0$. Thus we proved first parties of \eqref{Psi1211} and \eqref{Q21R12}.

The remaining parts are proved similarly by considering the matrix $\psi(\eta\eta^*)$ for $\eta=\eta_2+\alpha e$ where $\eta_2\in\cK_2$.
\end{proof}

Before next observations, let us study some further consequences of the condition \eqref{basic}.
Let $\eta=\eta_1+\eta_2+\alpha e$, $ y= y_1+ y_2+\beta e$, $\eta_i, y_i\in\cK_i$, $i=1,2$, be such that $\eta_1\neq 0$, $\eta_2\neq 0$, $\alpha\neq 0$, and $ y_1=-\ov{\alpha}\eta_1$, $ y_2=-\alpha \ov{\eta_2}$, $\beta=\Vert \eta_1\Vert^2+\Vert {\eta_2}\Vert^2$. It follows from \eqref{nnn}  that $(\eta, y)\in\cZ_\phi$. By a sequence of elementary calculations one can check that
\beq\label{jpl}
\la y,\psi(\eta\eta^*) y\ra = |\alpha|^2 c_1 + 2\Re\,\alpha^2c_2
,
\eeq
where
\beg
c_1 &=&  \la \eta_1,\Psi_{11}(\eta_2,\eta_2)\eta_1\ra+ \la \ov{\eta_2},\Psi_{22}(\eta_1,\eta_1)\ov{\eta_2}\ra - 2\lambda\Vert \eta_1\Vert^2 \Vert {\eta_2} \Vert^2   \\
&&\nonumber
{}+ \la \eta_1,\Psi_{11}(\eta_1,\eta_2)\eta_1\ra + \la \eta_1,\Psi_{11}(\eta_2,\eta_1)\eta_1\ra  
{}+  \la \ov{\eta_2},\Psi_{22}(\eta_1,\eta_2)\ov{\eta_2}\ra + \la \ov{\eta_2},\Psi_{22}(\eta_2,\eta_1)\ov{\eta_2}\ra \\[3pt]
c_2&=&
\la \eta_1,\Psi_{12}(\eta_1,\eta_2)\ov{\eta_2}\ra + \la \eta_1,\Psi_{12}(\eta_2,\eta_1)\ov{\eta_2}\ra 
{}- \Vert \eta_1\Vert^2\la \eta_1, Q_{12}\ov{\eta_2}\ra - \Vert {\eta_2}\Vert^2\la R_{21}\eta_1, \ov{\eta_2}\ra
\eeg
Since \label{strona} the expression \eqref{jpl} is zero for every $\alpha$, $c_p=0$ for $p=1,2$. Note, that if we replace $\eta_1,\eta_2$ by $\gamma_1\eta_1,\gamma_2\eta_2$, where $\gamma_1,\gamma_2\in\bC$ are arbitrary, then the whole expressions are still equal to zero. For instance, for 'modified' $c_1$ we obtained the following equality
\be
0&=&
|\gamma_1|^2|\gamma_2|^2\Big(\la \eta_1,\Psi_{11}(\eta_2,\eta_2)\eta_1\ra+ \la \ov{\eta_2},\Psi_{22}(\eta_1,\eta_1)\ov{\eta_2}\ra - 2\lambda\Vert \eta_1\Vert^2 \Vert \ov{\eta_2} \Vert^2\Big)   \\
&&
{}+ \ov{\gamma_1}|\gamma_1|^2\gamma_2\la \eta_1,\Psi_{11}(\eta_1,\eta_2)\eta_1\ra + \gamma_1|\gamma_1|^2\ov{\gamma_2}\la \eta_1,\Psi_{11}(\eta_2,\eta_1)\eta_1\ra  \\
&&
{}+ \ov{\gamma_1}\gamma_2|\gamma_2|^2\la \ov{\eta_2},\Psi_{22}(\eta_1,\eta_2)\ov{\eta_2}\ra + \gamma_1\ov{\gamma_2}|\gamma_2|^2\la \ov{\eta_2},\Psi_{22}(\eta_2,\eta_1)\ov{\eta_2}\ra
\ee
As the above equality holds for any $\gamma_1,\gamma_2$, all coefficients of this complex polynomial in the variables $\gamma_1,\gamma_2$ should vanish, i.e.
\beq
\label{suma}
\la \eta_1,\Psi_{11}(\eta_2,\eta_2)\eta_1\ra+ \la \ov{\eta_2},\Psi_{22}(\eta_1,\eta_1)\ov{\eta_2}\ra = 2\lambda\Vert \eta_1\Vert^2 \Vert {\eta_2} \Vert^2
\eeq
\beq
\label{111121}
\la \eta_1,\Psi_{11}(\eta_1,\eta_2)\eta_1\ra = \la \eta_1,\Psi_{11}(\eta_2,\eta_1)\eta_1\ra=0
\eeq
\beq
\label{222122}
\la \ov{\eta_2},\Psi_{22}(\eta_1,\eta_2)\ov{\eta_2}\ra = \la \ov{\eta_2},\Psi_{22}(\eta_2,\eta_1)\ov{\eta_2}\ra =0
\eeq
By the similar arguments for $c_2$, we obtain
\beq
\label{112122}
\la \eta_1,\Psi_{12}(\eta_1,\eta_2)\ov{\eta_2}\ra =0
\eeq
\beq
\label{112212}
\la \eta_1,\Psi_{12}(\eta_2,\eta_1)\ov{\eta_2}\ra =0
\eeq
\beq
\label{Q12xi1}
\la \eta_1, Q_{12}\ov{\eta_2}\ra =0
\eeq
\beq
\label{R21xi2}
\la R_{21}\eta_1, \ov{\eta_2}\ra = 0
\eeq
%
%
\begin{step}\label{RiQ}
It follows that
\beq\label{R21Q12}
R_{21}=0,
\qquad\qquad\qquad\qquad\qquad\qquad
Q_{12}=0
\eeq
and consequently
\beq\label{Ps1211}
\Ps{1}{2}{1}{1}=0,\qquad\qquad\qquad\qquad\Ps{1}{2}{2}{2}=0.
\eeq
\end{step}
\begin{proof}
Let us consider equation \eqref{R21xi2}.
Since it holds for any $\eta_1\in\cK_1$, $\eta_2\in\cK_2$, 
$R_{21}=0$. Similarly, \eqref{Q12xi1} 
implies $Q_{12}=0$. Thus \eqref{R21Q12} is proved. Now, take into account \eqref{Psi1211}, and \eqref{Ps1211} follows.
\end{proof}
%
\begin{step}\label{diagtrzy}
For any $\eta_1\in\cK_1$ and $\eta_2\in\cK_2$
\beq\label{P11122212}
\Ps{1}{1}{1}{2}=0=\Ps{1}{1}{2}{1},\qquad\qquad\Ps{2}{2}{1}{2}=0=\Ps{2}{2}{2}{1}.
\eeq
\end{step}
\begin{proof}
Let $\eta=\eta_1+\eta_2+e$. Then due to Observations \ref{diag}, \ref{diag12}, \ref{RiQ}
\beq
\psi(\eta\eta^*)=
\left(\ba{c:c:c}
\begin{aligned}&\lambda\eta_1\eta_1^*+\Ps{1}{1}{2}{2}\\&{}+\Ps{1}{1}{1}{2}+\Ps{1}{1}{2}{1}\end{aligned}
& \Ps{1}{2}{1}{2}+\Ps{1}{2}{2}{1} & \lambda \eta_1 \\[3mm]\hdashline\\[-3mm]
\Ps{2}{1}{1}{2}+\Ps{2}{1}{2}{1} & \begin{aligned}&\Ps{2}{2}{1}{1} +\lambda\ov{\eta_2}\ov{\eta_2}^*\\&{} +\Ps{2}{2}{1}{2} +\Ps{2}{2}{2}{1}\end{aligned} & \lambda \ov{\eta_2} \\[3mm]\hdashline\\[-3mm]
\lambda \eta_1^* & \lambda \ov{\eta_2}^* & \lambda
\ea\right)
\eeq
Apply Proposition \ref{blockpos} for $y_1\in\cK_1$ arbitrary, $y_2=0$ and $y_3=e$. It follows that the following scalar matrix is positive
\beq\label{mac}
\left(\ba{ccc}
\lambda|\la y_1,\Aj\ra|^2 + 2\Re\la y_1,\Ps{1}{1}{1}{2}y_1\ra + \la y_1,\Ps{1}{1}{2}{2}y_1\ra & 0 & \lambda\la y_1,\Aj\ra \\
0 & 0 & 0 \\
\lambda\la\Aj,y_1\ra & 0 & \lambda
\ea\right)
\eeq
If $\lambda=0$, then positivity of the matrix is equivalent to the inequality
\beq\label{lazero}
2\Re\la y_1,\Ps{1}{1}{1}{2} y_1\ra + \la y_1,\Ps{1}{1}{2}{2} y_1\ra\geq 0.
\eeq
If we replace $\eta_1$ by $te^{i\theta}\eta_1$, where $t\in\bR$, and $\theta$ is such that $e^{i\theta}\la y_1,\Ps{1}{1}{1}{2}y_1\ra=|\la y_1,\Ps{1}{1}{1}{2}y_1\ra|$, the inequality still holds. Thus
$$
2t|\la y_1,\Ps{1}{1}{1}{2}y_1\ra|+\la y_1,\Ps{1}{1}{2}{2}y_1\ra\geq 0
$$
for any $t\in\bR$ and, consequently $\la y_1,\Ps{1}{1}{1}{2}y_1\ra=0$. Since $y_1$ is arbitrary, the first part of \eqref{P11122212} follows.

Now, let $\lambda>0$. Since the matrix \eqref{mac} is positive, its $2\times 2$ minors are nonnegative. Hence, in particular,
$$
2\lambda\Re\la y_1,\Ps{1}{1}{1}{2} y_1\ra + \lambda\la y_1,\Ps{1}{1}{2}{2}y_1\ra\geq 0.
$$
But this condition is equivalent to \eqref{lazero}, so again it leads to the first part of \eqref{P11122212}.

The second part of \eqref{P11122212} is proved analogously by considering $y_1=0$ and $y_2$ arbitrary.
\end{proof}

\begin{step}\label{diagdwa}
For any $\eta_1\in\cK_1$ and $\eta_2\in\cK_2$,
\beq\label{P11222211}
\Psi_{11}(\eta_2,\eta_2)=\lambda\Vert {\eta_2}\Vert^2 \jed_{B(\cK_1)},\qquad\quad\qquad
\Psi_{22}(\eta_1,\eta_1)=\lambda\Vert \eta_1\Vert^2 \jed_{B(\cK_2)}.
\eeq
\end{step}
\begin{proof}
Let $\eta_1\in\cK_1$, $\eta_2\in\cK_2$ be arbitrary, $\alpha=1$,
and $\eta=\eta_1+\eta_2+e$. Due to previous observations 
\beq
\psi(\eta\eta^*)=
\left(\ba{ccc}
\lambda\eta_1\eta_1^*+\Ps{1}{1}{2}{2} & \Ps{1}{2}{1}{2}+\Ps{1}{2}{2}{1} & \lambda \eta_1 \\
\Ps{2}{1}{1}{2}+\Ps{2}{1}{2}{1} & \Ps{2}{2}{1}{1} +\lambda\ov{\eta_2}\ov{\eta_2}^* & \lambda \ov{\eta_2} \\
\lambda \eta_1^* & \lambda \ov{\eta_2}^* & \lambda
\ea\right)
\eeq
We apply Proposition \ref{blockpos} for $y_1=\eta_1$, $y_2=\ov{\eta_2}$ and $\beta=1$. If one take into account equalities 
\eqref{112122}, \eqref{112212} then it is clear that the scalar matrix $(\la y_i,\psi_{ij}(\eta\eta^*)y_j\ra)_{i,j}$ reduces to
\beq\label{redu}
\left(\ba{ccc}
\lambda\Vert \eta_1\Vert^4+\la \eta_1,\Ps{1}{1}{2}{2} \eta_1\ra & 0 & \lambda\Vert \eta_1\Vert^2 \\
0 & \la \ov{\eta_2},\Ps{2}{2}{1}{1}\ov{\eta_2}\ra+\lambda\Vert {\eta_2}\Vert^4 & \lambda \Vert {\eta_2}\Vert^2 \\
\lambda\Vert \eta_1\Vert^2 & \lambda \Vert {\eta_2}\Vert^2 & \lambda
\ea\right)
\eeq
If $\lambda=0$, then positivity of the above matrix implies that both
$\la \eta_1,\Ps{1}{1}{2}{2} \eta_1\ra$ and 
$\la \ov{\eta_2},\Ps{2}{2}{1}{1}\ov{\eta_2}\ra$ 
should be nonnegative.
Therefore the condition \eqref{suma} implies
$$
\la \eta_1,\Ps{1}{1}{2}{2} \eta_1\ra= \la \ov{\eta_2},\Ps{2}{2}{1}{1}\ov{\eta_2}\ra=0
,$$ and \eqref{P11222211} follows.

For $\lambda>0$ we use the fact that the determinant of the matrix \eqref{redu} is nonnegative. By strightforward calculations we obtain the inequality
$
\la \eta_1,\Ps{1}{1}{2}{2} \eta_1\ra\la \ov{\eta_2},\Ps{2}{2}{1}{1}\ov{\eta_2}\ra\geq \lambda^2\Vert \eta_1\Vert^4\Vert {\eta_2}\Vert^4.
$ 
Combination with \eqref{suma} leads to
$$
\la \eta_1,\Ps{1}{1}{2}{2} \eta_1\ra= \la \ov{\eta_2},\Ps{2}{2}{1}{1}\ov{\eta_2}\ra=\lambda\Vert \eta_1\Vert^2\Vert{\eta_2}\Vert^2.
$$
Since these equalities hold for any $\eta_1$ and $\eta_2$, we arrive at \eqref{P11222211}.
\end{proof}

\begin{step}\label{offdiag}
For $\eta_1\in\cK_1$ and $\eta_2\in\cK_2$,
\beq\label{1212}
\Ps{1}{2}{1}{2}=0,\qquad\qquad\Ps{1}{2}{2}{1}=0.
\eeq
\end{step}
\begin{proof}
If $\lambda=0$, then
$$
\psi(\eta\eta^*)=\left(\ba{ccc} 0 & \Ps{1}{2}{1}{2}+\Ps{1}{2}{2}{1} & 0 \\
\Ps{2}{1}{1}{2}+\Ps{2}{1}{2}{1} & 0 & 0 \\
0 & 0 & 0
\ea\right)
$$
for any $\eta\in\cK$. Hence, positivity of the matrix implies $\Ps{1}{2}{1}{2}+\Ps{1}{2}{2}{1}=0$. Using again the aforementioned substitution argument (cf. paragraph preceeding \eqref{suma} on page \pageref{strona}), we get \eqref{1212}.
Thus, we can assume $\lambda>0$.
Firstly, we will show that
\beq\label{12121}
\Ps{1}{2}{1}{2}^*\eta_1=0,\qquad\qquad\Ps{1}{2}{2}{1}\Ad=0.
\eeq
Let us prove the first equality. Observe that for a given $\eta_1\in\cK_1$, a map
$\cK_2\times\cK_2\ni(\eta_2,\eta_2')\mapsto \la \ov{\eta_2'},\Ps{1}{2}{1}{2}^*\Aj\ra$ is a sesqulinear form. Then, \eqref{112122} and polarization identity (cf. Proposition \ref{sesq}) imply $\la \ov{\eta_2'},\Ps{1}{2}{1}{2}^*\Aj\ra=0$ for every $\eta_2,\eta_2'\in\cK_2$.
Thus, $\Ps{1}{2}{1}{2}^*\Aj=0$ for every $\eta_1\in\cK_1$ and $\eta_2\in\cK_2$.
In order to prove the second equality in \eqref{12121} one need to consider a sesquilinear form $\cK_1\times\cK_1\ni(\eta_1',\eta_1)\mapsto \la \eta_1',\Ps{1}{2}{2}{1}\Ad\ra$.

Secondly, we will prove that
\beq\label{12122}
\Ps{1}{2}{1}{2}\Ad=0,\qquad\qquad\Ps{1}{2}{2}{1}^*\Aj=0.
\eeq
Let $\eta=\eta_1+\eta_2-ie$. Then, according to previous observations,
\beq
\psi(\eta\eta^*)=
\left(\ba{ccc}
\lambda\eta_1\eta_1^*+\lambda \Vert\eta_2\Vert^2\jed_{B(\cK_1)} & \Ps{1}{2}{1}{2}+\Ps{1}{2}{2}{1} & -i\lambda \eta_1 \\[2mm]
\Ps{2}{1}{1}{2}+\Ps{2}{1}{2}{1} & \lambda\ov{\eta_2}\ov{\eta_2}^* +\lambda\Vert\Aj\Vert^2\jed_{B(\cK_2)} & i\lambda \ov{\eta_2} \\[2mm]
i\lambda \eta_1^* & -i\lambda \ov{\eta_2}^* & \lambda
\ea\right).
\eeq
Further, apply Proposition \ref{blockpos} for $y_1=\Vert\eta_2\Vert^2\Aj+\lambda^{-1}\Ps{1}{2}{1}{2}\Ad$, $y_2=\Vert\Aj\Vert^2\Ad$ and $y_3=e$. We infer that the scalar matrix
\beq\label{macierzduza}
\left(\ba{c:c:c}
\begin{aligned}
\lambda \Vert\Aj\Vert^4\Vert\eta_2\Vert^4+\lambda\Vert\Aj\Vert^2\Vert\eta_2\Vert^6 \\
+\lambda^{-1}\Vert\eta_2\Vert^2\Vert\Ps{1}{2}{1}{2}\Ad\Vert^2
\end{aligned}
&
\lambda^{-1}\Vert\Aj\Vert^2\Vert\Ps{1}{2}{1}{2}\Ad\Vert^2 &
-i\lambda\Vert\Aj\Vert^2\Vert\eta_2\Vert^2 \\[2mm]\hdashline\\[-2mm]
\lambda^{-1}\Vert\Aj\Vert^2\Vert\Ps{1}{2}{1}{2}\Ad\Vert^2 &
\lambda \Vert\Aj\Vert^4\Vert\eta_2\Vert^4 
+ \lambda \Vert\Aj\Vert^6\Vert\eta_2\Vert^2
&
i\lambda\Vert\Aj\Vert^2\Vert\eta_2\Vert^2 \\[2mm]\hdashline\\[-2mm]
i\lambda\Vert\Aj\Vert^2\Vert\eta_2\Vert^2 &
-i\lambda\Vert\Aj\Vert^2\Vert\eta_2\Vert^2 & \lambda
\ea\right)
\eeq
is positive.
Let us verify formulas for coefficients of the above matrix.
Due to \eqref{112122}, $\Vert y_1\Vert^2=\Vert\Aj\Vert^2\Vert\eta_2\Vert^4+\lambda^{-2}\Vert\Ps{1}{2}{1}{2}\Ad\Vert^2$ and $\la\Aj,y_1\ra=\Vert\Aj\Vert^2\Vert\eta_2\Vert^2$. Thus the $(1,1)$-coefficient is equal to
\be
\lefteqn{\la y_1,(\lambda\eta_1\eta_1^*+\lambda \Vert\eta_2\Vert^2\jed_{B(\cK_1)})y_1\ra=}\\
&=&\lambda|\la\Aj,y_1\ra|^2+\lambda\Vert\eta_2\Vert^2\Vert y_1\Vert^2\\
&=&
\lambda\Vert\Aj\Vert^4\Vert\eta_2\Vert^4+\lambda\Vert\eta_2\Vert^2(\Vert\Aj\Vert^2 \Vert\eta_2\Vert^4+\lambda^{-2}\Vert\Ps{1}{2}{1}{2}\Ad\Vert^2)\\
&=&
\lambda\Vert\Aj\Vert^4\Vert\eta_2\Vert^4+\lambda\Vert\Aj\Vert^2\Vert\eta_2\Vert^6 +\lambda^{-1}\Vert\eta_2\Vert^2\Vert\Ps{1}{2}{1}{2}\Ad\Vert^2.
\ee
$(1,2)$-coefficient is calculated as follows
\beg
\lefteqn{\la y_1,\Ps{1}{2}{1}{2}y_2\ra+\la y_1,\Ps{1}{2}{2}{1})y_2\ra=}\nonumber\\
&=&
\Vert\Aj\Vert^2\Vert\eta_2\Vert^2\la\Aj,\Ps{1}{2}{1}{2}\Ad\ra + \lambda^{-1}\Vert\Aj\Vert^2\Vert\Ps{1}{2}{1}{2}\Ad\Vert^2 \label{fiu1} \\
&&{}
+ \Vert\Aj\Vert^2\la y_1,\Ps{1}{2}{2}{1}\Ad\ra \label{fiu2}
\\
&=&\lambda^{-1}\Vert\Aj\Vert^2\Vert\Ps{1}{2}{1}{2}\Ad\Vert^2\nonumber
\eeg
The first term in the line \eqref{fiu1} is zero due to \eqref{112122}, while the term in the line \eqref{fiu2} is zero because of \eqref{12121}. The form of remaining coefficients can be verified directly. Determinant of the matrix \eqref{macierzduza} is equal to
$$-\lambda\Vert\Aj\Vert^6\Vert\eta_2\Vert^4\Vert\Ps{1}{2}{1}{2}\Ad\Vert^2-\lambda^{-1}\Vert\Aj\Vert^4\Vert\Ps{1}{2}{1}{2}\Ad\Vert^4.$$
Since it must be nonnegative, $\Ps{1}{2}{1}{2}\Ad=0$. Thus the first part of \eqref{12122} is verified. The second part can be proved similarly. One should consider now
$y_1=\Vert\eta_2\Vert^2\Aj$, $y_2=\Vert\Aj\Vert^2\Ad+\lambda^{-1}\Ps{1}{2}{2}{1}^*\Aj$, $y_3=e$.

After all, we are ready to prove \eqref{1212}. For the first part, fix $\eta_1\in\cK_1$ and consider a sesqulinear vector valued form
$\cK_2\times\cK_2\ni(\eta_2',\eta_2)\mapsto\Ps{1}{2}{1}{2}\ov{\eta_2'}\in\cK_1$. According to Proposition \ref{sesq}, \eqref{12122} implies $\Ps{1}{2}{1}{2}\ov{\eta_2'}=0$
for every $\eta_2,\eta_2'\in\cK_2$. Hence, the first part of \eqref{1212} follows.
To show the second part one need to consider a sesquilinear vector valued form
$\cK_1\times\cK_1\ni(\eta_1,\eta_1')\mapsto\Ps{1}{2}{2}{1}^*\eta_1'\in\cK_1$ and use \eqref{12121} to argue that it is zero for every $\eta_1,\eta_1'$ and $\eta_2$.
\end{proof}

Now, let us come back to the conditions listed on page \pageref{tabela}. Below is the table which shows where each condition is proved.
\begin{center}
\begin{tabular}{lclcl}
(A1) & \& & (A2) & -- & eq. \eqref{Psi1111i2222} in Observation \ref{diag} \\
(B1) & \& & (B2) & -- & eq. \eqref{P11222211} in Observation \ref{diagdwa} \\
(C1) & \& & (C2) & -- & eq. \eqref{P11122212} in Observation \ref{diagtrzy} \\
(D1) & \& & (D2) & -- & eq. \eqref{Ps1211} in Observation \ref{RiQ} \\
(E1) & \& & (E2) & -- & eq. \eqref{1212} in Observation \ref{offdiag} \\
(F1) & \& & (F2) & -- & eq. \eqref{R11Q22} in Observation \ref{diag} \\
(G1) & \& & (G2) & -- & eq. \eqref{Q11R22} in Observation \ref{diag} \\
(H1) & \& & (H2) & -- & eq. \eqref{Q21R12} in Observation \ref{diag12} \\
(J1) & \& & (J2) & -- & eq. \eqref{R21Q12} in Observation \ref{RiQ}
\end{tabular}
\end{center}
Therefore $\psi=\lambda\phi_{\cK_1,\cK_2}$, and
the proof of Theorem \eqref{special} is finished.
\end{proof}

\begin{proof}[Proof of Theorem \ref{mainm}]
For a map $\phi$ given by \eqref{block} define $\ti{\cK}_1=\ran A_1\subset \cH_1$ and $\ti{\cK}_2=\ran \ov{A_2}\subset\cH_2$, where $\ov{A}=(A^*)^\rT$ for a linear operator $A$. 
Let $F_i:\ti{\cK}_i\to\cH_i$ be the isometric embedding of $\ti{\cK}_i$ into $\cH_i$ for $i=1,2$. Further, define $E_i:\ti{\cK}_i\to\cK_i$, $i=1,2$, by
$E_1=A_1^*F_1$ and $E_2=A_2^\rT{F_2}=(\ov{A_2})^*{F_2}$. Eventually, let $E:\ti{\cK}\to\cK$ and $F:\ti{\cK}\to\cH$ be linear operators defined by the following block-matrices
\beq
\label{EF}
E=\left(\ba{ccc}
\nu^{-1/4}E_1 & 0 & 0 \\ 0 & \nu^{1/4}E_2 & 0 \\ 0 & 0 & \id_\bC
\ea\right),
\qquad
F=\left(\ba{ccc}
\nu^{1/4}\e{1}F_1 & 0 & 0 \\ 0 & \nu^{-1/4}\e{2}\ov{F_2} & 0 \\ 0 & 0 & \id_\bC
\ea\right).
\eeq
Since both $F_1$ and $F_2$ are injective by the definition, $F$ is injective too. Moreover, both $F_1^*A_1^{}$ and $F_2^*\ov{A_2^{}}$ are surjective operators, so $E_1$ and $E_2$ being their hermitian conjugations, are injective operators. Hence, $E$ is also injective.
One can check that
\be
\lefteqn{F\phi_{\ti{\cK}_1,\ti{\cK}_2}(E^*XE)F^*=}\\
&=&
\left(\ba{c:c:c}
{
\begin{aligned}
&F_1^{}F_1^*A_1^{}X_{11}^{}A_1^*F_1^{}F_1^*\\
&{}+\nu\Tr(F_2^\rT A_2^{}X_{22}^\rT A_2^*\ov{F_2^{}})F_1^{}F_1^*
\end{aligned}
} & 0 & \e{1}F_1^{}F_1^*A_1X_{31} \\[4mm]\hdashline \\[-3mm]
0 &
\begin{aligned}
&\ov{F_2}F_1^\rT A_2^{}X_{22}^{\rT}A_2^*\ov{F_2}F_2^\rT \\
&{}+\nu^{-1}\Tr(F_2^\rT A_2^{}X_{22}^\rT A_2^*\ov{F_2})\ov{F_2^{}}F_2^\rT \\
\end{aligned}
& \e{2}\ov{F_2}F_2^\rT A_2X_{32}^\rT \\[4mm] \hdashline \\[-3mm]
\me{1}X_{31}A_1^*F_1^{}F_1^* &  \me{2} X_{23}^\rT A_2 \ov{F_2}F_2^\rT & X_{33}
\ea\right)
\ee
for $X\in B(\cK)$, where $\phi_{\ti{\cK}_1,\ti{\cK}_2}$ is given by \eqref{MOhd}. It follows from the definition of $F_1$ and $F_2$ that $F_1^{}F_1^*A_1^{}=A_1^{}$, and $\ov{F_2^{}}F_2^\rT A_2=\ov{F_2^{}F_2^*\ov{A_2^{}}}=\ov{\ov{A_2}}=A_2$. Consequently, $F_1F_1^*=P_1$ and
$\ov{F_2}F_2^\rT=\ov{F_2F_2^*}=\ov{P_{\ran\ov{A_2}}}=P_{\ran A_2}^{}=P_2^{}$.
Thus, we arrive at the equality $F\phi_{\ti{\cK}_1,\ti{\cK}_2}(E^*XE)F^*=\phi(X)$. Applying Lemma \ref{lift} and Theorem \ref{special}, we infer that $\phi$ is exposed.
\end{proof}
\subsection{Examples of optimal positive maps}
Let $\cK_1$, $\cK_2$ be finite dimensional Hilbert spaces and $\cK=\cK_1\oplus\cK_2\oplus\bC$.
In \cite{RSC} Sarbicki, Chru{\'s}ci{\'n}ski and one of the authors described another generalization of the map of Miller and Olkiewicz. We proposed further generalization to the map $\Omega_{\cK_1,\cK_2}$ given by \eqref{RSCgen2} (cf. Example \ref{e:RSCgen2}). It was shown in Corollary \ref{c:examp} that the map $\Lambda_{\cK_1,\cK_2}$ is positive. Here we describe its further properties.

Firstly, observe that
\beq
\label{Omegadec}
\Omega_{\cK_1,\cK_2}(X)=\phi_{\cK_1,\cK_2}(X)+
\left(\ba{ccc}
\Tr(X_{11})\jed_{B(\cK_1)}-X_{11} & 0 & 0 \\
0 & 0 & 0 \\
0 & 0 & 0
\ea\right)
\eeq
and
\beq
\label{Lambdadec}
\Lambda_{\cK_1,\cK_2}(X)=\Omega_{\cK_1,\cK_2}(X)+
\left(\ba{ccc}
0 & 0 & 0 \\
0 & \Tr(X_{22})\jed_{B(\cK_1)}-X_{22}^\rT & 0 \\
0 & 0 & 0
\ea\right)
\eeq
for $X=(X_{ij})\in B(\cK)$. Since $\psi_1(X)=\Tr(X)\jed -X$ and $\psi_2(X)=\Tr(X)\jed - X^\rT$ are positive maps (cf. \eqref{psig}), the maps $\Omega_{\cK_1,\cK_2}$ and $\Lambda_{\cK_1,\cK_2}$ are no longer extremal. However, we will show that $\Omega_{\cK_1,\cK_2}$ is still an example of an optimal map.
Let us 
Let us start with the following
\begin{prop}
\label{RSCgenstrongop}
The map $\Omega_{\cK_1,\cK_2}$ satisfies the spanning property.
\end{prop}
\begin{proof}
We will show that $\Omega_{\cK_1,\cK_2}$ satisfies condition \eqref{spanpro}.
Let $\cZ_{\Omega}$ denote the set of all pairs such that $\la y,\Omega_{\cK_1,\cK_2}(\eta\eta^*)y\ra=0$. As previously, we assume that $\eta=\eta_1+\eta_2+\alpha \eps$, $y=y_1+y_2+\beta\eps$, where $\eta_i,y_i\in\cK_i$, $i=1,2$, and $\alpha,\beta\in\bC$. It follows from \eqref{Omegadec} that $(\eta,y)\in\cZ_\Omega$ if and only if $(\eta,y)\in\cZ_\phi$ (cf. page \pageref{stronaZ}), and $y_1,\eta_1$ are linearly dependent. Taking into account the description of $\cZ_\phi$ given by \eqref{0n0}--\eqref{nnn} we infer that $(\eta,y)\in\cZ_\Omega$ if and only if at least one of the following conditions holds
\beg
\alpha=0, & \textrm{and} &  y_1\perp \eta_1,\; y_2=0 \label{0n}\\
\eta_1=0,\; {\eta_2}=0 & \textrm{and} & \beta=0 \label{n0}\\
\alpha\neq 0,\; \eta_1\neq 0,\; {\eta_2}\neq 0 & \textrm{and} &
y_1=-\dfrac{\ov{\alpha}\beta}{\Vert \eta_1\Vert^2+\Vert {\eta_2}\Vert^2} \eta_1, \;
y_2=-\dfrac{\alpha\beta}{\Vert \eta_1\Vert^2+\Vert {\eta_2}\Vert^2}\ov{\eta_2}
\label{nn}
\\
\eta_1=0,\; \eta_2\neq 0 & \textrm{and} &
y_1=0,\; y_2=-\frac{\alpha\beta}{\Vert\eta_2\Vert}\ov{\eta_2}+\rho_2\textrm{ for }\rho_2\in\cK_2\ominus\bC\ov{\eta_2}
\label{extra}
\eeg
Now, assume that a completely positive map $\psi:B(\cK)\to B(\cK)$ satisfies $\la y,\psi(\eta\eta^*)y\ra=0$ for every $(\eta,y)\in\cZ_\Omega$. Since $\cK$ is finite dimensional, $\psi$ can be written in the Kraus form, i.e. $\psi(X)=\sum_{j=1}^mS_jXS_j^*$ for some $S_j\in B(\cK)$. Observe that the condition $\la y,\psi(\eta\eta^*)y\ra=0$ implies $\la y,S_j\eta\ra=0$ for every $j$. If $\eta_1\in\cK_1$, $\eta_2\in\cK_2$, then $(\eta_1+\eta_2,\eps)\in \cZ_\Omega$ (cf. \eqref{0n}), hence $\la \eps,S_j(\eta_1+\eta_2)\ra=0$. Since $\eta_1,\eta_2$ are arbitrary, $S_j(\eta_1+\eta_2)\in\eps^\perp=\cK_1\oplus\cK_2$, i.e. $S_j(\cK_1\oplus\cK_2)\subset\cK_1\oplus\cK_2$. On the other hand, if $y_1\in\cK_1$, $y_2\in\cK_2$, then $(\eps,y_1+y_2)\in\cZ_\Omega$ according to \eqref{n0}. Thus $\la y_1+y_2, S\eps\ra=0$, so $S_j\eps\in (\cK_1\oplus\cK_2)^\perp=\bC\eps$, i.e. $S_j\eps=\mu\eps$ for some $\mu\in\bC$.
Further, assume that $\eta=\eta_1+\eta_2+\alpha\eps$, where $\Vert\eta_1\Vert^2+\Vert\eta_2\Vert^2>0$, $\alpha\neq 0$, and consider $y=\ov{\alpha}\eta_1+\alpha\ov{\eta_2}-(\Vert\eta_1\Vert^2+\Vert\eta_2\Vert^2)\eps$.
Then $(\eta,y)\in\cZ_\Omega$ due to \eqref{nn}. Therefore
$$
\alpha\la\eta_1,S_j\eta_1\ra+ \alpha\la\eta_1,S_j\eta_2\ra+ \ov{\alpha}\la\ov{\eta_2},S_j\eta_1\ra+ \ov{\alpha}\la\ov{\eta_2}S_j\eta_2\ra-\alpha\mu(\Vert\eta_1\Vert^2+\Vert\eta_2\Vert^2)=0.
$$
Since this equality holds for any $\alpha$, we got
\beq
\label{S1}
\la\eta_1,S_j\eta_1\ra+ \la\eta_1,S_j\eta_2\ra-\mu(\Vert\eta_1\Vert^2+\Vert\eta_2\Vert^2)=0,
\eeq
\beq
\label{S2}
\la\ov{\eta_2},S_j\eta_1\ra+ \la\ov{\eta_2},S_j\eta_2\ra=0.
\eeq
The equality \eqref{S1} is satisfied for any $\eta_1,\eta_2$. Assume $\eta_2\neq 0$ and replace $\eta_2$ by $z\eta_2$, $z\in\bC$. Then we arrive at
$$
\la\eta_1,S_j\eta_1\ra-\mu\Vert\eta_1\Vert+ z\la\eta_1,S_j\eta_2\ra-|z|^2\mu\Vert\eta_2\Vert^2=0.
$$
Since it holds for any $z$, $\mu=0$, $\la\eta_1,S_j\eta_1\ra=0$, and $\la\eta_1,S_j\eta_2\ra=0$. Thus, $S_j\eps=0$ and  $S_j(\cK_1\oplus\cK_2)\subset\cK_2$. Similarly, we derive from \eqref{S2} that, in particular, $\la\ov{\eta_2},S_j\eta_1\ra=0$
for any $\eta_1,\eta_2$. It implies $S_j\big|_{\cK_1}=0$.
Now, for $\eta_2\in\cK_2$, $\rho_2\in\cK_2\ominus\bC\ov{\eta_2}$ and $\alpha\in\bC$ consider vectors $\eta=\eta_2+\alpha\eps$, $y=\alpha\ov{\eta_2}+\rho_2-\Vert\eta_2\Vert^2\eps$. It follows from \eqref{extra} that $(\eta,y)\in \cZ_\Omega$. Thus
$
\la\ov{\alpha}\ov{\eta_2}+\rho_2,S_j\eta_2\ra=0.
$
Since it holds for any $\alpha$ and $\rho_2$, we arrive at $S_j\big|_{\cK_2}=0$.

Summing up, $S_j=0$ for every $j$, therefore $\psi=0$.
\end{proof}
\begin{prop}
The map $\Omega_{\cK_1,\cK_2}$ is a nondecomposable and optimal map.
\end{prop}
\begin{proof}
To prove that $\Omega_{\cK_1,\cK_2}$ is a nondecomposable map we show that it is an  entanglement witness for some PPT operator $Z$. Let $\cE_\cK=(\eps_{1,1},\ldots,\eps_{1,k_1},\eps_{2,1},\ldots,\eps_{2,k_2},\eps)$ be an orthonormal basis of $\cK$ composed of real vectors such that $(\eps_{1,1},\ldots,\eps_{1,k_1})$ and $(\eps_{2,1},\ldots,\eps_{2,k_2})$ are bases of $\cK_1$ and $\cK_2$ respectively, and $\cK_3=\bC \eps$.
Define
\begin{eqnarray}
Z & = & \sum_{i=1}^{k_1} \eps_{1,i}^{}\eps_{1,i}^*\otimes (\eps_{1,i}^{}\eps_{1,i}^*+\eps\eps^*)
+ \sum_{i=1}^{k_2} \eps_{2,i}^{}\eps_{2,i}^*\otimes (\eps_{2,i}^{}\eps_{2,i}^*+\eps\eps^*)
\label{ZZZ}
\\ && {}
+ \eps\eps^*\otimes (\jed_{B(\cK_1)}+\jed_{B(\cK_2)}+k\eps\eps^*)
\nonumber
\\
&&
{} - \sum_{i=1}^{k_1}(\eps_{1,i}^{}\eps^*\otimes \eps_{1,i}^{}\eps^*+\eps\eps_{1,i}^*\otimes\eps\eps_{1,i}^*)
- \sum_{i=1}^{k_2}(\eps_{2,i}^{}\eps^*\otimes \eps\eps_{2,i}^*+\eps\eps_{2,i}^*\otimes\eps_{2,i}\eps^*)
\nonumber
\end{eqnarray}
where $k_i=\dim\cK_i$, $i=1,2$, and $k=\max\{ k_1,k_2\}$.
It is a positive operator on $\cK\otimes\cK$, because it can be decomposed in the form
\be
Z&=&
\sum_{i=1}^{k_1}(\eps_{1,i}\otimes\eps_{1,i}-\eps\otimes\eps)(\eps_{1,i}\otimes\eps_{1,i}-\eps\otimes\eps)^*
+\sum_{i=1}^{k_2}(\eps_{2,i}\otimes\eps-\eps\otimes\eps_{2,i})(\eps_{2,i}\otimes\eps-\eps\otimes\eps_{2,i})^*
\\
&&{}+
\jed_{B(\cK_1)}\otimes \eps\eps^*+\sum_{i=1}^{k_2}\eps_{2,i}^{}\eps_{2,i}^*\otimes \eps_{2,i}^{}\eps_{2,i}^* + (k-k_1)\eps\eps^*\otimes\eps\eps^*.
\ee
Similarly, one can check that
\be
Z^\Gamma &=&
\sum_{i=1}^{k_1}(\eps_{1,i}\otimes\eps-\eps\otimes\eps_{1,i})(\eps_{1,i}\otimes\eps-\eps\otimes\eps_{1,i})^*
+\sum_{i=1}^{k_2}(\eps_{2,i}\otimes\eps_{2,i}-\eps\otimes\eps)(\eps_{2,i}\otimes\eps_{2,i}-\eps\otimes\eps)^*
\\
&&{}+
\jed_{B(\cK_2)}\otimes \eps\eps^*+\sum_{i=1}^{k_1}\eps_{1,i}^{}\eps_{1,i}^*\otimes \eps_{1,i}^{}\eps_{1,i}^* + (k-k_2)\eps\eps^*\otimes\eps\eps^*,
\ee
so $Z$ is a PPT operator. Finally,
\be
\Tr(\cC_{\Omega_{\cK_1,\cK_2}}^t Z) &=&
\sum_{i=1}^{k_1}(\la\eps_{1,i},\Omega_{\cK_1,\cK_2}(\eps_{1,i}^{}\eps_{1,i}^*)\eps_{1,i}\ra+ \la\eps,\Omega_{\cK_1,\cK_2}(\eps_{1,i}^{}\eps_{1,i}^*)\eps\ra)
\\ && {}
+
\sum_{i=1}^{k_2}(\la\eps_{2,i},\Omega_{\cK_1,\cK_2}(\eps_{2,i}^{}\eps_{2,i}^*)\eps_{2,i}\ra+ \la\eps,\Omega_{\cK_1,\cK_2}(\eps_{2,i}^{}\eps_{2,i}^*)\eps\ra)
\\ && {}
-
\sum_{i=1}^{k_1}(\la\eps_{1,i},\Omega_{\cK_1,\cK_2}(\eps_{1,i}^{}\eps^*)\eps\ra +
\la\eps,\Omega_{\cK_1,\cK_2}(\eps^{}\eps_{1,i}^*)\eps_{1,i}^{}\ra)
\\&& {}
-
\sum_{i=1}^{k_2}(\la\eps,\Omega_{\cK_1,\cK_2}(\eps_{2,i}^{}\eps^*)\eps_{2,i}\ra +
\la\eps_{2,i}^{},\Omega_{\cK_1,\cK_2}(\eps^{}\eps_{2,i}^*)\eps^{}\ra)
\\&& {}
+k\la\eps,\Omega_{\cK_1,\cK_2}(\eps\eps^*)\eps\ra
\\ &=& k_1+k_2-2k_1-2k_2+k=-\min\{k_1,k_2\} <0.
\ee
Therefore, $\Omega_{\cK_1,\cK_2}$ is an entanglement witness for the PPT operator $Z$, hence it is nondecomposable.

Optimality of $\Omega_{\cK_1,\cK_2}$ follows directly from Proposition \ref{RSCgenstrongop}. 
\end{proof}
\begin{remark}
As regards the map $\Lambda_{\cK_1,\cK_2}$, it is no longer optimal. It follows from \eqref{Lambdadec} and the fact that the map $\psi(X)=\Tr(X)\jed-X^\rT$ is completely positive. However, by considering the PPT operator $Z$ given by \eqref{ZZZ}, one can show that $\Lambda_{\cK_1,\cK_2}$ is still nondecomposable.
\end{remark}

\section{A family of positive maps from $\bM_3(\bC)$ into $\bM_3(\bC)$}
In this section we will discuss the case when all spaces $\cK_i$ and $\cH_i$ are one-dimensional. In this case $\cK=\bigoplus_{i=1}^3\cK_i=\bC^3$ and $\cH=\bigoplus_{i=1}^3\cH_i=\bC^3$, thus $B(\cK)=B(\cH)=\bM_3(\bC)$. The general form of \eqref{merging} in this case is
\beq
\label{merging3}
\phi\left(\ba{ccc}
x_{11} & x_{12} & x_{13} \\ x_{21} & x_{22} & x_{23} \\ x_{31} & x_{32} & x_{33}
\ea\right)
=\left(\ba{ccc}
f_1x_{11}+w_2x_{22} & 0 & b_1x_{13}+c_1x_{31} \\
0 & f_2x_{22}+w_1x_{11} & b_2x_{23}+c_2x_{32} \\
\ov{b_1}x_{31}+\ov{c_1}x_{13} & \ov{b_2}x_{32}+\ov{c_2}x_{23} & x_{33}
\ea\right).
\eeq
\subsection{Positivity}
It follows from Proposition \ref{mergprop} that positivity of $\phi$ implies that $f_i$ and $w_i$ are nonnegative constants for $i=1,2$. Moreover, inequality \eqref{epsineq} yields here $|b_i|+|c_i|\leq f_i^{1/2}$, $i=1,2$. Thus, let us introduce the following parameters
$$
\mu_i^{}=f_i^{1/2}, \quad \sigma_i^{}=w_i^{1/2},\quad \varepsilon_i=|b_i|+|c_i|,\quad
\delta_i=(\mu_i^2-\varepsilon_i^2)^{1/2}.
$$
Then, necessary condition for positivity of a map \eqref{merging3} is that it should be of the form
\beq
\label{merging3pos}
\phi\left(\ba{ccc}
x_{11} & x_{12} & x_{13} \\ x_{21} & x_{22} & x_{23} \\ x_{31} & x_{32} & x_{33}
\ea\right)
=\left(\ba{ccc}
(\varepsilon_1^2+\delta_1^2)x_{11}+\sigma_2^2x_{22} & 0 & b_1x_{13}+c_1x_{31} \\
0 & (\varepsilon_2^2+\delta_2^2)x_{22}+\sigma_1^2x_{11} & b_2x_{23}+c_2x_{32} \\
\ov{b_1}x_{31}+\ov{c_1}x_{13} & \ov{b_2}x_{32}+\ov{c_2}x_{23} & x_{33}
\ea\right).
\eeq
As an immediate consquence of Theorem \ref{t:pos} we get the following criterion for positivity.
\begin{prop}
\label{3pos}
The map $\phi:\bM_3(\bC)\to\bM_3(\bC)$ given by \eqref{merging3pos} is positive if and only if 
\beq
\label{ineqpos3}
\sigma_1\sigma_2+\delta_1\delta_2\geq \varepsilon_1\varepsilon_2.
\eeq
\end{prop}
For future considerations let us define the following further characteristics of a map $\phi$ given by \eqref{merging3pos}. Let $\vec{b}=(|b_1|,|b_2|)^\rT$ and $\vec{c}=(|c_1|,|c_2|)^\rT$. They are vectors lying in the first quadrant of $\bR^2$. Moreover, let $s_i=\max\left\{|b_i|,|c_i|\right\}$ for $i=1,2$,
and
$s=\max\{\Vert\vec{b}\Vert,\Vert\vec{c}\Vert\}$. Finally, let $\delta=(\delta_1^2+\delta_2^2)^{1/2}$ and $\veps=(\veps_1^2+\veps_2^2)^{1/2}$.
\subsection{Complete (co)positivity}
Let us observe that the Choi matrix $\cC_\phi\in\bM_3(\bC)\otimes\bM_3(\bC)=\bM_9(\bC)$ of the map $\phi$ given by \eqref{merging3pos} is of the form
\beq
\label{Choiphi}
\cC_\phi=
\left(\ba{ccc|ccc|ccc}
\veps_1^2+\delta_1^2 & \cdot & \cdot & \cdot & \cdot & \cdot & \cdot & \cdot & b_1 \\
\cdot & \sigma_1^2 & \cdot & \cdot & \cdot & \cdot & \cdot & \cdot & \cdot \\
\cdot & \cdot & \cdot & \cdot & \cdot & \cdot & \ov{c_1} & \cdot & \cdot \\ \hline
\cdot & \cdot & \cdot & \sigma_2^2 & \cdot & \cdot & \cdot & \cdot & \cdot \\
\cdot & \cdot & \cdot & \cdot & \veps_2^2+\delta_2^2 & \cdot & \cdot & \cdot & b_2 \\
\cdot & \cdot & \cdot & \cdot & \cdot & \cdot & \cdot & \ov{c_2} & \cdot \\ \hline
\cdot & \cdot & c_1 & \cdot & \cdot & \cdot & \cdot & \cdot & \cdot \\
\cdot & \cdot & \cdot & \cdot & \cdot & c_2 & \cdot & \cdot & \cdot \\
\ov{b_1} & \cdot & \cdot & \cdot & \ov{b_2} & \cdot & \cdot & \cdot & 1
\ea\right),
\eeq
where zeros are replaced by dots. Recall that complete positivity of $\phi$ is equivalent to positive definiteness of $\cC_\phi$. Analogously complete copositivity is equivalent to positive definiteness of $\cC_\phi^\Gamma$. Thus we have
\begin{thm}
\label{t:3cp}
Let $\phi$ be a map given by \eqref{merging3pos}. The following conditions are equivalent:
\begin{enumerate}
\item[(i)]
$\phi$ is completely positive (respectively completely copositive);
\item[(ii)]
$\phi$ is $2$-positive (respectively $2$-copositive);
\item[(iii)]
$c_1=c_2=0$ (respectively $b_1=b_2=0$) and
\beq
\label{cp2}
\delta_1\delta_2\geq \veps_1\veps_2.
\eeq
\end{enumerate}
\end{thm}
\begin{proof}
(ii) $\Rightarrow$ (iii)
Proposition \ref{2pos} asserts that $c_1=c_2=0$ is a necessary condition for $2$-positivity. Then, 
\eqref{merging3pos} reduces to 
\beq
\label{3cp}
\phi\left(\ba{ccc}
x_{11} & x_{12} & x_{13} \\ x_{21} & x_{22} & x_{23} \\ x_{31} & x_{32} & x_{33}
\ea\right)
=\left(\ba{ccc}
(|b_1|^2+\delta_1^2)x_{11}+\sigma_2^2x_{22} & 0 & b_1x_{13} \\
0 & (|b_2|^2+\delta_2^2)x_{22}+\sigma_1^2x_{11} & b_2x_{23} \\
\ov{b_1}x_{31} & \ov{b_2}x_{32} & x_{33}
\ea\right).
\eeq
Let $H$ be a positive element in $\bM_2(\bC)\otimes\bM_3(\bC)$ given by
$$
H=\left(\ba{ccc|ccc}
1 & \cdot & \cdot & \cdot & 1 & 1 \\
\cdot & \cdot & \cdot & \cdot & \cdot & \cdot \\
\cdot & \cdot & \cdot & \cdot & \cdot & \cdot \\ \hline
\cdot & \cdot & \cdot & \cdot & \cdot & \cdot \\
1 & \cdot & \cdot & \cdot & 1 & 1 \\
1 & \cdot & \cdot & \cdot & 1 & 1 \\
\ea\right).
$$
Then
$$
\id_{\bM_2(\bC)}\otimes\phi (H)=
\left(\ba{ccc|ccc}
|b_1|^2+\delta_1^2 & \cdot & \cdot & \cdot & \cdot & b_1 \\
\cdot & \sigma_1^2 & \cdot & \cdot & \cdot & \cdot \\
\cdot & \cdot & \cdot & \cdot & \cdot & \cdot \\ \hline
\cdot & \cdot & \cdot & \sigma_2^2 & \cdot & \cdot \\
\cdot & \cdot & \cdot & \cdot & |b_2|^2+\delta_2^2 & b_2 \\
\ov{b_1} & \cdot & \cdot & \cdot & \ov{b_2} & 1 \\
\ea\right).
$$
Since $\phi$ is 2-positive, the above matrix is positive definite. Therefore, 
\beq
\label{wyzn}
\left|
\ba{ccc}
|b_1^{}|^2+\delta_1^2 & 0 & b_1 \\
0 & |b_2^{}|^2+\delta_2^2 & b_2 \\
\ov{b_1} & \ov{b_2} & 1
\ea
\right|\geq 0
\eeq
and  \eqref{cp2} follows.

(iii) $\Rightarrow$ (i)
The Choi matrix of $\phi$ is of the form 
\beq
\label{ciszero}
\cC_\phi=
\left(\ba{ccc|ccc|ccc}
|b_1|^2+\delta_1^2 & \cdot & \cdot & \cdot & \cdot & \cdot & \cdot & \cdot & b_1 \\
\cdot & \sigma_1^2 & \cdot & \cdot & \cdot & \cdot & \cdot & \cdot & \cdot \\
\cdot & \cdot & \cdot & \cdot & \cdot & \cdot & \cdot & \cdot & \cdot \\ \hline
\cdot & \cdot & \cdot & \sigma_2^2 & \cdot & \cdot & \cdot & \cdot & \cdot \\
\cdot & \cdot & \cdot & \cdot & |b_2|^2+\delta_2^2 & \cdot & \cdot & \cdot & b_2 \\
\cdot & \cdot & \cdot & \cdot & \cdot & \cdot & \cdot & \cdot & \cdot \\ \hline
\cdot & \cdot & \cdot & \cdot & \cdot & \cdot & \cdot & \cdot & \cdot \\
\cdot & \cdot & \cdot & \cdot & \cdot & \cdot & \cdot & \cdot & \cdot \\
\ov{b_1} & \cdot & \cdot & \cdot & \ov{b_2} & \cdot & \cdot & \cdot & 1
\ea\right)
\eeq
Since \eqref{cp2} is satisfied, the inequality \eqref{wyzn} holds, and consequently $\cC_\phi$ is positive definite.
Hence, complete positivity of $\phi$ follows.
\end{proof}

\subsection{Decomposablity vs. nondecomposability}
For a map $\phi$ given by \eqref{merging3pos} let
$\vec{b}=(|b_1|,|b_2|)^\rT$ and $\vec{c}=(|c_1|,|c_2|)^\rT$. They are vectors lying in the first quadrant of $\bR^2$. Moreover, let $s_i=\max\left\{|b_i|,|c_i|\right\}$ for $i=1,2$,
and
$s=\max\{\Vert\vec{b}\Vert,\Vert\vec{c}\Vert\}$. Finally, let $\delta=(\delta_1^2+\delta_2^2)^{1/2}$ and $\veps=(\veps_1^2+\veps_2^2)^{1/2}$.
In this subsection we always assume that $\phi$ is positive, hence it satisfies \eqref{ineqpos3}.

\begin{prop}
\label{dec_dep}
If vectors $\vec{b}$, $\vec{c}$ are linearly dependent, then $\phi$ is decomposable.
\end{prop}
\begin{proof}
Firstly, assume that $\vec{b}=0$ or $\vec{c}=0$.
Taking into account Proposition \ref{cp2} it is enough to consider the case $\delta_1\delta_2<\veps_1\veps_2$. Assume that $\vec{c}=0$. Then necessarily $\veps_i=|b_i|>0$ for $i=1,2$ and $\cC_\phi$ is of the form \eqref{ciszero}. Let $k=\dfrac{\delta_1\delta_2}{|b_1||b_2|}$. Consider decomposition $\cC_\phi=\cC_1^{}+\cC_2^\Gamma$, where
\be
\cC_1&=&
\left(\ba{ccc|ccc|ccc}
|b_1|^2+\delta_1^2 & \cdot & \cdot & \cdot & (1-k)b_1\ov{b_2} & \cdot & \cdot & \cdot & b_1 \\
\cdot & \cdot & \cdot & \cdot & \cdot & \cdot & \cdot & \cdot & \cdot \\
\cdot & \cdot & \cdot & \cdot & \cdot & \cdot & \cdot & \cdot & \cdot \\ \hline
\cdot & \cdot & \cdot & \cdot & \cdot & \cdot & \cdot & \cdot & \cdot \\
(1-k)\ov{b_1}b_2 & \cdot & \cdot & \cdot & |b_2|^2+\delta_2^2 & \cdot & \cdot & \cdot & b_2 \\
\cdot & \cdot & \cdot & \cdot & \cdot & \cdot & \cdot & \cdot & \cdot \\ \hline
\cdot & \cdot & \cdot & \cdot & \cdot & \cdot & \cdot & \cdot & \cdot \\
\cdot & \cdot & \cdot & \cdot & \cdot & \cdot & \cdot & \cdot & \cdot \\
\ov{b_1} & \cdot & \cdot & \cdot & \ov{b_2} & \cdot & \cdot & \cdot & 1
\ea\right) \\[2mm]
\cC_2&=&
\left(\ba{ccc|ccc|ccc}
\cdot & \cdot & \cdot & \cdot & \cdot & \cdot & \cdot & \cdot & \cdot \\
\cdot & \sigma_1^2 & \cdot & -(1-k)b_1\ov{b_2} & \cdot & \cdot & \cdot & \cdot & \cdot \\
\cdot & \cdot & \cdot & \cdot & \cdot & \cdot & \cdot & \cdot & \cdot \\ \hline
\cdot & -(1-k)\ov{b_1}b_2 & \cdot & \sigma_2^2 & \cdot & \cdot & \cdot & \cdot & \cdot \\
\cdot & \cdot & \cdot & \cdot & \cdot & \cdot & \cdot & \cdot & \cdot \\
\cdot & \cdot & \cdot & \cdot & \cdot & \cdot & \cdot & \cdot & \cdot \\ \hline
\cdot & \cdot & \cdot & \cdot & \cdot & \cdot & \cdot & \cdot & \cdot \\
\cdot & \cdot & \cdot & \cdot & \cdot & \cdot & \cdot & \cdot & \cdot \\
\cdot & \cdot & \cdot & \cdot & \cdot & \cdot & \cdot & \cdot & \cdot
\ea\right) .
\ee
By direct calculation of minors one can check that $\cC_1$ is positive definite. According to \eqref{ineqpos3},
$$\sigma_1^2\sigma_2^2\geq (|b_1||b_2|-\delta_1\delta_2)^2=(1-k)^2|b_1|^2|b_2|^2.$$
Hence, the matrix $\cC_2$ is also positive definite. Consequently, $\cC_1$ is the Choi matrix of a completely positive map while $\cC_2^\Gamma$ is the Choi matrix of a completely copositive one. Therefore $\phi$ is decomposable.
The case $\vec{b}=0$ can be considered analogously.

Now, assume that $\vec{b}\neq 0$ and $\vec{c}\neq 0$.
Define $\lambda=\dfrac{\delta_1\delta_2}{\veps_1\veps_2}$. Since $\vec{b}$, $\vec{c}$ are dependent, $\vec{c}=\kappa\vec{b}$ for some $\kappa>0$. 
Consider $\lambda\leq 1$. Let
$$
\cC_{1,1}=
\left(\ba{ccc|ccc|ccc}
(1+\kappa)|b_1|^2+\dfrac{1}{1+\kappa}\delta_1^2 & \cdot & \cdot & \cdot & (1-\lambda)(1+\kappa)b_1\ov{b_2} & \cdot & \cdot & \cdot & b_1 \\
\cdot & \cdot 
& \cdot & \cdot 
& \cdot & \cdot & \cdot & \cdot & \cdot \\
\cdot & \cdot & \cdot & \cdot & \cdot & \cdot & \cdot & \cdot & \cdot \\ \hline
\cdot & \cdot  
& \cdot & \cdot 
& \cdot & \cdot & \cdot & \cdot & \cdot \\
(1-\lambda)(1+\kappa)\ov{b_1}b_2 & \cdot & \cdot & \cdot & (1+\kappa)|b_2|^2+\dfrac{1}{1+\kappa}\delta_2^2 & \cdot & \cdot & \cdot & b_2 \\
\cdot & \cdot & \cdot & \cdot & \cdot & \cdot & \cdot & \cdot & \cdot \\ \hline
\cdot & \cdot & \cdot & \cdot & \cdot & \cdot & \cdot & \cdot & \cdot \\
\cdot & \cdot & \cdot & \cdot & \cdot & \cdot & \cdot & \cdot & \cdot \\
\ov{b_1} & \cdot & \cdot & \cdot & \ov{b_2} & \cdot & \cdot & \cdot & \dfrac{1}{1+\kappa}
\ea\right)
$$
$$
\cC_{1,2}=
\left(\ba{ccc|ccc|ccc}
\cdot & \cdot & \cdot & \cdot & \cdot & \cdot & \cdot & \cdot & \cdot \\
\cdot & \dfrac{\kappa}{1+\kappa}\,\sigma_1^2
& \cdot & -\dfrac{(1-\lambda)(1+\kappa)}{\kappa}\,\ov{c_1}c_2
& \cdot & \cdot & \cdot & \cdot & \cdot \\
\cdot & \cdot & \cdot & \cdot & \cdot & \cdot & \cdot & \cdot & \cdot \\ \hline
\cdot & -\dfrac{(1-\lambda)(1+\kappa)}{\kappa}\,c_1\ov{c_2}
& \cdot & \dfrac{\kappa}{1+\kappa}\,\sigma_2^2
& \cdot & \cdot & \cdot & \cdot & \cdot \\
\cdot & \cdot & \cdot & \cdot & \cdot & \cdot & \cdot & \cdot & \cdot \\
\cdot & \cdot & \cdot & \cdot & \cdot & \cdot & \cdot & \cdot & \cdot \\ \hline
\cdot & \cdot & \cdot & \cdot & \cdot & \cdot & \cdot & \cdot & \cdot \\
\cdot & \cdot & \cdot & \cdot & \cdot & \cdot & \cdot & \cdot & \cdot \\
\cdot & \cdot & \cdot & \cdot & \cdot & \cdot & \cdot & \cdot & \cdot
\ea\right)
$$
$$
\cC_{2,1}=
\left(\ba{ccc|ccc|ccc}
\dfrac{1+\kappa}{\kappa}\,|c_1|^2 + \dfrac{\kappa}{1+\kappa}\,\delta_1^2& \cdot & \cdot & \cdot & \dfrac{(1-\lambda)(1+\kappa)}{\kappa}\,\ov{c_1}c_2 & \cdot & \cdot & \cdot & \ov{c_1} \\
\cdot & \cdot 
& \cdot & \cdot 
& \cdot & \cdot & \cdot & \cdot & \cdot \\
\cdot & \cdot & \cdot & \cdot & \cdot & \cdot & \cdot & \cdot & \cdot \\ \hline
\cdot & \cdot 
& \cdot & \cdot 
& \cdot & \cdot & \cdot & \cdot & \cdot \\
\dfrac{(1-\lambda)(1+\kappa)}{\kappa}\,c_1\ov{c_2} & \cdot & \cdot & \cdot & \dfrac{1+\kappa}{\kappa}\,|c_2|^2+\dfrac{\kappa}{1+\kappa}\,\delta_2^2 & \cdot & \cdot & \cdot & \ov{c_2} \\
\cdot & \cdot & \cdot & \cdot & \cdot & \cdot & \cdot & \cdot & \cdot \\ \hline
\cdot & \cdot & \cdot & \cdot & \cdot & \cdot & \cdot & \cdot & \cdot \\
\cdot & \cdot & \cdot & \cdot & \cdot & \cdot & \cdot & \cdot & \cdot \\
c_1 & \cdot & \cdot & \cdot & c_2 & \cdot & \cdot & \cdot & \dfrac{\kappa}{1+\kappa}
\ea\right) .
$$
$$
\cC_{2,2}=
\left(\ba{ccc|ccc|ccc}
\cdot & \cdot & \cdot & \cdot & \cdot & \cdot & \cdot & \cdot & \cdot \\
\cdot & \dfrac{1}{1+\kappa}\,\sigma_1^2 & \cdot & -(1-\lambda)(1+\kappa)b_1\ov{b_2} & \cdot & \cdot & \cdot & \cdot & \cdot \\
\cdot & \cdot & \cdot & \cdot & \cdot & \cdot & \cdot & \cdot & \cdot \\ \hline
\cdot & -(1-\lambda)(1+\kappa)\ov{b_1}b_2 & \cdot & \dfrac{1}{1+\kappa}\,\sigma_2^2 & \cdot & \cdot & \cdot & \cdot & \cdot \\
\cdot & \cdot & \cdot & \cdot & \cdot & \cdot & \cdot & \cdot & \cdot \\
\cdot & \cdot & \cdot & \cdot & \cdot & \cdot & \cdot & \cdot & \cdot \\ \hline
\cdot & \cdot & \cdot & \cdot & \cdot & \cdot & \cdot & \cdot & \cdot \\
\cdot & \cdot & \cdot & \cdot & \cdot & \cdot & \cdot & \cdot & \cdot \\
\cdot & \cdot & \cdot & \cdot & \cdot & \cdot & \cdot & \cdot & \cdot
\ea\right) .
$$
We will show that all these matrices are positive definite. For instance, for $\cC_{1,1}$ we calculate nontrivial principal minors:
\be
\lefteqn{\left|\ba{cc}
(1+\kappa)|b_1|^2+\dfrac{1}{1+\kappa}\delta_1^2 &
(1-\lambda)(1+\kappa)b_1\ov{b_2} \\
(1-\lambda)(1+\kappa)\ov{b_1}b_2 & (1+\kappa)|b_2|^2+\dfrac{1}{1+\kappa}\delta_2^2
\ea\right|=}\\
&=&\left(1-(1-\lambda)^2\right)(1+\kappa)^2|b_1|^2|b_2|^2+\delta_1^2|b_2|^2+\delta_2^2|b_1|^2+\dfrac{\delta_1^2\delta_2^2}{(1+\kappa)^2}\geq 0
\ee
$$
\left|\ba{ccc}
(1+\kappa)|b_1|^2+\dfrac{1}{1+\kappa}\delta_1^2 & (1-\lambda)(1+\kappa)b_1\ov{b_2} & b_1 \\
(1-\lambda)(1+\kappa)\ov{b_1}b_2 & (1+\kappa)|b_2|^2+\dfrac{1}{1+\kappa}\delta_2^2 & b_2 \\
\ov{b_1} & \ov{b_2} & \dfrac{1}{1+\kappa}
\ea\right|=
\dfrac{\delta_1^2\delta_2^2}{(1+\kappa)^3}-\lambda^2(1+\kappa)|b_1|^2|b_2|^2= 0
$$
The last equality follows from the definitions of $\lambda$ and $\kappa$. They imply
$\delta_1^2\delta_2^2=\lambda^2\veps_1^2\veps_2^2=\lambda^2(1+\kappa)^4|b_1|^2|b_2|^2.$
In order to prove positive definiteness of $\cC_{1,2}$ let us observe that
\eqref{ineqpos3} and definition of $\kappa$  imply
$$\sigma_1^2\sigma_2^2\geq(1-\lambda)^2\veps_1^2\veps_2^2= (1-\lambda)^2\dfrac{(1+\kappa)^4}{\kappa^4}|c_1|^2|c_2|^2.$$
Positive definiteness of $\cC_{2,1}$ and $\cC_{2,2}$ can be proved analogously. Now, observe that
$\cC_\phi=\cC_{1,1}^{}+\cC_{1,2}^{}+\cC_{2,1}^\Gamma+\cC_{2,2}^\Gamma$.
Thus, $\phi$ is decomposable.

For $\lambda\geq 1$, we consider decomposition $\cC_\phi=\cC_1+\cC_2^\Gamma$, where
$$
\cC_1=
\left(\ba{ccc|ccc|ccc}
(1+\kappa)|b_1|^2+\dfrac{1}{1+\kappa}\delta_1^2 & \cdot & \cdot & \cdot & \cdot & \cdot & \cdot & \cdot & b_1 \\
\cdot & \sigma_1^2
& \cdot & \cdot 
& \cdot & \cdot & \cdot & \cdot & \cdot \\
\cdot & \cdot & \cdot & \cdot & \cdot & \cdot & \cdot & \cdot & \cdot \\ \hline
\cdot & \cdot  
& \cdot & \sigma_2^2
& \cdot & \cdot & \cdot & \cdot & \cdot \\
\cdot & \cdot & \cdot & \cdot & (1+\kappa)|b_2|^2+\dfrac{1}{1+\kappa}\delta_2^2 & \cdot & \cdot & \cdot & b_2 \\
\cdot & \cdot & \cdot & \cdot & \cdot & \cdot & \cdot & \cdot & \cdot \\ \hline
\cdot & \cdot & \cdot & \cdot & \cdot & \cdot & \cdot & \cdot & \cdot \\
\cdot & \cdot & \cdot & \cdot & \cdot & \cdot & \cdot & \cdot & \cdot \\
\ov{b_1} & \cdot & \cdot & \cdot & \ov{b_2} & \cdot & \cdot & \cdot & \dfrac{1}{1+\kappa}
\ea\right)
$$
$$
\cC_{2}=
\left(\ba{ccc|ccc|ccc}
\dfrac{1+\kappa}{\kappa}\,|c_1|^2 + \dfrac{\kappa}{1+\kappa}\,\delta_1^2& \cdot & \cdot & \cdot & \cdot & \cdot & \cdot & \cdot & \ov{c_1} \\
\cdot & \cdot 
& \cdot & \cdot 
& \cdot & \cdot & \cdot & \cdot & \cdot \\
\cdot & \cdot & \cdot & \cdot & \cdot & \cdot & \cdot & \cdot & \cdot \\ \hline
\cdot & \cdot 
& \cdot & \cdot 
& \cdot & \cdot & \cdot & \cdot & \cdot \\
\cdot & \cdot & \cdot & \cdot & \dfrac{1+\kappa}{\kappa}\,|c_2|^2+\dfrac{\kappa}{1+\kappa}\,\delta_2^2 & \cdot & \cdot & \cdot & \ov{c_2} \\
\cdot & \cdot & \cdot & \cdot & \cdot & \cdot & \cdot & \cdot & \cdot \\ \hline
\cdot & \cdot & \cdot & \cdot & \cdot & \cdot & \cdot & \cdot & \cdot \\
\cdot & \cdot & \cdot & \cdot & \cdot & \cdot & \cdot & \cdot & \cdot \\
c_1 & \cdot & \cdot & \cdot & c_2 & \cdot & \cdot & \cdot & \dfrac{\kappa}{1+\kappa}
\ea\right) .
$$
The inequality $\delta_1\delta_2\geq\veps_1\veps_2$ imply that both matrices are positive definite. Hence, $\phi$ again is positive.
\end{proof}
In next proposition we give some sufficient condition for nondecomposability.
\begin{prop}
Assume that a positive map \eqref{merging3pos} satisfies $\vec{b}\neq 0$ and $\vec{c}\neq 0$. If
\beq
\label{nondec}
s(\veps^2+\delta^2)^{1/2}<\Vert\vec{b}\Vert^2 + \Vert\vec{c}\Vert^2,
\eeq
then $\phi$ is nondecomposable.
\end{prop}
\begin{proof}
Consider $Z\in\bM_3(\bC)\otimes\bM_3(\bC)$ where
$$
Z=
\left(\ba{ccc|ccc|ccc}
\gamma & \cdot & \cdot & \cdot & \cdot & \cdot & \cdot & \cdot & -\ov{b_1} \\
\cdot & \cdot & \cdot & \cdot & \cdot & \cdot & \cdot & \cdot & \cdot \\
\cdot & \cdot & 1 & \cdot & \cdot & \cdot & -c_1 & \cdot & \cdot \\ \hline
\cdot & \cdot & \cdot & \cdot & \cdot & \cdot & \cdot & \cdot & \cdot \\
\cdot & \cdot & \cdot & \cdot & \gamma & \cdot & \cdot & \cdot & -\ov{b_2} \\
\cdot & \cdot & \cdot & \cdot & \cdot & 1 & \cdot & -c_2 & \cdot \\ \hline
\cdot & \cdot & -\ov{c_1} & \cdot & \cdot & \cdot & s_1^2 & \cdot & \cdot \\
\cdot & \cdot & \cdot & \cdot & \cdot & -\ov{c_2} & \cdot & s_2^2 & \cdot \\
-b_1 & \cdot & \cdot & \cdot & -b_2 & \cdot & \cdot & \cdot & \gamma^{-1}s^2
\ea\right)
$$
for $\gamma>0$.
One can easily check that both $Z$ and $Z^\Gamma$ are positive, hence $Z$ is a PPT operator.
One can check that
$$
\Tr(\cC_\phi^\rT Z)= \gamma(\veps^2+\delta^2)+\gamma^{-1}s^2- 2\Vert\vec{b}\Vert^2- 2\Vert\vec{c}\Vert^2.
$$
The above expression attains minimal value for $\gamma =s(\veps^2+\delta^2)^{-1/2}$. For such a value we have
$$
\Tr(\cC_\phi^\rT Z)=2\left[s(\veps^2+\delta^2)^{1/2}-\Vert\vec{b}\Vert^2-\Vert\vec{c}\Vert^2\right].
$$
If \eqref{nondec} is satisfied, then $\Tr(\cC_\phi^\rT Z)<0$. Therefore $\phi$ is an entanglement witness for $Z$, hence $\phi$ is nondecomposable
\end{proof}

\begin{cor}
Assume that $\delta=0$, and $\Vert\vec{b}\Vert\geq\Vert\vec{c}\Vert>0$. If
\beq
\label{opiu}
\la\vec{b},\vec{c}\ra<\frac{1}{2}\left(1+\frac{\Vert\vec{c}\Vert^2}{\Vert\vec{b}\Vert^2}\right)\Vert\vec{c}\Vert^2,
\eeq
then $\phi$ is nondecomposable.

In particular, for $\Vert\vec{b}\Vert=\Vert\vec{c}\Vert$, the map $\phi$ is nondecomposable if and only if the vectors $\vec{b}$, $\vec{c}$ are linearly independent.
\end{cor}
\begin{proof}
By simple calculations one can easily check, that if $\delta=0$, then inequalities \eqref{nondec} and \eqref{opiu} are equivalent. For the special case $\Vert\vec{b}\Vert=\Vert\vec{c}\Vert$, the inequality \eqref{opiu} is equivalent to inequality $\la\vec{b},\vec{c}\ra<\Vert\vec{b}\Vert\Vert\vec{c}\Vert$.
\end{proof}
\subsection{Extremality}
We say that two maps $\phi,\phi':B(K)\to B(H)$ are equivalent, if there are invertible operators $Q\in B(K)$ and $R\in B(H)$ such that $\phi'(X)=R\phi(QXQ^*)R^*$. 
One can check that any map $\phi:\bM_3(\bC)\to \bM_3(\bC)$ of the form \eqref{merging3pos} such that both $\veps_1$ and $\veps_2$ are strictly positive, is equivalent to a map $\phi'$ of the form \eqref{merging3pos} such that $\veps_1=\veps_2=1$ and $b_i,c_i\geq 0$ for $i=1,2$. Indeed, given a map $\phi$ of the form \eqref{merging3pos} one should take $Q=\mathrm{diag}(q_1,q_2,1)$ and $R=\mathrm{diag}(r_1,r_2,1)$ where $q_1,q_2,r_1,r_2\in\bC$ satisfy conditions
$|q_i||r_i|=(|b_i|+|c_i|)^{-1}$ and the numbers $q_ir_ib_i$, $q_i\ov{r_i}c_i$ are both nonnegative, $i=1,2$. 

\begin{thm}
Let $\phi$ be a map of the form \eqref{merging3pos}. Then the following are equivalent:
\begin{enumerate}
\item
$\phi$ is exposed,
\item
$\phi$ is extremal,
\item
each of the following conditions is satisfied
\begin{enumerate}
\item $\vec{b}\neq 0$ and $\vec{c}\neq 0$, 
\item $\delta_1=\delta_2=0$, 
\item $\sigma_1\sigma_2=\vep_1\vep_2$, 
\item $\langle\vec{b},\vec{c}\rangle=0$.
\end{enumerate}
\end{enumerate}
\end{thm}
\begin{proof}
(2) $\Rightarrow$ (3) 
As we explained in the introduction to this subsection, one may restrict to the case $\vep_1=\vep_2=1$ and $b_i\geq 0$, $c_i\geq 0$, $i=1,2$. Further, using the same arguments as in the proof of Theorem \ref{mainm} (cf. \eqref{EF}), one can also assume that $\sigma_1=\sigma_2=:\sigma$. First of all observe that if $\delta_1\delta_2+\sigma_1\sigma_2>1$, then it is rather obvious that $\phi$ is not extremal. The same observation is valid for the case $\sigma_1\sigma_2=1$ and $\delta_1\delta_2=0$ but one of $\delta_1$ and $\delta_2$ is strictly positive. So, assume that $\delta_1\delta_2+\sigma_1\sigma_2=1$. If  follows from the proof of Proposition \ref{dec_dep} that $\phi$ is a nontrivial sum of two positive maps if $\vec{b}$ and $\vec{c}$ are linearly dependent. Hence, assume further, that $\vec{b}$ and $\vec{c}$ are nonzero vectors. Since $b_i+c_i=1$ for $i=1,2$, we have $b_1>0$, $c_2>0$ or $b_2>0$, $c_1>0$. Assume that the first pair if inequalities is satisfied. Let $0<\lambda\leq \min \{b_1,c_2\}$. We will show that $\lambda\phi_0\leq \phi$, where
$\phi_0$ is the map of the form
\beq
\label{merging3pos0}
\phi_0\left(\ba{ccc}
x_{11} & x_{12} & x_{13} \\ x_{21} & x_{22} & x_{23} \\ x_{31} & x_{32} & x_{33}
\ea\right)
=\left(\ba{ccc}
x_{11}+x_{22} & 0 & x_{13} \\
0 & x_{11}+x_{22} & x_{32} \\
x_{31} & x_{23} & x_{33}
\ea\right).
\eeq
Indeed, observe that
\be
\lefteqn{(\phi-\lambda\phi_0)\left(\ba{ccc}
x_{11} & x_{12} & x_{13} \\ x_{21} & x_{22} & x_{23} \\ x_{31} & x_{32} & x_{33}
\ea\right)=} \\
&=&\left(\ba{ccc}
(1-\lambda+\delta_1^2)x_{11}+(\sigma^2-\lambda)x_{22} & 0 & (b_1-\lambda)x_{13}+c_1x_{31} \\
0 & (1-\lambda+\delta_2^2)x_{22}+(\sigma^2-\lambda)x_{11} & b_1x_{23}+(c_2-\lambda)x_{32} \\
(b_1-\lambda)x_{31}+c_1x_{13} & b_2x_{32}+(c_2-\lambda)x_{23} & (1-\lambda)x_{33}
\ea\right).
\ee
This map is equivalent to 
\be
\lefteqn{\left(\ba{ccc}
x_{11} & x_{12} & x_{13} \\ x_{21} & x_{22} & x_{23} \\ x_{31} & x_{32} & x_{33}
\ea\right)\mapsto} \\
&&\left(\ba{ccc}
(1-\lambda+\delta_1^2)x_{11}+(\sigma^2-\lambda)x_{22} & 0 & (1-\lambda)^{-1/2}\big((b_1-\lambda)x_{13}+c_1x_{31}\big) \\
0 & (1-\lambda+\delta_2^2)x_{22}+(\sigma^2-\lambda)x_{11} & (1-\lambda)^{-1/2}\big(b_1x_{23}+(c_2-\lambda)x_{32}\big) \\
(1-\lambda)^{-1/2}\big((b_1-\lambda)x_{31}+c_1x_{13}\big) & (1-\lambda)^{-1/2}\big(b_2x_{32}+(c_2-\lambda)x_{23}\big) & x_{33}
\ea\right).
\ee
The latter is of the form \eqref{merging3pos}. Let us equip all coefficients described in \eqref{merging3pos} for this map with primes. Observe that $\vep_i'=(1-\lambda)^{-1/2}(b_i+c_i-\lambda)=(1-\lambda)^{1/2}$. Moreover, $\delta_i'=\delta_i$, and $\sigma_i'=(\sigma^2-\lambda)^{1/2}$. Consequently, $\delta_1'\delta_2' +\sigma_1'\sigma_2' =\delta_1\delta_2 +\sigma^2-\lambda =1-\lambda=\vep_1'\vep_2'$, and according to Proposition \ref{3pos} the map $\phi-\lambda\phi_0$ is positive. Finally, observe that $\phi$ is not a multiple of $\phi_0$ unless all conditions (a) -- (d) are satisfied. Thus extremality of $\phi$ implies all these conditions. 

(3) $\Rightarrow$ (1). It is a special case of Theorem \ref{mainm}.
\end{proof}

\section{Final remarks}
In \cite{MM10} one of the authors formulated a conjecture that each extremal positive map  $\phi:B(\cK)\to B(\cH)$ with the property that $\mathrm{rank}\,\phi(P)=1$ for some one-dimensional projection $P\in B(\cK)$, must be of the form \eqref{ads}. The map constructed by Miller and Olkiewicz is a counterexample for this conjecture. However, motivated by Theorem \ref{mainm} we conjecture that maps which have the form  \eqref{block} or have the form being a sort of variation of \eqref{block}, should be typical maps satisfying the property described in the begining of this section (cf. \cite{MO2}).

Let us remind that maximal faces in the cone $\bP(\cK,\cH)$ are of the form $\cF_{\xi,x}=\{\phi:\,\phi(\xi\xi^*)x=0\}$ for $\xi\in\cK$ and $x\in\cH$ \cite{EK}. Thus, if the conjecture were confirmed, we would be very close to a general form for extremal (or exposed) positive maps.

\section*{Acknoledgements}
The paper is supported by ERC AdG QOLAPS. MM thanks Seung-Hyeok Kye for useful remarks and questions concerning the case $\cK=\cH=\bC^3$. We also thank Marek Miller for careful reading the manuscript and useful remarks.

\end{document}